\documentclass[11pt]{article}
\usepackage{amssymb,amsmath,amsthm,amsfonts, commath, mathrsfs, mathtools, secdot}
\usepackage[T1]{fontenc}		
\usepackage[a4paper, margin=1in]{geometry}

\usepackage{bbm}
\usepackage{extarrows}
\usepackage{enumerate}
\usepackage{multicol}	
\usepackage{fge}
\usepackage{oldgerm}
\usepackage{hyperref} 
\usepackage{paralist} 
\usepackage{accents}

\usepackage{tikz}
\usetikzlibrary{datavisualization}
\usetikzlibrary{datavisualization.formats.functions}
\usepackage{pgfplots}
\usepackage{tkz-fct}
\pgfplotsset{every axis/.append style={
		axis x line=middle,    
		axis y line=middle,    
		axis line style={<->}, 
		xlabel={$x$},          
		ylabel={$y$},          
	},
	cmhplot/.style={color=blue,mark=none,line width=1pt,<->},
	soldot/.style={color=blue,only marks,mark=*},
	holdot/.style={color=blue,fill=white,only marks,mark=*},
}

\tikzset{>=stealth}

\usepackage[natbib=true, backend=biber, style=numeric, maxbibnames=10]{biblatex}
\usepackage{csquotes}
\sectiondot{subsection}
\usepackage{xcolor}

\usepackage{color}

\newtheoremstyle{Assump}%
{3pt}
{3pt}
{\itshape}
{}
{\bfseries}
{.}
{.5em}
{\thmname{#1} \thmnumber{#2} \thmnote{\normalfont#3}}

\newtheorem{theorem}{Theorem}[section]
\newtheorem{lem}[theorem]{Lemma}
\newtheorem{prop}[theorem]{Proposition}
\newtheorem{cor}[theorem]{Corollary}
\theoremstyle{definition}\newtheorem{example}[theorem]{Example}
\theoremstyle{definition}
\theoremstyle{definition}\newtheorem{remark}[theorem]{Remark}
\theoremstyle{definition}\newtheorem{defi}[theorem]{Definition}

\numberwithin{equation}{section}

\newcommand{\ubar}[1]{\underaccent{\bar}{#1}}

\newcommand{\e}{\operatorname{e}}
\newcommand{\N}{\mathbb{N}}
\newcommand{\R}{\mathbb{R}}

\newcommand{\Q}{\mathbb{Q}}
\newcommand{\D}{{\bf D}}											
\newcommand{\C}{{\bf C}}											
\newcommand{\B}{\mathscr{B}}
\newcommand{\F}{\mathbb{F}}											
\newcommand*\diff{\mathop{}\!\mathrm{d}}								
\newcommand{\E}{\mathbb{E}}											
\newcommand{\Pro}{\mathbb{P}}										
\newcommand{\indep}{\perp \!\!\! \perp}									
\newcommand{\ind}{\operatorname{\mathbf{1}}}							
\newcommand{\convw}[1]{\xRightarrow[ {#1}]{}}							
\renewcommand{\max}[1]{\underset{#1}{\operatorname{max}}\;}				
\renewcommand{\min}[1]{\underset{#1}{\operatorname{min}}\;}					
\renewcommand{\sup}[1]{\underset{#1}{\operatorname{sup}}\;}					
\renewcommand{\inf}[1]{\underset{#1}{\operatorname{inf}}\;}					
\renewcommand{\setminus}{\mathbin{\fgebackslash}}						
\newcommand{\union}[2]{\overset{#2}{\bigcup\limits_{#1}}\;}					
\newcommand{\intersection}[2]{\overset{#2}{\bigcap\limits_{#1}}\;}				
\newcommand{\conv}[2]{\; \xrightarrow[ {#1}]{#2} \; }						
    
\newcommand{\dJ}{d_{J_1}}								
\newcommand{\dM}{d_{M_1}}							    

\title{Weak Convergence of Stochastic Integrals on Skorokhod Space in Skorokhod's $J_1$ and $M_1$ Topologies}
\date{July 30, 2025}
\author{Andreas S\o jmark\thanks{London School of Economics, Department of Statistics, London, WC2A 2AE, UK. \texttt{a.sojmark@lse.ac.uk}} \,\; and \, Fabrice Wunderlich\thanks{University of Oxford, Mathematical Institute, Oxford, OX2 6GG, UK. \texttt{wunderlich@maths.ox.ac.uk}}  }

	\date{}

\begin{document}
	\maketitle
\vspace{-9pt}

	\begin{abstract}
	We provide criteria for Itô integration to behave continuously with respect to Skorokhod's $J_1$ and $M_1$ topologies, when the integrands and integrators converge weakly or in probability. The results are novel in the $M_1$ setting and unify existing theories in the $J_1$ case. Beyond sufficient criteria, we present an example of uniformly convergent martingale integrators for which the continuity breaks down. Moreover, we show that, for families of local martingales, $M_1$ tightness in fact implies $J_1$ tightness under a mild localised uniform integrability condition. Finally, we apply our results to study scaling limits of models of anomalous diffusion driven by continuous-time random walks. This yields new results on weak $M_1$ and $J_1$ convergence to stochastic integrals against subordinated stable processes. In the case of superdiffusive scaling, an interesting counterexample is obtained.
	\end{abstract}


\section{Introduction}

The topic of this paper is the weak convergence of stochastic integrals on Skorokhod space with respect to Skorokhod's $J_1$ and $M_1$ topologies. By Skorokhod space, we mean the space of all c\`adl\`ag paths $x:[0,\infty)\rightarrow  \mathbb{R}^d$, for a given dimension $d\geq 1$, which we denote by $\D_{\R^d}[0,\infty)$. We let $\dJ$ and $\dM$ denote a choice of metrics inducing, respectively, the $J_1$ and $M_1$ topology on $\D_{\R^d}[0,\infty)$. For their precise definitions and some key properties, see Appendix \ref{app:J1_M1_tops}. Here we only stress that $\D_{\R^d}[0,\infty)$ is a Polish space when equipped with either of the two topologies.

Consider a sequence of $d$-dimensional semimartingales $(X^n)_{n\geq 1}$ on some filtered probability spaces $(\Omega^n, \mathcal{F}^n, \F^n, \Pro^n)$ and suppose $X^n \Rightarrow X$ on $(\D_{\R^d}[0,\infty),\rho )$ for  $\rho=\dJ$ or $\rho=\dM$. It is then natural to look for conditions such that $X$ is again a semimartingale and we have a form of weak `continuity' for the operation of Itô integration in the sense that
\begin{equation}\label{eq:intro_conv}
	\Bigl(X^n \,, \,\int_0^\bullet \, H^n_{s-} \, \diff X^n_s \Bigr) \;  \; \Rightarrow \; \Bigl(X\,,\,\int_0^\bullet \, H_{s-} \, \diff X_s \Bigr) \quad \text{ on } \quad  (\D_{\R^{2d}}[0,\infty),\rho )
\end{equation}
if also $H^n \Rightarrow H$ on $(\D_{\R^d}[0,\infty) ,  \tilde{\rho}\, )$ with $\tilde{\rho}\in \{\dM, \dJ\}$ and with the $H^n$ adapted to $\mathbb{F}^n$. Here, and throughout, the Itô integrals are defined component-wise.

\begin{remark}[Convergence together]\label{rem:together} Given $Y^n$ in $\D_{\mathbb{R}^k}[0,\infty) $ and $Z^n$ in $\D_{\mathbb{R}^m}[0,\infty)$, we say the pairs $(Y^n,Z^n)$ converge \emph{together} (for $\rho$) if they converge on $(\D_{\mathbb{R}^{k+m}}[0,\infty),\rho)$, while they converge \emph{jointly} (for $\rho$) if it is on the product space $(\D_{\mathbb{R}^{k}}[0,\infty),\rho) \times (\D_{\mathbb{R}^{m}}[0,\infty),\rho)$. Joint convergence can of course be with different metrics $\rho, \tilde{\rho} \in \{\dM, \dJ\}$ for $Y^n$ and $Z^n$.
\end{remark}

Working with the $J_1$ topology, an elegant and effective theory was developed in the seminal papers of Jakubowski, M\'emin, \& Pag\`es \cite{jakubowskimeminpages} and Kurtz \& Protter \cite{kurtzprotter}. In particular, \eqref{eq:intro_conv} holds under the following two conditions: Firstly, the pairs $(H^n,X^n)$ must converge weakly \emph{together} on $(\D_{\R^{2d}}[0,\infty),\dJ )$. Secondly, the semimartingale integrators $X^n$ must enjoy some uniform regularity: \cite{jakubowskimeminpages} enforces the P-UT condition (predictable uniform tightness), originally due to Stricker \cite{stricker}, while \cite{kurtzprotter} introduced the UCV condition (uniformly controlled variations). We will not make use of these conditions, but we recall them in Appendix \ref{app:UCV_and_PUT} for reference.

In this paper, we shall instead rely on a new notion of good decompositions for the integrators, and we will focus on different modes of convergence such as joint weak convergence of the integrands and integrators with respect to $\rho, \tilde{\rho} \in \{\dM, \dJ\}$. We find that the notion of good decompositions is easy to work with and can be helpful in pinpointing precisely what is required of each part of the semimartingale decomposition for the integrators.

We stress that convergence results for stochastic integrals continue to find extensive use. Recent examples with Brownian integrators in the limit include scaling limits of Hawkes processes \cite{Rosenbaum}, convergence to SPDEs such as random perturbations of the Schr\"odinger equation \cite{Pellegrini}, and convergence results for Mean Field Games \cite{carmona-delarue-lacker, Lacker2} and McKean--Vlasov control problems \cite{Lacker_limit}.

\subsection{Convergence in the $M_1$ topology}\label{sect:$J_1$to$M_1$}

While the $J_1$ topology requires the jump times and jump sizes of convergent paths to approximately match those of the limit, the $M_1$ topology is much more flexible. For example, a steeper and steeper continuous ascent may result in a jump, and many smaller jumps of essentially the same sign may morph into one large jump. Still, the topology is strong enough that it is well suited for invariance principles with many key functionals preserving the convergence.

We recall the $M_1$ topology and its key properties in Appendix \ref{app:J1_M1_tops}. Recent applications in which it plays a decisive role include compactness arguments and limit theorems for queuing problems \cite{Mikosch, PangWhitt, ramanan_M1}, variants of the Bouchaud trap model \cite{arous_btm, croydonM1}, dispersing billiards with cusps \cite{Jung_Zhang, Melbourne_Varandas}, the supercooled Stefan problem \cite{Dembo_LiCheng, mercy, Nadtochiy_Skholni2}, and singular controls \cite{Asaf_cohen_M1, Fu_Horst}.

Outside the safe confines of metrizability, Jakubowski \cite{jakubowski2} introduced another topology on Skorokhod space, called the $S$ topology, which is weaker than the $M_1$ topology (see \cite{balanjakubowski}). Part of the motivation for this comes from close connections to the aforementioned P-UT condition (due to results of \cite{stricker}), so Jakubowski relies on this when studying stochastic integral convergence on $\D_{\R}[0,1]$ with the $S$ topology in \cite{jakubowski}. Two of the main results \cite[Thms.~1 and 6]{jakubowski} yield the remarkable fact that $(H^n,X^n)\Rightarrow (H,X)$ on $(\D_{\R}[0,1],S) \times (\D_{\R}[0,1],S)$ is sufficient for
\begin{equation}\label{eq:S_top}
	\int_0^\bullet H^n_{s-} \, \diff X^n_s \; \Rightarrow \; \int_0^\bullet H_{s-} \, \diff X_s \quad \text{on} \quad (\D_{\R}[0,1],S),
\end{equation}
provided the $X^n$ satisfy the P-UT condition and provided the interplay of $H^n$ and $X^n$ respects an additional condition that replaces the convergence together in $J_1$ employed by \cite{jakubowskimeminpages, kurtzprotter}. This latter condition will also play a crucial role in our analysis (for details, see Section \ref{subsect:main_results}).

Notably, \cite{jakubowski} also discusses convergence on $\D_{\R}[0,1]$ for a class of topologies compatible with integration in the sense of the discussion preceding \cite[Thm.~4]{jakubowski}. Though details are not given, it is noted that the $M_1$ topology is compatible with integration in that sense, and so \cite[Thm.~4]{jakubowski} implies that
\eqref{eq:S_top} holds on $(\D_{\R}[0,1],\dM)$ when the integrators $X^n$ are weakly convergent on $(\D_{\R}[0,1],\dM)$ and the relevant conditions are satisfied by $X^n$ and $H^n$.

As far as we are aware, \cite{jakubowski} remains the only work to have studied the general question \eqref{eq:intro_conv} beyond the $J_1$ setting. By this, we mean the analysis of general structural conditions such that Itô integrals preserve the convergence of any given families of weakly convergent integrators and integrands on Skorokhod space. There has, however, been related work on certain fast-slow dynamical systems driven by $M_1$ convergent processes, with \cite{chevryev_etal} establishing the convergence to corresponding stochastic differential equations in a suitable rough path topology, while convergence in the Skorokhod topologies fails. See also \cite{Chevryev24, Melbourne13} for further results.

\subsection{Overview of the paper}

In Section \ref{sect:stoch_int_conv}, we give our main results on the weak `continuity' of Itô integration in the sense of \eqref{eq:intro_conv}. Section \ref{subsect:prelim_setup} introduces two key conditions, including the notion of good decompositions. Section \ref{subsect:main_results} then states the central result for the $M_1$ and $J_1$ topologies, together with some simple sufficient criteria. Section \ref{subsect:probability_and_intregrands} presents variants of the central result, relaxing the weak convergence conditions on the integrands and considering convergence in probability, and gives a simple-to-use result for local martingale integrators that are continuous in the limit.

In Section \ref{sect:martingale_integrators}, we focus on local martingale integrators and their structural properties on Skorokhod space. Section \ref{sect:GD_role_sharpness} constructs what appears to be the first example addressing how \eqref{eq:intro_conv} can fail for sequences of martingale integrators. Next, Section \ref{sect:$M_1$_martingales_and_counterex} presents the remarkable finding that, for local martingales, $M_1$ tightness translates to $J_1$ tightness under a mild condition of localised uniform integrability. Section \ref{sect:preservation_of_local_martingality} shows that a slight strengthening of the latter condition yields the preservation of the local martingale property in the limit.

In Section \ref{sect:avci_alternative}, we consider an alternative condition for controlling the interplay between the integrands and integrators. This is presented in Section \ref{subsect:alternative_criteria} and Section \ref{subsect:concrete_avci_alternative} then shows how this can be helpful for a general class of pure jump processes.

Section \ref{sect:CTRW_applications} applies our framework to study scaling limits in models of anomalous diffusion expressed via stochastic integrals for continuous-time random walks. Section \ref{subsect:weak_cont_CTRW} rectifies some issues in the existing literature and, in certain regimes, establishes new results on $J_1$ and $M_1$ convergence which answer questions raised in \cite{scalas, hahn}. Section \ref{sect:counter} provides a novel counterexample showing that, in certain other regimes, the convergence can fail in a fundamental way.

\section{Convergence of stochastic integrals}
\label{sect:stoch_int_conv}

This section presents a series of general results on the weak convergence of stochastic integrals in the $J_1$ and $M_1$ topologies. Convergence in probability is also addressed.

\subsection{Conditions on the integrands and integrators}\label{subsect:prelim_setup}

In \cite{jakubowski}, Jakubowski introduced a condition for excluding, asymptotically, cases where significant increments of the integrands appear immediately before those of the integrators (a problematic situation where non-vanishing mass of the approximating integrals need not propagate to the limit). This condition relies on the following function adapted from \cite[p.~2144]{jakubowski}.

\begin{defi}[Consecutive increment function] \ \label{defi:3.19} 
	For every $\delta>0$ and $T>0$, we define a function $\hat{w}^T_\delta: \D_{\R^d}[0,\infty) \times \D_{\R^d}[0,\infty) \to \R_+$ of the largest consecutive increment within a $\delta$ period of time on the compact interval $[0,T]$, namely
	\[
	\hat{w}^T_{\delta}(x,y)  := \operatorname{sup}\bigl\{  |x^{(i)}(s)- x^{(i)}(t)| \wedge |y^{(i)}(t)-y^{(i)}(u)|  :   s < t  < u \le s+\delta \leq T, \hspace{0.5pt} 1\leq i \leq d    \bigr\}
	\]
	 where $z^{(i)}$ denotes the $i$-th coordinate of a given $z\in \D_{\R^d}[0,\infty)$. The supremum is of course restricted to positive times $ s\geq 0$ and, as usual, $a\wedge b= \operatorname{min}\{a,b\}$.
\end{defi}

\begin{defi}[Asymptotically vanishing consecutive increments]\label{def:avco} \
	Let $X^n$ and $H^n$ be $d$-dimensional càdlàg processes on given probability spaces $(\Omega^n, \mathbbm{F}^n, \Pro^n)$. We say that the pairs $(H^n, X^n)_{n\ge 1}$ have \emph{asymptotically vanishing consecutive increments} if, for every $\gamma>0$ and $T> 0$,
	\begin{align} 
		\lim\limits_{\delta \downarrow 0}  \; \limsup\limits_{n\rightarrow \infty} \; \Pro^n\bigl( \, \hat w_{\delta}^T(H^n, \, X^n) \; > \; \gamma \, \bigr)  =  0. \tag{AVCI} \label{eq:oscillcond}
	\end{align}
\end{defi}

The condition \ref{eq:oscillcond} is a restatement of \cite[Eq.~(6)]{jakubowski} and it will provide the main handle on the interplay between the integrands and integrators. Finally, we also need the integrators to be sufficiently well-behaved. To this end, we shall rely on the following concept.

\begin{defi}[Good decompositions]\label{defn:GD} Let $(X^n)_{n\geq 1}$ be a sequence of semimartingales on a given family of probability spaces $(\Omega^n, \mathcal{F}^n, \F^n, \Pro^n)$. We say that the sequence has \emph{good decompositions} if, for the given filtrations $\F^n$, there exist decompositions
	\[
	X^n = M^n + A^n,\quad M^n  \text{ local martingales},\quad  A^n \text{ finite variation processes},
	\]
	such that, for every $t>0$, we have 
	\begin{align}\label{eq:Mn_An_condition}\tag{GD}
		\lim_{R\rightarrow \infty}  \limsup_{n\rightarrow\infty } \;  \mathbb{P}^n\bigl(\text{TV}_{[0,t]}(A^n)>R\bigr)=0 \quad \text{and}  \quad
		\limsup_{n\rightarrow \infty} \; \mathbb{E}^n\bigl[\, |\Delta M^n_{t \land \tau^n_c}|\, \bigr] <\infty,
	\end{align}
	for all $c>0$, where $\tau^n_c:= \operatorname{inf}\{s >0: |M^n|^*_s \ge c \}$. 
\end{defi}

Here, $\text{TV}_{[0,t]}(A^n)$ denotes the total variation of $A^n$ on $[0,t]$ and $\Delta M^n_t:= M^n_t - M^n_{t-}$ denotes the jump of $M^n$ at time $t$. Moreover, we have used the notation $|M^n|^*_t:= \operatorname{sup}_{s\in[0,t]} |M^n_s|$.

 If $X^n=M^n+A^n$, where the $M^n$ are local martingales with uniformly bounded jumps and the $A^n$ are of tight total variation, then \eqref{eq:Mn_An_condition} holds. Let us consider a non-trivial example of martingales with diverging jump sizes which are \emph{not} of tight total variation on compacts, but we can directly verify the criterion on the jumps in \eqref{eq:Mn_An_condition}.

\begin{example}\label{ex:martingales_GD_directly}Fix $(\Omega, \mathcal{F}, \Pro)=([0,1], \mathscr{B}([0,1]), \text{{\fontfamily{cmss}\selectfont Leb}})$, and define, for $t\geq0$, the processes
     \begin{equation*}
         A^n_t(\omega)  =   
     n\ind_{[0,\frac 1 n]}(\omega) \ind_{[1,\infty)}(t) +  \sum_{k=1}^{n^2} \Bigl[ n^{2}\ind_{[0,\frac 1{n^{k+1}}]}(\omega) \!-\! \bigl(n^{2} +(-1)^k\bigr) \!\ind_{[\frac 1 n-\frac 1{n^{k+1}},\frac 1 n]}(\omega) \Bigr]\! \ind_{[1+\frac k {n^2},\infty)}(t).
     \end{equation*}
     There is a single large jump up at time $1$ with small probability, and then a growing number of larger jumps up or down with smaller and smaller probability between times $1$ and $2$. The natural right-continuous filtration is obtained from $\mathcal{F}^n_0=\{\emptyset, \Omega\}$, $\mathcal{F}^n_1= \{\emptyset, \Omega,[0,1/n], (1/n,1]\}$, and $\mathcal{F}^n_{1+k/n^2}= \sigma(\{\emptyset, \Omega,[0,1/n], [0,1/n^{\ell+1}], [1/n-1/n^{\ell+1}, 1/n] : \ell=1,...,k\})$, for $k=1,...,n^2$. As in the Doob--Meyer decomposition, we can see that $M^n:=A^n - B^n$ is a martingale, where $B^n$ is the simple right-continuous process generated by $B^n_0\equiv 0$, $B^n_1\equiv 1$, and
     $$ B^n_{1+\frac{k}{n^2}}  = 1+ \frac{1}{n} \quad \text{if } k \text{ is odd},\quad \text{or }\quad  B^n_{1+\frac{k}{n^2}} = 1 \quad \text{if } k \text{ is even} ,\quad\text{for}\quad  k=1,...,n^2. $$
     While $M^n$ converges ucp to $t\mapsto \ind_{[1,\infty)}(t)$, neither $A^n$ nor $B^n$ are of tight total variation on compacts, as $\operatorname{TV}_{[0,2]}(M^n)=\operatorname{TV}_{[0,2]}(B^n)= 1+ n$ on the event $(1/n,1]$. Moreover, we have that $\E[\operatorname{sup}_{0\le s\le 2} |\Delta M^n_s|]\ge \E[|\Delta M^n_{1+k/n^2}|]\ge n \to \infty$ as $n\to \infty$. Nevertheless, 
  given $c>0$, one readily sees that $\E[ |\Delta M^n_{t \land \tau^n_c}|] \le 2$ for all $n\ge c$ and $t\geq1$, due to $|\Delta M^n_{t\wedge \tau^n_c}|=|\Delta M^n_{1}|=n$ on the event $[0,1/n]$, so $M^n$ satisfies the jump condition in \eqref{eq:Mn_An_condition}.
\end{example}

 We shall write $X^n \Rightarrow_{f.d.d.} X$ on $D \subseteq [0,\infty)$ if, for all $m\ge 1$ and $t_1,...,t_m \in D$, we have
$$ \Pro^n \circ (X^n_{t_1},...,X^n_{t_m})^{-1} \; \Rightarrow \; \Pro \circ (X_{t_1},...,X_{t_m})^{-1}\, .$$
If $X^n \Rightarrow X$ for $J_1$ or $M_1$, then this holds on the set $D$ of almost sure continuity points of $X$. We can then recall the following result from \cite{jakubowskimeminpages}, modified to rely on good decompositions.

\begin{prop}[Preservation of the semimartingale property]\label{prop:preserve_semimartin} Let $H^n,X^n$ be adapted càdlàg processes defined on filtered probability spaces $(\Omega^n, \mathcal{F}^n, \F^n, \Pro^n)$. Suppose $(|X^n|^*_t)_{n\ge 1}$ is tight, for all $t>0$, and $(H^n,X^n) \Rightarrow_{f.d.d.} (H,X)$ on a dense subset $D\subseteq [0,\infty)$, for some $H,X \in \D_{\R^d}[0,\infty) $. If the $X^n$ admit good decompositions \eqref{eq:Mn_An_condition}, then $X$ is a semimartingale for the natural filtration generated by the pair $(H,X)$.
\end{prop}

While this is often left implicit in our writing, it is of course crucial that \eqref{eq:Mn_An_condition} holds with respect to the given filtrations $\F^n$ for which the $X^n$ and the $H^n$ are adapted.

\subsection{Weak convergence of stochastic integrals on Skorokhod space}\label{subsect:main_results}

For a given $\R^d$-valued semimartingale $X=(X^{(1)},...,X^{(d)})$ and a given $\R^d$-valued integrand $H=(H^{(1)},...,H^{(d)})$, we define the stochastic integral of $H_-$ against $X$ componentwise, i.e.
\begin{equation}\label{eq:integrals_componentwise}
\int_0^\bullet \, H_{s-} \, \diff X_s \; := \; \Bigl( \, \int_0^\bullet \, H^{(1)}_{s-} \, \diff X_s^{(1)}, \, ... \,, \, \int_0^\bullet \, H^{(d)}_{s-} \, \diff X_s^{(d)} \, \Bigr).
\end{equation}
We can now state our first result on the weak convergence properties of stochastic integrals.

\begin{theorem}[Weak `continuity' of It{\^o} integrals] \ \label{thm:3.19}
Let $(\Omega^n, \mathcal{F}^n, \F^n, \Pro^n)$ be a family of filtered probability spaces supporting a sequence of semimartingales $(X^n)_{n\ge 1}$ with good decompositions \eqref{eq:Mn_An_condition}. Let $(H^n)_{n\ge 1}$ be any given sequence of adapted càdlàg processes for the same filtered probability spaces such that (i) there is joint weak convergence
	\begin{equation}\label{eq:joint_conv}
		(H^n,X^n) \; \Rightarrow \; (H,X)  \quad \text{ on } \quad (\D_{\R^d}[0,\infty)\, , \, \tilde{\rho}\, )\times (\D_{\R^d}[0,\infty)\, , \,\rho)
	\end{equation}
	with $\rho,\tilde{\rho} \in \{ \dM, \dJ \}$, for some càdlàg limits $H$ and $X$, and (ii) the pairs $(H^n, X^n)$ satisfy \eqref{eq:oscillcond}. Then, $X$ is a semimartingale in the filtration generated by the pair $(H,X)$ and
	\begin{equation}\label{eq:concerted_stoch_int_conv}
		\Bigl( X^n, \, \int_0^\bullet \, H^n_{s-} \, \diff X_s^n \Bigr) \;  \; \Rightarrow \; \Bigl(  X, \, \int_0^\bullet \, H_{s-} \, \diff X_s \Bigr) \quad \text{on} \quad (\D_{\R^{2d}}[0,\infty), \, \rho).
	\end{equation}
	If, in addition, $(H^n,X^n) \, \Rightarrow \, (H,X)$ on $(\D_{\R^{2d}}[0,\infty),\rho)$, then it furthermore holds that
	\begin{equation}\label{eq:conv_together}
		\Bigl(  H^n, \, X^n, \, \int_0^\bullet \, H^n_{s-} \, \diff X_s^n \Bigr) \;  \; \Rightarrow \; \Bigl( H, \, X, \, \int_0^\bullet \, H_{s-} \, \diff X_s \Bigr) \quad \text{on} \quad (\D_{\R^{3d}}[0,\infty), \, \rho).
	\end{equation}
\end{theorem}
\begin{remark} Naturally, all statements also hold with $[0,\infty)$ replaced by $[0,T]$, for $T>0$.
\end{remark}

While \eqref{eq:oscillcond} seems difficult to work with in general, our next result singles out two simple sufficient criteria. These provide intuition, and may be all that is needed in many applications. The first criterion makes a precise connection
to the conditions of \cite{jakubowskimeminpages,kurtzprotter} (and \cite{shiryaev,graham}) in the $J_1$ setting. Concerning the second criterion, it was noted in \cite[Rmk.~4]{jakubowski} that if the integrands or integrators form a tight sequence for the $J_1$ topology with all limit points supported on the space of continuous paths, then \eqref{eq:oscillcond} is satisfied. As a natural generalisation of this, we make the useful observation that one only needs to rule out common jumps in the limit.

\begin{prop}[Sufficient conditions for AVCI]  \label{prop:3.3}
	In the setting of Theorem \ref{thm:3.19}, the condition \eqref{eq:oscillcond} is satisfied if one of the following two criteria holds:
	\begin{enumerate}[(i)]
		\item The pairs $(H^n, X^n)$ converge together to $(H, X)$ weakly in the $J_1$ topology, meaning that \eqref{eq:joint_conv} holds on $(\D_{\mathbb{R}^{2d}}[0,\infty), \, \dJ)$. \label{it:prop_suff_cond_AVCO}
		\item The limiting processes $H$ and $X$ almost surely have no common discontinuities, that is,
		\[
		\operatorname{Disc}(H) \; \cap \; \operatorname{Disc}(X) \; = \; \emptyset \quad \text{almost surely}. \vspace{-2pt}
		\]
	\end{enumerate}
\end{prop}

\begin{remark}[Weak and strong $M_1$ topology]\label{rem:weak$M_1$} Whitt's monograph \cite{whitt} defines a `strong' and `weak' $M_1$ topology for $d\geq 2$. The standard (or `strong') metric $\dM$ from Definition \ref{def:$M_1$_metric} induces the `strong' topology. The `weak' version is simply the product topology on $\D_{\mathbb{R}^{d}}[0,\infty) \cong (\D_{\mathbb{R}}[0,\infty))^{ d} $ when each $\D_{\mathbb{R}}[0,\infty)$ is endowed with the $M_1$ topology. All our results also hold if $\dM$ is taken to be a metric inducing the `weak' $M_1$ topology. The proofs remain identical, except for the arguments to control \eqref{eq:T2} in the proof of Theorem \ref{thm:3.19} and its later variant, Theorem \ref{prop:4.36} (see Section \ref{sect:proof_main_weak_cont}). With the $d$-fold product topology, one instead relies on multiple applications of Proposition \ref{thm:ContinuityOfSimpleIntegralsInM_1Topology} in dimension one and otherwise proceeds analogously.
\end{remark}

In applications, tightness is typically much easier to verify individually for the $H^n$ and $X^n$. This is sufficient for obtaining joint weak convergence of the pairs $(H^n,X^n)$ on the product space $\D_{\mathbb{R}^d}[0,\infty) \times \D_{\mathbb{R}^d}[0,\infty)$. By contrast, convergence together (i.e., for the `strong' topology) requires tightness in $\D_{\mathbb{R}^{2d}}[0,\infty)$ which is a substantially more restrictive property.

\begin{remark}[Matrices and dot product integrals]\label{rem:matrix_dot}
	If $\boldsymbol{H}$ is a $(k \times m)$-matrix and $\boldsymbol{X}$ is an $(m \times d)$-matrix, we can define a matrix stochastic integral $\int_0^\bullet \, \!\boldsymbol{H}_{s-} \diff \boldsymbol{X}_s$ as the $(k \times d)$-matrix whose components are the `dot product' stochastic integrals \[
	\Bigl( \, \int_0^\bullet \, \boldsymbol{H}_{s-} \diff \boldsymbol{X}_s\, \Bigr)_{j,\ell} \; := \; \int_0^\bullet \,\boldsymbol{H}^{(\text{row}\,j)}_{s-} \cdot \diff \boldsymbol{X}^{(\text{col}\,\ell)}_s \; :=  \; \sum_{i=1}^m 	\int_0^\bullet \, H^{(j,i)}_{s-}\, \diff X^{(i,\ell)}_s.
	\]
	By Proposition \ref{prop:A3}, linearly combining components is a continuous operation for convergence together on $(\D_{\R^{kd}}[0,\infty), \rho)$. Thus, the continuous mapping theorem and Theorem \ref{thm:3.19} give weak convergence of matrix stochastic integrals on $(\D_{\R^{k\times d}}[0,\infty), \rho)$, by identifying $\R^{k \times d}$ with $\R^{kd}$, if we assume \eqref{eq:oscillcond} holds for each sequence of pairs $(\boldsymbol{H}^{n,(\text{row}\,j)},\boldsymbol{X}^{n,(\text{col}\,\ell)})_{n\ge 1}$.
\end{remark}

\subsection{Convergence in probability and relaxed conditions on integrands}\label{subsect:probability_and_intregrands}
In view of Theorem \ref{thm:3.19}, it is not surprising that, if, in place of weak convergence, we require the integrands and integrators to converge in probability on Skorokhod space, then this mode of convergence carries over to the stochastic integrals. The next result makes this precise.

\begin{theorem}[Functional convergence in probability] \label{thm:2.12}   
	Let $(X^n)_{n\ge 1}$ be a sequence of $d$-dimensional semimartingales with good decompositions \eqref{eq:Mn_An_condition} on a common filtered probability space $(\Omega, \mathcal{F}, \F, \Pro)$. Consider a sequence of adapted càdlàg processes $(H^n)_{n\ge 1}$ on the same filtered probability space such that
	\begin{align} \Pro \left( \, \tilde \rho(H^n, H) \, + \, \rho(X^n, X) \; > \; \gamma\, \right) \; \conv{n\to \infty} \; \quad 0, \label{eq:probconv_HnXn}
	\end{align}
	for every $\gamma>0$, for some adapted càdlàg processes $H$ and $X$, where $\tilde{\rho},\rho \in \{\dM, \dJ\}$. Furthermore, suppose the pairs $(H^n, X^n)$ satisfy \eqref{eq:oscillcond}. Then, $X$ is a semimartingale in the natural filtration generated by the pair $(H, X)$ and, on $\D_{\R^{2d}}[0,\infty)$, we have
	\begin{align} \Pro \Bigl( \, \rho \Bigl( \, \bigl(\, X^n,\, \int_0^\bullet \, H^n_{s-} \; \diff X^n_s\, \bigr) \; ,  \; \bigl( \, X, \, \int_0^\bullet \, H_{s-} \; \diff X_s\, \bigr) \, \Bigr) \; > \; \gamma\, \Bigr) \; \conv{n \to \infty} \; \quad 0 \label{eq:probconv_integral} \end{align}
	for all\;\,$\gamma>0$. If, additionally, $\rho((H^n, X^n), (H, X))$ tends to zero in probability, then there is also convergence together with the $H^n$ in \eqref{eq:probconv_integral} on the strong space $\D_{\R^{3d}}[0,\infty)$.
\end{theorem} 

Convergence in probability is also considered in \cite{jakubowski,jakubowskimeminpages,kurtzprotter}. As in those works, the proof amounts only to the slightest of modifications to the proof of weak convergence in Theorem \ref{thm:3.19}.

The results of \cite{jakubowski} for the $S$ topology on $\D_{\R}[0,1]$ show that we can aim for weaker assumptions on the integrands. The next result adapts this to our setting. To this end, we shall need the function $N^T_\delta$ that returns the number of $\delta$-increments of a path up to time $T$, namely
\begin{align} N^T_\delta(\alpha) \; := \;   \sup{} \!\bigl\{ n \,  : \,  0=t_1\le t_2 \le ...\le t_{2n}=T \; , \;  |\alpha(t_{2i})-\alpha(t_{2i-1})|\ge \delta \bigr\} \label{eq:maxnumosc} \end{align}
for $\alpha:[0,T] \to \R^d$.
Moreover, we will write $\operatorname{Disc}_\Pro(H,X):=\{s>0: \Pro(\Delta(H,X)\neq 0)>0\}$, and we recall that this defines a countable set of times for c\`adl\`ag processes $H$ and $X$.

\begin{prop}[Relaxed assumptions on the integrands]\label{remark:weaker_integrnd_conv}Consider the same setting as Theorem \ref{thm:3.19}, but let the convergence of $(H^n,X^n)$ in \eqref{eq:joint_conv} be replaced by the following:
	\begin{enumerate}
		\renewcommand{\theenumi}{F\arabic{enumi}}
		\item  \label{it:alt_int_crit1} The finite-dimensional distributions of the $H^n$ converge along a dense subset $D \subseteq [0,\infty) \setminus \operatorname{Disc}_\Pro(H,X)$, together with the full distributions of $X^n$, i.e., 
		\[
		(X^n, \, H^n_{t_1}, \, H^n_{t_2},\, ...,\, H^n_{t_k}) \; \;  \Rightarrow \; \;  (X, \, H_{t_1}, \, H_{t_2},\, ...,\, H_{t_k})
		\]
		on $(\D_{\R^d}[0, \infty), \rho) \times (\R^{d\times k}, |\cdot|\, )$,	for all $k\ge 1$ and $t_1,...,t_k \in D$;
		\item \label{it:alt_int_crit2} The integrands are uniformly stochastically bounded on compacts, i.e., the family $\{|H^n|^*_T\}_{n \ge 1}$ is tight, for each $T>0$;
		\item \label{it:alt_int_crit3} The maximal number of `large increments' of the $H^n$ is tight, i.e., $\{N^T_\delta(H^n)\}_{n\ge 1}$ is tight for any $\delta>0$ and $T>0$, where $N^T_\delta$ is given by \eqref{eq:maxnumosc}.
	\end{enumerate}
	Given \eqref{it:alt_int_crit1}--\eqref{it:alt_int_crit3}, it then still holds that $X$ is a semimartingale in the filtration generated by the limiting pair $(H,X)$ and we still have the weak convergence
	\begin{equation}\label{eq:relaxed_integrands_conv}
		\Bigl( \, X^n, \, \int_0^\bullet \, H^n_{s-} \, \diff X_s^n\, \Bigr) \;  \; \Rightarrow \;\Bigl( \, X, \, \int_0^\bullet \, H_{s-} \, \diff X_s\, \Bigr) \quad \text{ on } \quad (\D_{\R^{2d}}[0,\infty), \, \rho).
	\end{equation}
	Furthermore, if \eqref{it:alt_int_crit1} holds with convergence in probability in place of weak convergence, for all $t_1,...,t_k \in D$, then we obtain the convergence in probability as in \eqref{eq:probconv_integral}.
\end{prop}

 We stress that \eqref{it:alt_int_crit1}--\eqref{it:alt_int_crit3} are satisfied if  $(X^n, H^n)_{n\geq1}$ is tight on the $M_1$ product space. The proof of Proposition \ref{remark:weaker_integrnd_conv} follows by analogy with that of Theorem \ref{thm:3.19}, as the only properties of the integrands we make use of, beyond \eqref{eq:oscillcond}, are \eqref{it:alt_int_crit1}--\eqref{it:alt_int_crit3}. One quickly
 sees that neither \eqref{it:alt_int_crit1} nor \eqref{it:alt_int_crit2} can be dropped in the statement. More surprisingly, the next example shows that this is also the case for \eqref{it:alt_int_crit3}, even for martingale integrators with a continuous limit.

\begin{example} 
	Consider a rescaled classical random walk $X^n:= n^{-1/2} \sum_{k=1}^{\lfloor n \bullet \rfloor} \xi_k$ for i.i.d.~random variables $\xi_k$ such that $\Pro(\xi_1=1)=\Pro(\xi_1=-1)=1/2$, and consider integrands
	$$H^n_t \; := \; \sum_{k=1}^\infty  \, -\operatorname{sgn}(X^n_{\tau^n_{2k-1}}) \ind_{[\tau^n_{2k-1},\, \tau^n_{2k})}(t),$$ where the stopping times are defined inductively by $\tau^n_0\equiv 0$ and 
		$\tau^n_{2k+1}:= (\tau^n_{2k}+n^{-1})\wedge 2n^{-1/4}$ as well as
		$\tau^n_{2k}:= \operatorname{inf} \{ t> \tau^n_{2k-1} : X^n_t = 0\} \wedge 2n^{-1/4}$, for all $k\ge 1$. Therefore, the $H^n$ take the value $1$ on periods of upcrossings of the $X^n$ from $-n^{-1/2}$ to $0$ and the value $-1$ on periods of downcrossing from $n^{-1/2}$ to $0$ (up to time $t=2n^{-1/4} $ and they are constantly zero thereafter). Clearly, the $X^n$ are martingales converging weakly to a Brownian motion with respect to the topology of uniform convergence on compacts. Since the limit is a continuous process, the pairs $(H^n,X^n)$ have \eqref{eq:oscillcond} due to Proposition \ref{prop:3.3}. Moreover, the $H^n$ satisfy \eqref{it:alt_int_crit1}-- \eqref{it:alt_int_crit2}, but \emph{not} \eqref{it:alt_int_crit3}, and the finite-dimensional distributions of the integrals $\int H^n_{s-} \diff X^n_s$ escape to infinity. To see this, let $S_j:=\sum_{i=0}^j \xi_i$ and $Z_m:= \sum_{j=1}^m \ind_{\{S_j =0\}}$,  the number of zeros of $(S_j)_{j=1}^m$, which is equal to the number of zeros of $(X^n_{k/n})_{k=1}^m$. Thus, with $r^n:= \operatorname{max}\{k \ge 1 : \tau^n_{2k} < 2n^{-1/4}\}$, we have $r^n=Z_{\lfloor 2n^{3/4} \rfloor-1}\ge Z_{2 \lfloor n^{3/4} \rfloor -2}$. By \cite[Thm.~4.3.5]{durrett} and the tightness of $X^n$, for any $\varepsilon>0$ we can find $L_\varepsilon, K_\varepsilon>0$ and $n\ge N$ so that 
	\begin{align}
		\Pro \bigl( \, r^n \, > \, 2L_\varepsilon \lfloor n^{3/4} \rfloor -2 \; , \; |X^n|^*_1 \, \le \, K_\varepsilon \, \bigr) \; \ge \; 1- \varepsilon. \label{eq:example_F3_not_omittable}
	\end{align}
	Hence, \eqref{it:alt_int_crit3} is not satisfied for all $\delta\le 1$. For $t>0$ and $n$ large, on the event from \eqref{eq:example_F3_not_omittable},
	\begin{align*} \int_0^t \, H^n_{s-} \diff X^n_s \; &= \; \sum_{k=1}^{r^n} \, - \operatorname{sgn}(X^n_{\tau^n_{2k-1}})\, (X^n_{\tau^n_{2k}}-X^n_{\tau^n_{2k-1}}) \, + \, (X^n_{2n^{-1/4}} - X^n_{\tau^n_{2r^n+1}}) \\
		&= \; n^{-1/2}\, r^n \, + \, (X^n_{2n^{-1/4}} - X^n_{\tau^n_{2r^n+1}}) \; \ge \;\frac{2L_\varepsilon \lfloor n^{3/4} \rfloor -2}{\sqrt{n}} \; - 2 K_\varepsilon \; \conv{n\to \infty} \; \infty.
	\end{align*}
\end{example}

Despite this example, local martingale integrators that are continuous in the limit can allow for more general integrands, and we can even have rather simple looking criteria. However, the criteria must then link the behaviour of the integrands and integrators, as achieved by the weak convergence condition \eqref{eq:interplay_cond} below, and we need some additional regularity of the integrators: we will say that a sequence $(Y^n)_{n\ge 1}$ of local martingales has \emph{vanishing jumps in expectation at hitting times of its quadratic variation} if there is a $\delta>0$ such that, for every $t>0$,
\begin{align} \limsup_{n\to \infty} \;  \E^n[ \, |\Delta Y^n_{t\wedge \tau^n_c}|\, ] \; = \; 0 \label{eq:vanishing_jumps_condition} \end{align}
for all $0<c<\delta$, where $\tau^n_c:= \operatorname{inf}\{ s>0: [Y^n]_s \ge  c\}$.

\begin{cor}[Local martingale integrators with continuous limit] \label{cor:2.14}
	Consider a family of filtered probability spaces $(\Omega^n, \mathcal{F}^n, \F^n, \Pro^n)$ supporting a sequence of càdlàg local martingales $(X^n)_{n\ge 1}$. Suppose $X^n \Rightarrow X$ on $(\D_{\R^d}[0,\infty) ,  \dM )$ for a limit $X$ with values in $\C_{\R^d}[0,\infty)$, defined on a filtered probability space $(\Omega, \mathcal{F}, \F, \Pro)$. Let $(H^n)_{n\ge 1}$ be a sequence of predictable processes on $(\Omega^n, \mathcal{F}^n, \F^n, \Pro^n)$ such that (i) the Itô integrals $Y^n:=\int_0^\bullet H^n_s \diff X^n_s$ are local martingales, (ii) the $Y^n$ satisfy \eqref{eq:vanishing_jumps_condition}, and (iii) there is weak convergence 
	\begin{align} \label{eq:interplay_cond}
		\int_0^\bullet \, H^n_s \, \diff [X^n]_s  \Rightarrow \int_0^\bullet \, H_s \, \diff [X]_s \quad \text{and} \quad  \int_0^\bullet \, (H^n_s)^2 \, \diff [X^n]_s\Rightarrow  \int_0^\bullet \, H_s^2 \, \diff [X]_s \end{align}
	on $(\D_{\R^{d}}[0,\infty) , \dM )$. Then, it holds that
	\begin{equation*}
		\Bigl( X^n \, , \, \int_0^\bullet \, H^n_{s} \, \diff X_s^n\Bigr) \;  \; \Rightarrow \; \Bigl( X \, , \, \int_0^\bullet \, H_{s} \, \diff X_s\Bigr) \quad \text{ on } \quad (\D_{\R^{2d}}[0,\infty), \, |\cdot|^*_\infty).
	\end{equation*}
\end{cor}

An inspection of the proof reveals that, instead of assuming \eqref{eq:vanishing_jumps_condition} for the $Y^n$, it alternatively suffices to have \eqref{eq:vanishing_jumps_condition} hold for the integrators $X^n$ if additionally there is tightness of the running supremum of the integrands $H^n$ on compacts. We also note that predictability of the $H^n$ in Corollary \ref{cor:2.14} simply serves to ensure the stochastic integrals are well-defined; it can be replaced by progressive measurability if the $X^n$ are continuous. As a simple illustration, Corollary \ref{cor:2.14} applies to Brownian motions $X^n$ and progressively measurable integrands $H^n$ satisfying $H^n \Rightarrow H$ on $\textbf{L}_{\text{loc}}^2([0,\infty), \R^d)$ in the usual topology. Indeed, \eqref{eq:interplay_cond} is then readily deduced from Skorokhod's representation theorem and the Cauchy--Schwartz inequality.

\subsection{Convergence of simple integrals}\label{sect:simple_integrals_$M_1$}

For simple stochastic integrals, we can have pathwise $M_1$ convergence. We show this from first principles, working directly with the definition of the $M_1$ metric. This will be an important ingredient in the proofs of the weak convergence results stated above.

We shall say that $h:[0,T]\rightarrow \mathbb{R}^d$ is a simple left-continuous path if it is of the form
\begin{equation}\label{eq:ClassicalLeftContinuousPaths} 
	h = h_0 \mathbbm{1}_{\{0\}} + \sum_{i=1}^{k} h_i\mathbbm{1}_{(t_i, \,t_{i+1}]} = h(0) \mathbbm{1}_{\{0\}} + \sum_{i=1}^{k} h(t_{i+1})\mathbbm{1}_{(t_i, \,t_{i+1}]}
\end{equation}
for fixed $h_0,\ldots,h_{k+1} \in \mathbb{R}$ and a strictly increasing sequence of times $0 = t_1  < \cdots < t_{k+1} = T $.

\begin{prop}[$M_1$ continuity of simple integrals]\label{thm:ContinuityOfSimpleIntegralsInM_1Topology} 
	Consider $y, y_n \in \D_{\R^d}[0,T]$, and let $x, x_n\in\D_{\R^d}[0,T]$ be such that $x(\bullet\, -)$ and $x_n(\bullet \, -)$ are simple left-continuous paths of the form \eqref{eq:ClassicalLeftContinuousPaths}, for all $n\ge 1$. Suppose that $\operatorname{Disc}(x) \cap \operatorname{Disc}(y) = \emptyset$ and that the number of jumps of the $x_n$ are uniformly bounded in $n \ge 1$, i.e., there exists $\tilde{K}>0$ such that 
	$$ \sup{n \geq 1} \; \operatorname{card}(\, \{s \in [0,\, T]\, : \, \Delta x_n(s) \neq 0\}\, ) \; \leq \; \tilde{K}.$$
	If $x_n\to x$ and $y_n \to y$ in $(\D_{\R^d}[0,T]\, ,\, \dM )$, then
\begin{equation}\label{eq:simple_integral_map}\Bigl(x_n, \,y_n,\!\int_0^\bullet \!\!x_n(s-) \diff y_n(s)\Bigr) \rightarrow  \Bigl(x, \,y,\!\int_0^\bullet\!\! x(s-)\diff y(s)\Bigr)\quad \text{in}\quad (\D_{\mathbbm{R}^{3d}}[0,\,T]\, , \,\dM).
    \end{equation}
\end{prop}

\section{Martingale integrators} \label{sect:martingale_integrators}

This section first explores when and why the weak continuity results for stochastic integrals may fail for local martingale integrators. Next, we study $J_1$ versus $M_1$ tightness of local martingales. Finally, we study the preservation of the local martingale property in the limit.

\subsection{Exploding integrals for uniformly vanishing martingales}\label{sect:GD_role_sharpness}

Based on an example due to Jacod, \cite[Ex.~7.14]{graham} and \cite[Ex.~VI.6.27]{shiryaev} highlight that the convergence of stochastic integrals can go wrong without enough control on the integrators. The examples consider deterministic finite variation integrators whose total variation diverges on compacts. Surprisingly, we have not been able to locate any examples that address how, or even if, things can go wrong with local martingale integrators. Here, we provide a novel example that illustrates why the control on the jumps in \eqref{eq:Mn_An_condition} is needed to avoid explosion of the Itô integrals in Theorem \ref{thm:3.19}. This will turn out to serve as a counterexample to two existing claims in the literature, as we discuss below. The example constructs a sequence of martingales that converge uniformly to zero but fail to have good decompositions \eqref{eq:Mn_An_condition}.

\begin{example}\label{ex:main_counterexample}
	Fix $T>0$ and $\varepsilon\in(0,1)$. Set $\Omega:=[0,T]$ and $\mathcal{F}:=\B([0,T])$, equipped with $\Pro:= T^{-1}(1-\varepsilon)\text{{\fontfamily{cmss}\selectfont Leb}} + \varepsilon \delta_T$, where $\delta_{T}$ is the Dirac measure at $T$. Now let $\tau := \operatorname{Id}$ on $ [0,T]$, and define a filtration by $\mathcal{F}_t:= \sigma(\B([0,t\wedge T]))$ for which $\tau$ is then a stopping time. Next, define
	\begin{align*} f(t)  &:= \begin{cases} 
			t^{-1} \cos \bigl(t^{-1} \bigr) \, - \, \sin\bigl(t^{-1} \bigr),  \quad & t \neq 0 \\
			0, \quad & t=0,
		\end{cases}\\
		g^n(t) &:= (T+ n^{-1} -t)\ind_{[T+n^{-1}-n^{-1/2},\,  \, T)}(t),
	\end{align*}
	for each $n\geq 1$. We are then interested in a martingale that jumps exactly at the random time $\tau$, but only if $\tau < T$, and which has a random jump size of $\Psi^n(\tau)$, where
	$$ \Psi^n(t) \; := \; \frac{1}{1-\varepsilon} \; f(g^n(t)) \; \Pro((t,T]).$$ For $0<s<t<T$ and $t\ge T+n^{-1}-n^{-1/2}$, we have
	\begin{align*}
		\E \left[ \Psi^n(\tau)\, \ind_{\{\tau \le t\}} \, | \, \mathcal{F}_s \right] \; = \; \Psi^n(\tau)\, \ind_{\{\tau \le  s\}} \; + \; \ind_{\{\tau >s \}} \, \E \left[ \Psi^n(\tau)\, \ind_{\{s< \tau \le t\}} \, | \, \tau > s \right]
	\end{align*}
	where we have used that $\{\tau > s\}$ is an atom of $\mathcal{F}_s$. Now, for $\tilde \omega \in \{\tau >s\}$, we have
	\begin{align*}
		\E \left[ \Psi^n(\tau)\, \ind_{\{s< \tau \le t\}} \, | \, \tau > s \right](\tilde \omega) \; &= \; \int_{\Omega} \Psi^n(\tau(\omega))\, \ind_{\{s< \tau(\omega) \le t\}} \; \Pro(\diff \omega\,  | \, \tau > s) \\
		&= \; \frac{1-\varepsilon}{ T \Pro(\tau>s)}\;  \int_s^t \Psi^n(x)\, \diff x,
	\end{align*}
	which is clearly not zero in general, but we can correct for this by noting that
	\begin{align*}
		(1-\varepsilon) \, \int_s^t \Psi^n(x) \diff x  \; &= \;  \int_s^t f(g^n(x)) \diff x \; - \;  T^{-1}(1-\varepsilon) \int_s^t f(g^n(x))  x   \diff x \\
		&= \;  [\Lambda^n(x)]^t_s \; - \;  T^{-1}(1-\varepsilon)[\Lambda^n(x) x]^t_s \; + \; T^{-1}(1-\varepsilon) \int_s^t \Lambda^n(x) \diff x \\
		&= \; \Lambda^n(t)\Pro(\tau >t) - \Lambda^n(s)\Pro(\tau>s) \; + \; T^{-1}(1-\varepsilon) \int_s^t  \Lambda^n(x) \diff x,
	\end{align*}
	where
	$\Lambda^n(t) \; := \; g^n(t) \sin(g^n(t)^{-1})$. This
	is essentially the antiderivative of $f\circ g^n$ on $(T+n^{-1}-n^{-1/2},T)$. Note that $t\mapsto T^{-1} \Lambda^n(t\wedge \tau)$ is predictable, as a composition of the deterministic map $t\mapsto T^{-1}\Lambda^n(t)$ and the continuous process $t\mapsto t\wedge \tau$. Given the above computations, one readily sees that it compensates $t\mapsto \Psi^n(\tau) \ind_{\{\tau \le t\}}$. Defining $M^n_t:= \Psi^n(\tau) \ind_{\{\tau \le t\}} \, - \, T^{-1}\Lambda^n(t\wedge \tau)$, we thus obtain a martingale, for all $n\geq 1$. Figure \ref{fig:7} provides an illustration of this construction. Since $\Psi^n(t)=0$ for all $t \notin [T+n^{-1}-n^{-1/2},T)$,
	$$ |M^n(\omega)|^*_\infty \; \le \; |\Psi^n(\omega)| \; + \; T^{-1}|g^n|^*_T \; \le \; |\Psi^n(\omega)| \; + \; T^{-1} n^{-1/2} \; \rightarrow \;  0,$$
	as $n\rightarrow \infty$, for all $\omega \in \Omega$. Hence, $M^n$ converges to zero almost surely in the uniform norm.
    
   We also see that the processes $t \mapsto  \Psi^n(\tau) \ind_{\{\tau \le t\}}$ are of tight total variation. However, the predictable compensator will have exploding total variation on the event $\{\tau=T\}$ (which occurs with probability $\varepsilon>0$). Indeed, writing $\phi(x):= x \sin(1/x)$, it holds for any $\omega \in \{\tau=T\}$ that
\begin{align*}
	&\operatorname{TV}_{[0,T]}(\Lambda^n(\bullet \wedge \tau(\omega)))\; = \; \operatorname{TV}_{[T+n^{-1}-n^{-1/2}, \, T)}(\Lambda^n)\;  = \; \operatorname{TV}_{[n^{-1}, \, n^{-1/2})}(\phi), \\
	&\operatorname{TV}_{[n^{-1}, \, n^{-1/2})}(\phi)\; \; \ge \; \; \sum_{k\; : \; n^{-1} \, \le \, (2\pi k)^{-1}\; , \; (\frac{4k+1}{2} \pi)^{-1}\, \le \, n^{-1/2}} \; \frac{4k+1}{2} \pi \; \ge \; 2\pi \sum_{k= \lceil \frac{\sqrt{n}}{2\pi}\rceil  }^{\lfloor \frac{n}{2\pi}\rfloor} \; k,
\end{align*}
where the sum on the right-hand side tends to infinity as $n\to \infty$. Since $t \mapsto \Lambda^n(t \wedge \tau)$ is continuous on $[T+n^{-1}-n^{-1/2},\, T)$ and zero outside of it, for each $n\geq 1$ there is a left-continuous process $H^n$ with $|H^n|=1$ so that $\int_0^T n^{-1/2} \, H^n_s \diff \Lambda^n(s \wedge \tau)= n^{-1/2} \operatorname{TV}_{[0,T]}(\Lambda^n(\bullet \wedge \tau))$ which tends to infinity on a set of positive probability by what we just saw. In particular, the stochastic integrals of the left-continuous integrands $n^{-1/2} H^n$ against $M^n$ explode. Since both sequences converge to zero uniformly, almost surely, Proposition \ref{prop:3.3} and Theorem \ref{thm:3.19} imply that the $M^n$ fail to have good decompositions, as we would otherwise have a contradiction.
\end{example}

\begin{figure}[h]
	\begin{tikzpicture}[scale=0.85]
		\begin{axis}[
			xmin=0.85,xmax=1.025,
			ymin=-6,ymax=6,
			xlabel = {$t$},
			ylabel = {$\Psi^n(t)$},
			ytick = {},
			ytick style={draw=none},
			yticklabels={,,},
			xtick = {0.89,0.943,1},
			xtick style={draw=none},
			xticklabels={${\small T-\frac{\sqrt{n}-1}{n}} \qquad \; $,{\tiny \color{red!50} $\tau(\omega)$},$\, \quad T$},
			axis x line=middle,
			axis y line=middle,
			every axis x label/.style={at={(ticklabel* cs:1.03)}, anchor=west},
			every axis y label/.style={at={(ticklabel* cs:1.03)},anchor=south},
			]
			\addplot[cmhplot,-,domain=0:0.89]{0};
			\addplot[cmhplot,-,smooth, samples=500,domain=0.89:1]{(10/9)*((1/10)+(9/10)*(1-x))*(1/(1+(1/45)-x)*cos(deg((1/(1+(1/45)-x))) - sin(deg(1/(1+(1/45)-x))))};
			\addplot[cmhplot,-,domain=1:1.025]{0};
			\addplot[soldot]coordinates{(0.89,0.5)};
			\addplot[holdot]coordinates{(1,2.6)};
			\addplot[holdot]coordinates{(0.89,0)};
			\addplot[soldot]coordinates{(1,0)};
			\addplot +[mark=none, color=red!50] coordinates {(0.943, 0) (0.943, 2.2)};
			\addplot[soldot, color=red!50]coordinates{(0.943,2.1)} node[above,pos=1] {\tiny $\Psi^n(\tau(\omega))$} ;
		\end{axis}
	\end{tikzpicture}
	\quad
	\begin{tikzpicture}[scale=0.85]
		\begin{axis}[
			xmin=0.85,xmax=1.025,
			ymin=-0.4,ymax=0.4,
			xlabel = {$t$},
			ylabel = {$\Lambda^n(t)$},
			ytick = {},
			ytick style={draw=none},
			yticklabels={,,},
			xtick = {0.89,0.951, 1},
			xtick style={draw=none},
			xticklabels={$T-\frac{\sqrt{n}-1}{n}\; \;$,{\tiny \color{red!50} $\tau(\omega)$},$\, \quad T$},
			axis x line=middle,
			axis y line=middle,
			every axis x label/.style={at={(ticklabel* cs:1.03)}, anchor=west},
			every axis y label/.style={at={(ticklabel* cs:1.03)},anchor=south},
			]
			\addplot[cmhplot,-,smooth, samples=500, domain=0.89:1]{(1+(1/45)- x)*sin(deg(1/(1+(1/44)-x)))};
			\addplot[cmhplot,-,domain=0:0.89]{0};
			\addplot[cmhplot,-,domain=1:1.025]{0};
			\addplot[soldot]coordinates{(0.89,0.125)};
			\addplot[holdot]coordinates{(1,0.019)};
			\addplot[holdot]coordinates{(0.89,0)};
			\addplot[soldot]coordinates{(1,0)};
			\addplot +[mark=none, color=red!50] coordinates {(0.951, 0) (0.951, 0.065)};
			\addplot[soldot, color=red!50]coordinates{(0.951,0.066)} node[above,pos=1] {\tiny $\Lambda^n(\tau(\omega))$} ;
		\end{axis}
	\end{tikzpicture}
	\caption{Graphical representation of the functions $\Psi^n$ and $\Lambda^n$ with a realisation of $\tau$.}
	\label{fig:7}
\end{figure}
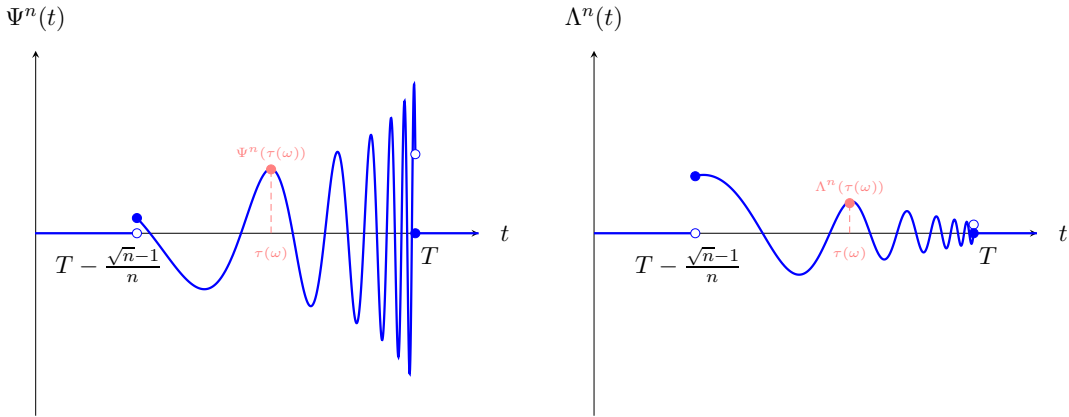

In the econometrics literature, the classical results of \cite{kurtzprotter, jakubowskimeminpages} have been used extensively, especially in work on cointegration theory \cite{paulaskas_rachev}. Within this literature, the recent note \cite{fabozzi_comment} pointed to a gap in the proof of \cite[Thm.~2]{caner1}. In particular, \cite[Thm.~2]{caner1} says that if $(X^n)_{n\geq 1}$ is a weakly $J_1$ convergent sequence of locally square-integrable martingales, then \eqref{eq:concerted_stoch_int_conv} holds in the $J_1$ topology whenever $(H^n,X^n)\Rightarrow (H,X)$ on $\D_{\mathbb{R}^2}[0,\infty)$. The aforementioned note \cite{fabozzi_comment} warned that \emph{``One is bound to believe that Theorem 2 in Caner (1997) is false''}, but they stopped short of constructing a counterexample. Example \ref{ex:main_counterexample} provides the missing counterexample.

We also note that Example \ref{ex:main_counterexample} yields a counterexample to a sufficient condition for $J_1$ convergence of stochastic integrals presented in \cite{duffie_protter}. In short, \cite[Lem.~4.1]{duffie_protter} says that if $X^n=M^n+A^n$ are special semimartingales for which the $A^n$ are of tight total variation, then also the predictable finite variation parts of the canonical decompositions of the $X^n$ will be of tight total variation. But then suppose $(M^n)_{n\geq 1}$ is any sequence of martingales that is $M_1$ tight and set
\begin{equation}\label{eq:counterex_remove_jumps}
X^n := M^n - J^n,\quad J^n_t:=\sum_{s\, \le \, t} \Delta X^n_s \ind_{\{|\Delta X^n_s|\ge  1\}}.
\end{equation}
These are semimartingales with jumps bounded by $1$, so they are special and have canonical decompositions $X^n=\hat{M}^n+\hat{A}^n$ with the jumps of the local martingale parts $\hat{M}^n $ bounded by $2$ (see e.g.~\cite[Lem.~I.4.41]{shiryaev}). Also, it follows from the $M_1$ tightness and Corollary \ref{cor:A9} that the $J^n$ in \eqref{eq:counterex_remove_jumps} are of tight total variation, so the aforementioned lemma would give that the $\hat{A}^n$ are of tight total variation. But then we have $M^n=\hat{M}^n+(\hat{A}^n+J^n)$, where the local martingales have jumps that are bounded uniformly in $n\geq 1$, and the finite variation terms are of tight total variation, so $M^n$ admits good decompositions \eqref{eq:Mn_An_condition}. Letting $M^n$ be as in Example \ref{ex:main_counterexample} we obtain a contradiction, so \cite[Lem.~4.1]{duffie_protter} is not true in general.

\subsection{$M_1$ versus $J_1$ tightness of local martingales}\label{sect:$M_1$_martingales_and_counterex} 

Later in this work, we shall discuss applications where the integrators are continuous-time random walks. Section \ref{subsect:intro_CTRW} provides a brief introduction to these. In \cite{Becker-Kern,meerschaert}, Becker-Kern, Meerschaert \& Scheffler showed that certain (rescaled) continuous-time random walks with i.i.d.~infinite mean waiting times and infinite variance jumps converge weakly in $M_1$ 
to an $\alpha$-stable L\'evy process time-changed by the generalised inverse of a $\beta$-stable subordinator. Several years later, Henry \& Straka \cite{henry_straka} proved that there is in fact $J_1$ convergence, relying on much more intricate arguments, while the $M_1$ case benefited from advantageous properties of composition in the $M_1$ topology. These findings turn out to be a special instance of a more general principle that we establish in Theorem \ref{cor:3.11} below: as the continuous-time random walks in question are sufficiently well-behaved martingales for $\alpha \in (1,2]$, their $J_1$ tightness is in fact automatic once there is $M_1$ tightness. As the next example illustrates, such a principle cannot be true without conditions.

\begin{example}\label{ex:$M_1$_not_$J_1$_martingales} For $n\geq 1$, let $\xi^n_1$ and $\xi^n_2$ be independent with  $\Pro(\xi^n_1=1)=\Pro(\xi^n_2=1)=1-1/n$ as well as $\Pro(\xi^n_1=-n+1)=\Pro(\xi^n_2=-n+1)=1/n$ on a common probability space $(\Omega , \mathcal{F},\mathbb{P})$. Consider the martingales $X^n:= \xi^n_1 \ind_{[1-2/n,\infty)}+\xi^n_2\ind_{[1-1/n,\infty)}$ in their natural filtrations. In the $M_1$ topology, these processes converge in probability to the limit $X= 2 \ind_{[1,\infty)}$, yet they clearly do \emph{not} converge weakly in the $J_1$ topology.
\end{example}

Notice that, in this example, the expectation of extreme values of the martingales does not clear away in the limit. That is, $\E[|X^n_1|\ind_{\{|X^n_1|\ge K\}}] \to 2$ as $n\to \infty$ for all $K>2$. In fact, the two consecutive jumps of size $1$, which appear with large probability and approximate, in $M_1$, a single limiting jump, are compensated by larger and larger jumps of opposite sign on events of increasingly small probability. Calling for uniform integrability of the sequence rules this out. To be precise, a family $\mathcal{A}:=\{X^i :  i \in I\}$ of processes $X^i$ on probability spaces $(\Omega^i, \mathcal{F}^i, \Pro^i)$ is said to be \emph{uniformly integrable} if, for every $\varepsilon>0$ and $T>0$, there exists $\delta:=\delta(\varepsilon, T)>0$ such that, for all $i\in I$, $0\le t \le T$ and $\Lambda_i \in \mathcal{F}^i$ with $\Pro^i(\Lambda_i)< \delta$, it holds that $\E^i[\, |X^i_t| \, \ind_{\Lambda_i}\, ] < \varepsilon$.

We note that if the elements of the family $\mathcal{A}$ have tight running suprema on all compact time intervals (which in particular holds if $\mathcal{A}$ is $M_1$ tight), then the above definition is equivalent to the following commonly used notion of uniform integrability in probability theory: for every $T>0$, we have $\operatorname{sup}_{i \in I,\, 0\le t \le T} \, \E^i[ |X^i_t| \ind_{\{|X^i_t|\ge K\}}] \to 0$ as $K \to \infty$. 

In order to inhibit the problematic situations described above, it will in fact suffice to have a suitable `localised' form of uniform integrability, as defined next.

\begin{defi}[Localised uniform integrability]\label{defn:local_UI} Consider a family $\mathcal{A}:=\{X^i :  i \in I\}$ on filtered probability spaces $(\Omega^i, \mathcal{F}^i, \mathbbm{F}^i, \Pro^i)$. We say that $\mathcal{A}$ is  \emph{localised uniformly integrable} if there are sequences of $\mathbbm{F}^i$--stopping times $(\rho^i_\ell)_{\ell\ge 1}$ so that, for every $\ell\ge 1$, the family $\mathcal{A}_{\ell}:=\{X^i_{\bullet \wedge \rho^i_\ell} :  i \in I\}$ is uniformly integrable and, for all $T>0$, we have $\operatorname{sup}_{i\in I} \Pro^i(\rho^i_\ell \le T) \to 0$ as $\ell\to \infty$.
\end{defi} 

We then obtain the following structural result for local martingales on Skorokhod space.

\begin{theorem}[$M_1$ vis-à-vis $J_1$ relative compactness for martingales] \label{cor:3.11} 
	Let $\mathcal{A}:=\{X^i :  i \in I\}$ be a family of martingales taking values in $\mathbf{D}_{\mathbb{R}^d}[0,\infty)$ and defined on given filtered probability spaces $(\Omega^i, \mathcal{F}^i, \mathbbm{F}^i, \Pro^i)$. Let also $\mathcal{A}$ be localised uniformly integrable in the sense of Definition \ref{defn:local_UI}. If the family $\mathcal{A}$ is weakly relatively compact in $\mathbf{D}_{\mathbb{R}^d}[0,\infty)$ for the $M_1$ topology, then it must also be weakly relatively compact for the $J_1$ topology.
\end{theorem}

For any sequence $(X^n)_{n\ge 1}$ of local martingales converging weakly on $\D_{\R^d}[0,\infty)$ to some $X$, we can find a sequence of stopping times $(\tau^n)_{n\ge 1}$ such that the $X^n_{\bullet \wedge \tau^n}$ are martingales and such that there is equivalence between $X^n \Rightarrow X$ and $X^n_{\bullet \wedge \tau^n} \Rightarrow X$. Indeed, the local martingale property allows us to find, for each $n\ge 1$, an $\mathbbm{F}^n$-stopping time $\tau^n$ such that $X^n_{\bullet \wedge \tau^n}$ is an $\mathbbm{F}^n$-martingale and it holds $\Pro^n(\tau^n \le n) \le n^{-1}$. The $M_1$ metric on $[0,\infty)$ being fully determined by the $M_1$ distance on compact time intervals (see Definition \ref{def:$M_1$_metric}), one obtains $|\,\E^n[f(X^n)]-\E^n[f(X^n_{\bullet \wedge \tau^n})]\,| \to 0$ as $n\to \infty$ for all bounded, $M_1$-Lipschitz $f$. Hence, we are able to deduce the following result from Theorem \ref{cor:3.11}.

\begin{cor}[$M_1$ vis-\`a-vis $J_1$ convergent local martingales] \label{cor:$M_1$_conv_impl_$J_1$_conv_martingales}
	Let $(X^n)_{n\ge 1}$ be a localised uniformly integrable sequence of local martingales on filtered probability spaces $(\Omega^n, \mathcal{F}^n, \mathbbm{F}^n, \Pro^n)$ converging weakly on $(\mathbf{D}_{\mathbb{R}^d}[0,\infty), \dM)$ to some càdlàg limit $X$. Then, $X^n \Rightarrow X$ on $(\mathbf{D}_{\mathbb{R}^d}[0,\infty), \dJ)$.
\end{cor}

Regarding the works on limit theorems for continuous-time random walks discussed above, we stress that the processes studied therein are indeed localised uniformly integrable martingales (see Proposition \ref{prop:CTRW_are_localised_UI}).

\subsection{Preserving local martingality in the limit} \label{sect:preservation_of_local_martingality}

By Proposition \ref{prop:preserve_semimartin}, if a sequence of local martingales is weakly $M_1$ convergent and satisfies \eqref{eq:Mn_An_condition}, then the limit is at least a semimartingale. While $M_1$ or $J_1$ limits of (local) martingales need not be local martingales in general (see e.g.~Example \ref{ex:$M_1$_not_$J_1$_martingales}), a natural question is how close the jump condition in \eqref{eq:Mn_An_condition} is to ensuring this. As follows from Example \ref{ex:martingales_GD_directly}, the latter is not sufficient. However, we can give a positive answer when strengthening the condition to
\begin{align}
		\lim\limits_{K\to \infty} \, \limsup\limits_{n\to \infty} \, \E^n\Big[ |\Delta M^n_{t\wedge \tau^n_c}| \ind_{\{|\Delta M^n_{t\wedge \tau^n_c}| \,> \, K\}}\Big] \; = \; 0 \label{eq:localised_uniform_integrability_of_martingale_jumps}
	\end{align} 
	for all $t>0$ and $c>0$. A sequence of local martingales satisfying this condition is in particular localised uniformly integrable, but our arguments need the additional control of \eqref{eq:localised_uniform_integrability_of_martingale_jumps}.

	\begin{prop}[Local martingale limits] \label{prop:propagation_of_local_martingale_property}
		Let $(H^n,M^n)$ be defined on probability spaces $(\Omega^n, \mathcal{F}^n, \Pro^n)$ so that each $M^n$ is a local martingale for the natural filtration generated by the pair $(H^n,M^n)$. Suppose that $(H^n,M^n) \Rightarrow (H,X)$ on $(\D_{\R^d}[0,\infty),\dM)^2$. If the $M^n$ satisfy \eqref{eq:localised_uniform_integrability_of_martingale_jumps}, then $X$ is a local martingale for the natural filtration generated by the pair $(H,X)$.
	\end{prop}

Existing results, such as \cite[Props.~IX.1.12 \& IX.1.17]{shiryaev}, are stated for the $J_1$ topology and have as a condition either uniform boundedness of the jumps or full uniform integrability. If there is uniform integrability, then it is immediate that just finite-dimensional convergence suffices to preserve the true martingale property in the limit, but a similar statement cannot be made for local martingales without significant additional conditions.

\section{An alternative in place of the AVCI condition}\label{sect:avci_alternative}

Beyond the sufficient conditions in Proposition \ref{prop:3.3}, checking for \eqref{eq:oscillcond} directly could be a rather intricate task. In this section, we thus provide an alternative criterion which works well in cases where there is a form of `nearly consecutive local independence' of the integrands $H^n$ and integrators $X^n$ along with some regularity of the increments of the integrators.

\subsection{Weak continuity of Itô integrals revisited}\label{subsect:alternative_criteria}

We will say two systems $\mathcal{A}, \mathcal{B}$ of measurable sets are \emph{{\footnotesize $\le$}-independent} if $\Pro(A\cap B) \le \Pro(A) \Pro(B)$ for all $A \in \mathcal{A}$, $B \in \mathcal{B}$, and we then write $\mathcal{A} \indep^{\!\footnotesize \le} \mathcal{B}$. Independence is of course a special case.

\begin{theorem}[Weak continuity of Itô integrals under AVCI alternative] \label{prop:4.36}  
	Let the assumptions of Theorem \ref{thm:3.19} hold except for \eqref{eq:oscillcond}. Suppose the following holds for every $T>0$: for every sequence of $H^n$-stopping times $\{\nu^n\}_{n\ge 1}$ bounded by $T$ there exists a sequence of (random) partitions $\{[\sigma_k^{n}, \sigma_{k+1}^{n}) \, : \, k \ge 1 \}_{n\ge 1}$ of $[0,T)$, where the $\sigma_k^{n}$ are $\mathbbm{F}^n$-stopping times, as well as random variables $Y^{n,k,T}_\delta$ with 
    $$ \operatorname{sup}\,  \{ |X^n_{r\wedge T}  -  X^n_{s}| \; : \; \sigma^{n}_{k} \le  s  <  \sigma^{n}_{k+1}  \le  r  \le  \sigma^{n}_{k+1}   +  \delta \}\; \le \; Y^{n,k,T}_\delta $$ 
    almost surely for all $n,k\ge 1$, $\delta>0$, satisfying that
	\begin{enumerate}[\quad (a)] 
		\item there exists $\tilde \delta>0$ such that for all $n, k \ge 1$ and $0<\delta \le \tilde{\delta}$ it holds that
		$$ \big\{ \, \nu^n \; \in \; [ \sigma^{n}_k, \, \sigma^{n}_{k+1} ) \, \big\} \; \; \indep^{\footnotesize \le} \; \; Y^{n,k,T}_\delta , $$
		\item for every $\lambda>0$ it holds that\,  $\lim_{\delta \downarrow 0} \limsup_{n\to \infty} \operatorname{sup}_{k \ge 1} \; \Pro^n\big(Y^{n,k,T}_\delta > \lambda \big) = 0.$
	\end{enumerate}
	Then, $X$ is a semimartingale in the natural filtration generated by the pair $(H, X)$ and
	\begin{equation}\label{eq:conv_avci_alternative}\Big( \, X^n, \, \int_0^\bullet \, H^n_{s-} \; \diff X^n_s\, \Big) \;  \; \Rightarrow \; \Big( \, X, \, \int_0^\bullet \, H_{s-} \; \diff X_s\, \Big) \quad \text{ in } \quad (\D_{\R^{2d}}[0,T], \, \rho)
    \end{equation}
	where $\rho \in \{\dJ,\dM\}$ depending on the topology in which $X^n$ converges. If $(H^n, X^n)$ converge in $\D_{\R^{2d}}[0,\infty)$ then the above convergence holds on $\D_{\R^{3d}}[0,\infty)$ together with $H^n$.
\end{theorem}
\begin{proof}The proof is given together with that of Theorem \ref{thm:3.19} in Section \ref{sect:proof_main_weak_cont}. Once the approximation property \eqref{eq:approx_simpleint_by_better_simpleint} is established under the alternative conditions of Theorem \ref{prop:4.36} (see Section \ref{subsec:Conseq_of_AVCO}), the proof for both theorems is identical (see Section \ref{sect:proof_main_weak_cont}).
\end{proof}

We stress that one can replace \eqref{eq:oscillcond} in Theorem \ref{thm:2.12} and Proposition \ref{remark:weaker_integrnd_conv} by the conditions of Theorem \ref{prop:4.36} and obtain the same respective results.

\begin{example} 
	Let $(X^n)_{n\ge 1}$ satisfy \eqref{eq:Mn_An_condition} and suppose that, for each $n\ge 1$, $X^n$ has stationary increments such that, for all $0<s\le t$, the increment $X^n_t-X^n_s$ is independent of $\sigma(X^n_r : 0\le r \le s-\alpha_n)$. Suppose also that $(H^n)_{n\ge 1}$ are càdlàg processes adapted to the delayed filtrations $ \{ \sigma(X^n_s : 0\le s \le t-2\alpha_n)\}_{t\ge 0}$. Then, $(H^n,X^n)  \Rightarrow (H,X)$ on $(\D_{\R}[0,\infty)  , \dM )\times (\D_{\R}[0,\infty)  ,  \rho)$ implies the weak convergence \eqref{eq:conv_avci_alternative} of the stochastic integrals, for $\rho \in \{\dJ,\dM\}$. Indeed, letting $\nu^n$ be any $H^n$-stopping time and choosing deterministic $\sigma_k^n:= (k-1)\alpha_n \wedge T$, we have
	$$ \left\{ \, \nu^n \, \in \, [\sigma_k^n, \, \sigma^n_{k+1}) \, \right\} \; \in \; \sigma(X^n_s \, : \, 0 \le s < (k \alpha_n \wedge T)-2\alpha_n)$$
	by the adaptedness of each $H^n$ to the delayed filtration. Clearly, the increments $X^n_r-X^n_s$ are independent of $\sigma(X^n_p \, : \, 0 \le p \le ((k-1) \alpha_n \wedge T)-\alpha_n)$ for all $\sigma_{k}^n \le s \le r \le \sigma^n_{k+1}+\delta$,
	which yields condition (a) of Theorem \ref{prop:4.36}, while (b) follows from the stationary increments and relative compactness of the $X^n$ in $M_1$.
\end{example}

Independence within assumption (a) of Theorem \ref{prop:4.36} should simply be seen as a guarantee for the process $H^n$, until just before time $\sigma^n_{k+1}$, not to depend on the behaviour of the process $X^n$ restarted at $\sigma^n_k$. In order to make Theorem \ref{prop:4.36} more easily applicable, we provide the following sufficient criterion for the condition (a).

\begin{lem}\ \label{lem:4.43} 
	Under the notation of Theorem \ref{prop:4.36}, if for every $n\ge 1$ and $k\ge 1$, the sets 
	$$ (H^n_{t_1},...,H^n_{t_\ell})^{-1}(A_1\times...\times  A_\ell)\; \cap \; \{ q_2 \, < \, \sigma^{n}_{k+1}\}\; \cap \; \{\sigma^n_k \, \le \, q_1\} $$
	with $\ell \ge 1$, $A_1,...,A_\ell \in \B(\R^d)$, $q_1,q_2 \in \Q$ and $t_1,...,t_\ell\le q_2$, are independent of 
	$$ \sigma \big(\{X^n_{\sigma^{n}_{k+1} \, + \, r} - X^n_{(\sigma^{n}_{k+1}-s)\vee \sigma^n_{k}} \; : \; s>0, \, 0\le r \le \tilde\delta \}\big) ,$$
	then condition (a) of Theorem \ref{prop:4.36} holds true.
\end{lem}
\begin{proof}\ 
	We simply note that $\{  \nu^n  <  \sigma^{n}_{k}  \} = \bigcup_{q \in \Q} \, \{ \nu^n \le q\} \cap \{ q < \sigma^n_{k}\}$. Therefore,
	\begin{align*}
		& \{ \, \nu^n \; \in \; [\sigma^n_k, \, \sigma^n_{k+1}) \, \} \; = \; \{\, \nu^n < \sigma^n_{k+1}\} \, \setminus \, \{\, \nu^n < \sigma^n_{k}\} \\
		\; &= \; \union{q_2 \in \Q}{} \intersection{q_1 \in \Q} \, \Big[  \bigl( \{\, \nu^n \in (q_1, q_2]\} \cap \{q_2 < \sigma^n_{k+1}\}\bigr) \cup \,  \bigl( \{\, \nu^n \le q_2\} \cap  \{\sigma^n_{k}\le q_1\} \cap \{q_2 < \sigma^n_{k+1}\}\bigr) \Big]
	\end{align*}
	Clearly, the system of sets of the form $(H^n_{t_1},...,H^n_{t_\ell})^{-1}(A_1\times...\times A_\ell) \cap  \{ q_2 \, < \, \sigma^{n}_{k+1}\} \cap \{\sigma^n_k \, \le \, q_1\}$, $q_1,q_2 \in \Q$, $t_1,...,t_\ell \le q_2$ is stable under intersection and generates a $\sigma$-algebra containing the involved sets above for all $H^n$-stopping times $\nu^n$.
\end{proof}

\subsection{A particular class of integrators and integrands}\label{subsect:concrete_avci_alternative}

We end this section by considering a concrete class of integrators and corresponding integrands which have a structure that makes the above conditions easy to verify. It is motivated in part by the applications involving continuous-time random walks discussed in the next section.

\begin{prop} \ \label{prop:finex}
	Let $X^n_0$, $\xi^n_k$ and $\Delta \sigma^n_j> 0$ be independent random variables and define $\sigma^n_k:=\sum_{j=1}^k \Delta \sigma^n_j$, $\chi^n_k:=(\xi^n_k,\Delta \sigma^n_k)$, as well as pure jump processes 
	$$ X^n_t \; = \; X^n_0\; + \; \sum_{k=1}^\infty \; f_{n,k}(\chi^n_k,\chi^n_{(k-1)\vee 1},...,\chi^n_{(k-J)\vee 1})  \ind_{[\sigma^n_k\, ,\, \infty)}(t)$$
	where $J\ge 0$ and the $f_{n,k}$ are measurable mappings. Further, assume the increments of $X^n$ to be stationary from the jumps, i.e.~$X^n_{\sigma^n_k +\, \bullet}-X^n_{\sigma^n_k} \sim X^n_\bullet$, and $\lim_{\delta \downarrow 0} \limsup_{n\uparrow \infty} \, \Pro^n( |X^n|^*_\delta >\lambda) =0$ for all $\lambda>0$.
    Let $g_{n,i}$ be measurable and define adapted càdlàg pure jump processes by
	\begin{equation}\label{eq:Hn_fin_application} H^n_t \; := \; \sum_{i=1}^{\infty} \, g_{n,i}(t^n_i,X^n_{ \Xi^n(t^n_i)})  \ind_{[t^n_i, \, t^n_{i+1})}(t)
	\end{equation}
	with $t^n_i\uparrow \infty$ as $i \to \infty$ and $\Xi^n(t^n_i):= \operatorname{max}\{\sigma^n_{k-J} \, : \, k\ge 1 \text{ with } k-J\ge 1, \, \sigma^n_k\le t^n_i\}$. Set $\operatorname{max} \emptyset:=0$. Then, $X^n$ and $H^n$ satisfy the assumptions (a) and (b) of Theorem \ref{prop:4.36}. 
\end{prop}
\begin{proof} \ Since (b) follows in a straightforward way from the stationarity and the vanishing of $X^n$ around zero in probability, we only need to verify (a) of Theorem \ref{prop:4.36}. For this, we will apply Lemma \ref{lem:4.43}. Fix $n\ge 1$ and $k\ge 1$, and note that it suffices to show that for any $ \ell \ge 1$ such that $t^n_{\ell},q_1\le q_2$ and $A_1,...,A_\ell \in \B(\R^d)$, we have \\[-2ex]
	$$ \bigcap_{i=1}^\ell \; \big\{ g_{n,i}(t^n_{i},X^n_{\Xi^n(t^n_{i})}) \in A_i\big\} \; \cap \; \{\sigma^n_{k+1}>q_2\} \; \cap \; \{\sigma^n_k\le q_1\} $$
	is independent of  $X^n_{\sigma^n_{k+1} \, + \, \bullet} - X^n_{\sigma^n_k}$. Now, note that on $\{\sigma^n_{k+1}>q_2\}$ it holds $\Xi^n(t_i^n)\in \{0, \sigma^n_1,..., \sigma^n_{k-J}\}$ for all $i=1,...,\ell$ since $t_i^n \le q_2 < \sigma^n_{k+1}$. We can write
	\begin{align*}
		& \big\{ g_{n,i}(t^n_{i},X^n_{\Xi^n(t^n_{i})}) \in A_i\big\} \; = \; \union{\ell=0}{k} \; \Big( \, \big\{ \, \sigma_\ell^n \le t^n_i < \sigma_{\ell+1}^n \, \big\} \; \cap \; \big\{ \, g_{n,i}(t^n_i, X^n_{\sigma^n_{\ell-J}}) \in A_i \, \big\} \, \Big)  
	\end{align*}
	on $\{\sigma^n_{k+1}>q_2\}$, where we have defined $\sigma^n_{j}:=0$ for all $j\le 0$. The involved sets, however, are measurable sets with respect to the $\sigma$-algebra
	$$ \sigma \big( \, X^n_{0^{-}},\, \sigma^n_1,...,\sigma^n_{k+1}, \, \xi^n_{1},..., \xi^n_{k-J}\, \big)$$
	while the process $X^n_{\sigma^n_{k+1} \, + \, \bullet} - X^n_{\sigma^n_k}$ is measurable with respect to $\sigma( \, \Delta \sigma^n_{k+1+j}, \, \xi^n_{k-J+j} :  j\ge 1)$. Hence, we obtain the desired independence.
\end{proof}

\section{Applications driven by continuous-time random walks}\label{sect:CTRW_applications}

The notion of continuous-time random walks (CTRWs) was introduced by Montroll \& Weiss  \cite{Montroll_Weiss} for the mathematical study of anomalous diffusion. It is a type of renewal-reward process whose precise definition is given in Section \ref{subsect:intro_CTRW} below. Related to the present work, \cite{scalas} studies functional weak convergence of stochastic integrals driven by CTRWs. This was motivated by the analysis of damped harmonic oscillators subject to stable L\'evy noise in \cite{Sokolev_Ebeling_Dybiec}. Indeed, stochastic integrals against stable processes are often considered for physical models with heavy tailed behaviour due to stable limit theorems \cite{Eliazar_Klafter} (see also \cite{Kulik_Pavlyukevich} for related work on Langevin equations with stable L\'evy noise). With a view towards the analysis of more general anomalous diffusion models, \cite[Ch.~6]{hahn} highlights the limitations of what is known about weak continuity results on Skorokhod space for stochastic integrals against general CTRW integrators (see in particular \cite[Rem.~6.4]{hahn}). We address this in the next subsections.

\subsection{Continuous-time random walk: definitions and scaling limits}\label{subsect:intro_CTRW}

Let $\{ \theta_k : -\infty < k < \infty \}$ be a sequence of i.i.d.~$\R^d$-valued random variables on a common probability space $(\Omega, \mathcal{F}, \Pro)$ that are in the normal domain of attraction of a non-degenerate $\alpha$-stable random variable $\tilde{\theta}$ with $0< \alpha< 2$ (for $\alpha=2$, $\tilde\theta$ is non-degenerate Gaussian). Given non-negative constants $c_0,c_1,...\ge 0$ such that $\sum_{j=0}^\infty c_j^\rho < \infty$ for some $0<\rho<\alpha$, we define\\[-1ex]
\begin{equation}\label{eq:linear_process_innovations} \zeta_i \; := \; \sum_{j=0}^\infty \; c_j \theta_{i-j}, \qquad i \ge 1.
\end{equation}
For $1<\alpha\le 2$, assume further that either $c_j=0$ for all but finitely many $j\ge 0$ or the sequence $(c_j)_{j\ge 0}$ is monotone and $\sum_{j=0}^\infty c_j^\rho <\infty$ for some $\rho<1$. Now let $J_1, J_2,...$ be positive i.i.d.~random variables in the normal domain of attraction of a $\beta$-stable random variable with $\beta \in (0,1)$, defined on the same probability space as the $\theta_k$.

Let $L(k):= J_1+\hdots+J_k$ give the time of the $k^{\text{th}}$ jump (with $L(0)\equiv 0$). Then $ N(t) :=  \max{} \{ k \ge 0  : L(k) \le t \}$ is the number of jumps by time $t$, and we refer to each process \\[-2ex]
\begin{align} X^n_t \; := \;  \frac{1}{n^{\frac \beta \alpha}} \; \sum_{k=1}^{N(nt)} \, \zeta_k \label{defi:CTRW}
\end{align}
as a (rescaled) \emph{correlated} CTRW, in view of the linear correlation structure \eqref{eq:linear_process_innovations}. Taking $c_0=1$ and $c_j=0$ for $j\geq1$ yields a standard (uncorrelated) CTRW. The CTRWs in \eqref{defi:CTRW} are further said to be \emph{uncoupled} if the sequences $(J_i)_{i\geq 1}$ and $(\zeta_k)_{k\geq 1}$ are independent. In this setting, it is shown in \cite{meerschaert2} that, for $0<\alpha\le 2$, there is weak convergence
\begin{equation}\label{eq:CTRW_$M_1$_conv}
	\; X^n \; \; \stackrel{\text{$M_1$}}{\convw{n\to \infty}} \; \; X=\Bigl(\,\sum_{j=0}^\infty c_j\Bigr)Z_{D^{-1}},
\end{equation}
where $Z$ is an $\alpha$-stable Lévy process ($0<\alpha<2$) or a Brownian motion ($\alpha=2$) with $Z_1\sim \tilde \theta$, and 
$ D^{-1}(t)  := \operatorname{inf}  \{  s\ge 0 \, : \, D(s)  >  t \}$ is the generalised inverse of a $\beta$-stable subordinator $D$.

In the uncorrelated case ($c_0>0$ and $c_j=0$ for $j\geq1$), this was first studied by \cite{Becker-Kern,meerschaert}, and \cite{henry_straka} later showed that \eqref{eq:CTRW_$M_1$_conv} in fact holds in the $J_1$ topology. As observed in \cite{meerschaert2}, \eqref{eq:CTRW_$M_1$_conv} \emph{cannot} be improved to $J_1$ convergence when not only $c_0>0$ but also $c_j>0$ for some $j\ge 1$.

\subsection{Weak continuity results for  continuous-time random walks}\label{subsect:weak_cont_CTRW}

Consider a sequence of uncorrelated, uncoupled CTRWs $X^n$ on the real line. Assuming also the $\theta_k$ are symmetric and $\alpha$-stable, the first main result of \cite{scalas} shows that, for $\alpha \in (1,2]$, $\int_0^\bullet f(s)\diff X^n_s$ converges weakly to $\int_0^\bullet f(s)\diff X_s$ on $(\D_{\R}[0,\infty),\dM)$, where $f$ is a given bounded continuous function \cite[Thm.~4.15]{scalas}. The second main result generalises this statement to $\alpha \in(0,2]$ and the general assumptions of Section \ref{subsect:intro_CTRW} \cite[Thm.~4.20]{scalas}. Unfortunately, a crucial step \cite[Eq.~(62)]{scalas} in the proof of the latter can be seen to be false for $\alpha \in [1,2]$ and the justification is lacking for $\alpha \in (0,1)$. Nevertheless, we shall see that a much more general statement is indeed true.

In the special case $\alpha =2$ (with centered $\theta_k$), where $X$ is a subordinated Brownian motion, \cite[Thm.~6.13]{hahn} confirms that $\int_0^\bullet H^n_{s-}\diff X^n_s$ converges weakly to $\int_0^\bullet H_s \diff X_s$ on $(\D_{\R}[0,\infty),\dJ)$ whenever $(H^n,X^n) \Rightarrow (H,X)$ on $(\D_{\R^2}[0,\infty),\dJ)$. The possible  extension to general integrands and general CTRWs, including the coupled case, is discussed in \cite{scalas, hahn}.

A \emph{coupled} CTRW $X^n$ refers to the case where $(\zeta_k,J_k)_{k\geq 1}$ is an i.i.d.~sequence, but, for each $k\geq 1$, $\zeta_k$ and $J_k$ may be dependent. In the uncorrelated setting, \cite{henry_straka} has shown that
\begin{equation}\label{eq:coupled_CTRW_$J_1$_conv}
	\; X^n \; \; \stackrel{\text{$J_1$}}{\convw{n\to \infty}} \; \; \bigl((Z^-)_{D^{-1}}\bigr)^+,
\end{equation}
where $x^-(t):=x(t-)$ and $x^+(t):=x(t+)$. If the CTRWs are in fact uncoupled, then $Z$ and $D$ are independent L\'evy processes, so they have no common jumps, a.s., and hence the limit simplifies to $Z_{D^{-1}}$ in agreement with \eqref{eq:CTRW_$M_1$_conv}. We can make the following key observation, the proof of which we have deferred to \cite{andreasfabrice} for brevity.

 \begin{theorem}
 [Good decompositions]\label{thm:CTRW_has_GD}
	Let $X^n$ be uncorrelated, possibly coupled, CTRWs, as given by \eqref{defi:CTRW} with $c_0>0$ and $c_j=0$ for all $j\geq 1$. Then $(X^n)_{n\ge 1}$ has good decompositions
\eqref{eq:Mn_An_condition} for its natural filtration. When $\alpha \in (0,1)$, this holds also for correlated CTRWs.
\end{theorem}

With this result, our framework applies. In the uncorrelated case, the CTRWs are $J_1$ convergent \eqref{eq:coupled_CTRW_$J_1$_conv}, so we can in particular conclude that $(X^n,\int_0^\bullet H^n_{s-}\diff X^n_s)$ converges weakly to $(X,\int_0^\bullet H_{s-} \diff X_s)$ on $(\D_{\R^{2d}}[0,\infty),\dJ)$ whenever the $H^n$ satisfy \eqref{it:alt_int_crit1}--\eqref{it:alt_int_crit3} and \eqref{eq:oscillcond}, by Proposition \ref{remark:weaker_integrnd_conv}. This of course holds if $(H^n,X^n) \Rightarrow (H,X)$ on $(\D_{\R^{2d}}[0,\infty),\dJ)$, by Proposition \ref{prop:3.3}. Thus, we refine and vastly generalise the results from  \cite{scalas, hahn} discussed above.

In the correlated case with $\alpha \in (0,1)$, \eqref{eq:CTRW_$M_1$_conv} holds for the $M_1$ topology only. The second part of Theorem \ref{thm:CTRW_has_GD} then tells us that $(X^n,\int_0^\bullet H^n_{s-}\diff X^n_s)$ converges weakly to $(X,\int_0^\bullet H_{s-} \diff X_s)$ on $(\D_{\R^{2d}}[0,\infty),\dM)$ if the $H^n$ satisfy \eqref{it:alt_int_crit1}--\eqref{it:alt_int_crit3} and \eqref{eq:oscillcond}, again by Proposition \ref{remark:weaker_integrnd_conv}. This holds, e.g., if $(H^n,X^n)\Rightarrow (H,X)$ on $(\D_{\R^{d}}[0,\infty),\dM)\times(\D_{\R^{d}}[0,\infty),\dM) $ with $H$ and $X$ having no common jump discontinuities a.s., by (ii) of Proposition \ref{prop:3.3}.

Instead of \eqref{eq:oscillcond}, in the previous paragraph one can alternatively seek to verify (a) and (b) of Theorem \ref{prop:4.36}. Correlated CTRWs $X^n$ with a finite correlation structure (i.e., there exists $J\geq 1$ such that $c_j=0$ for $j\ge J$) meet the requirements of Proposition \ref{prop:finex}, so in this case (a) and (b) of Theorem \ref{prop:4.36} are satisfied for integrands of the form \eqref{eq:Hn_fin_application}.

Finally, we can observe the following, which links back to the discussion in Section \ref{sect:$M_1$_martingales_and_counterex}.

\begin{prop}[Localised uniform integrability] \label{prop:CTRW_are_localised_UI}
    Let $X^n$ be uncorrelated, possibly coupled, CTRWs, as given by \eqref{defi:CTRW} with $c_0>0$ and $c_j=0$ for all $j\geq 1$. For $\alpha \in (1,2]$, the $X^n$ are martingales and $(X_n)_{n\geq 1}$ is localised uniformly integrable in the sense of Definition \ref{defn:local_UI}.
\end{prop}

If, e.g., we only had access to earlier results on scaling limits in the $M_1$ topology, this tells us that Corollary \ref{cor:$M_1$_conv_impl_$J_1$_conv_martingales} applies for $\alpha \in (1,2]$ and so we would obtain $J_1$ convergence.

\subsection{Explosion of integrals for correlated continuous-time random walks}\label{sect:counter}

An early version of \eqref{eq:CTRW_$M_1$_conv} was first proved in \cite{avram} for \emph{moving averages} $X^n$. These can be seen as special cases of correlated CTRWs with $J_i \equiv 1$ for $i\geq 1$. \cite{avram} (and \cite[Thm.~4.7.1]{whitt}) showed that, if $\mathbb{E}[\theta_0] = 0$ when $1 < \alpha < 2$, or if the law of $\theta_0$ is symmetric when $\alpha=1$, then
\begin{equation}\label{eq:mov_av_scaling_limit}
X^n = \frac{1}{n^{\frac 1 \alpha}} \; \sum_{k=1}^{
\lfloor n \bullet \rfloor } \, \zeta_k \; \; \stackrel{\text{$M_1$}}{\convw{n\to \infty}} \; \; \Bigl(\,\sum_{j=0}^\infty c_j\Bigr)Z.
\end{equation}
For $\alpha \in (0,2)$, this cannot be strengthened to $J_1$ if both $c_0>0$ and $c_j>0$ for some $j\geq 1$, while the weak convergence is in fact in $J_1$ when $c_j=0$ for all $j\geq 1$.

In \cite{jakubowski}, Jakubowski singled out the strict $M_1$ convergence in \eqref{eq:mov_av_scaling_limit} to motivate that, when studying the weak convergence of stochastic integrals, \emph{`for some naturally arising integrators the requirement of convergence in the usual Skorohod topology may be too strong'} (see \cite[Ex.~2]{jakubowski}). Whether or not weak continuity results for stochastic integrals actually apply to these integrators, however, is not discussed in that work. In fact, we are not aware of any later works that have addressed this point. The next example shows that, for strictly $M_1$ convergent moving averages and correlated CTRWs, we do \emph{not} have weak convergence in general.

\begin{prop}[Exploding integrals] \label{counterexample:3.4}
	Consider $\alpha \in (1,2)$, and let $X^n$ be given by \eqref{eq:mov_av_scaling_limit} with $\zeta_k = c_0\theta_k + c_1 \theta_{k-1}$, where $c_0, c_1>0$. One can choose centered $\theta_k$ in such a way that there exists a sequence of pure jump càdlàg integrands $H^n$ adapted to the natural filtration generated by the $X^n$ and converging uniformly to zero, almost surely, for which all finite-dimensional distributions of the stochastic integrals $\int_0^\bullet H^n_{s-} \diff X^n_s$ explode as $n\rightarrow \infty$.
\end{prop}

From this result, we can also conclude that: for $\alpha \in (1,2)$, correlated CTRWs in general do \emph{not} admit good decompositions \eqref{eq:Mn_An_condition}. Otherwise, we would have a contradiction to the conclusion of Theorem \ref{thm:3.19}, by virtue of \eqref{eq:mov_av_scaling_limit} and condition (ii) of Proposition \ref{prop:3.3}. Nevertheless, we note that the weak convergence framework introduced in this paper can be extended to yield $M_1$ continuity of stochastic integrals driven by correlated CTRWs under additional assumptions on the integrands; we pursue these extensions in \cite{andreasfabrice}.

\section{Proofs of the results from Section \ref{sect:stoch_int_conv}}\label{sect:proofs_sect3} 
We provide the proofs in the order in which the results were presented.

\subsection{The key consequence of AVCI and its alternative}
\label{subsec:Conseq_of_AVCO}

The central point of \eqref{eq:oscillcond} and the alternative condition given in Theorem \ref{prop:4.36} is the following: it allows us to start from specific stochastic integrals with respect to integrands discretised along path-dependent stopping times and then approximate these by stochastic integrals with respect to integrands discretised along deterministic times. This is a pivotal ingredient in the proof of Theorem \ref{thm:3.19} and is made precise by Proposition \ref{prop:3.22} below, after some definitions.

\begin{defi}\ \label{def:3.20} 
	For a given partition $\rho=\{0=s_0<s_1<...\}$ with $s_n \to \infty$ as $n\to \infty$, we denote by $I_{\rho}: \D_{\R^d}[0, \infty) \to \D_{\R^d}[0,\infty)$ the discretisation mapping defined by 
	$$ I_\rho(z) \; = \; \sum_{i=0}^{\infty} \, z(s_i) \ind_{[s_i,s_{i+1})}.$$
\end{defi}

\begin{defi}\ \label{def:3.21}
	Let $\rho^m=\{0=t^m_0< t^m_1 < ...\}$ be a sequence of deterministic partitions of $[0,\infty)$ with $t^m_i\to \infty$ as $i\to \infty$ such that $\rho^m \subseteq \rho^{m+1}$ and $|\rho^m|:= \operatorname{sup}_{i\ge 1} \,  |t^m_{i+1}-t^m_i| \, \to \, 0$ as $m\to \infty$. For all $m\ge 1$, $\varepsilon>0$, $z \in \D_{\R^d}[0,\infty),$ we define a partition $\rho^{m,\varepsilon}(z)$ given by the times  
	\begin{align*}
		\tau^{m,\varepsilon}_{k,0}(z)  :=  t_{k-1}^m \quad \text{and} \quad 
		\tau^{m,\varepsilon}_{k,i+1}(z)  := \inf{} \{ s \ge \tau^{m,\varepsilon}_{k,i}(z) \; : \; |z(s) - z(\tau^{m,\varepsilon}_{k,i}(z))| \; \ge \; \varepsilon \} \; \wedge t_k^m
	\end{align*}
	for all $k\ge 1$. As $z$ is càdlàg, for all $T>0$, the above partition restricted to $[0,T]$ is finite.
\end{defi}

It is clear from the definition of the $\tau^{m,\varepsilon}_{k,i}$ that we have $|I_{\rho^{m,\varepsilon}(z)}(z) - z |^*_T \le \varepsilon$
for all $z \in \D_{\R^d}[0,\infty)$, $T>0$, $\varepsilon>0$ and $m\ge 1$.

\begin{prop}[Approximating simple integrals by better simple integrals] \label{prop:3.22} 
	Let $(X^n)_{n\ge 1}$ and $(H^n)_{n\ge 1}$ be sequences of càdlàg processes on the same filtered probability spaces $(\Omega^n, \mathcal{F}^n, \F^n, \Pro^n)$ and fix $T>0$. Further, for all $\varepsilon>0$, let $N^T_\varepsilon(H^n)$, $|H^n|^*_T$ and $|X^n|^*_T$ be tight in $\R$, where $N^T_\varepsilon$ is defined as in \eqref{eq:maxnumosc}. If either of the following two holds true
	\begin{enumerate}
		\item $(H^n, X^n)$ satisfy (\ref{eq:oscillcond});
		\item $(H^n, X^n)$ satisfy the assumptions of Theorem \ref{prop:4.36},
	\end{enumerate}
	then, with the notation from Definitions \ref{def:3.20} and \ref{def:3.21}, it holds for all $\gamma,\varepsilon>0$ that
	\begin{align}
		\lim\limits_{m \to \infty} \limsup\limits_{n\to \infty} \,  \Pro^n \Bigl(  \Bigl| \int_0^\bullet  I_{\rho^m}(H^n)_{s-}  \diff X^n_s  -  \int_0^\bullet  I_{\rho^{m,\varepsilon}(H^n)}(H^n)_{s-}  \diff X^n_s \Bigr|_T^* \ge  \gamma \Bigr)  =  0 \label{eq:approx_simpleint_by_better_simpleint}
	\end{align}
\end{prop}
\begin{proof}\ 
	We first show the claim with the assumptions of Theorem \ref{prop:4.36}. Fixing $T>0$ and $\gamma, \varepsilon>0$,
	\begin{align*}
		&\int_0^t \; I_{\rho^m}(H^n)_{s-} \; \diff X^n_s \; - \; \int_0^t \; I_{\rho^{m,\varepsilon}(H^n)}(H^n)_{s-} \; \diff X^n_{s} \\
		& \qquad = \; \sum_{i=1}^{\infty} \, H_{t^m_{i-1}}^n \, \left( X^n_{t^m_{i}\wedge t} -X^n_{t^m_{i-1}\wedge t}\right) \; - \; \sum_{i=1}^{\infty}\sum_{j=0}^{\infty}\, H_{\tau^{m,\varepsilon}_{i,j}}^n \, \left( X^n_{\tau^{m,\varepsilon}_{i,j+1}\wedge t} -X^n_{\tau^{m,\varepsilon}_{i,j}\wedge t}\right) \\
		& \qquad = \;\sum_{i=1}^{\infty}\sum_{j=0}^{\infty} \, \left(H_{\tau^{m,\varepsilon}_{i,0}}^n - H_{\tau^{m,\varepsilon}_{i,j}}^n \right)\, \left( X^n_{\tau^{m,\varepsilon}_{i,j+1}\wedge t} -X^n_{\tau^{m,\varepsilon}_{i,j}\wedge t}\right) \; .
	\end{align*}
	By definition of the $\tau^{m,\varepsilon}_{i,j}:=\tau^{m,\varepsilon}_{i,j}(H^n)$, we have $|\tau^{m,\varepsilon}_{i,j+1} - \tau^{m,\varepsilon}_{i,j}| \le |\rho^m|$ for all $i,j$. Therefore,
	\begin{align*}
		&\Gamma_{m,n} \;:= \;\Pro^n \left( \, \left| \, \int_0^\bullet \; I_{\rho^m}(H^n)_{s-} \; \diff X^n_s \; - \; \int_0^\bullet \; I_{\rho^{m,\varepsilon}(H^n)}(H^n)_{s-} \; \diff X^n_s\, \right|_T^* \; \ge \; \gamma\,  \right) \\
		& \qquad \le \; \Pro^n \Big( \, \left| H^n\right|_T^*\; \,\sum_{i,\, j \, \ge \, 1 }\; \; \, \sup{0\, < \, z \, \le \, |\rho^m|} \; \left| X^n_{(\tau^{m,\varepsilon}_{i,j}+z)\wedge T} -X^n_{\tau^{m,\varepsilon}_{i,j}\wedge T}\right| \; \ind_{\{\tau^{m,\varepsilon}_{i,j+1} \neq \; \, \tau^{m,\varepsilon}_{i,j}\}}\; \ge \; \frac{\gamma}{2}\,  \Big) \; . 
	\end{align*}
	The number of summands in the above sum---by definition of the $\tau^{m,\varepsilon}_{i,j}$---is bounded by $N^T_\varepsilon(H^n)$ as defined in (\ref{eq:maxnumosc}). Now, let $\eta>0$. By tightness, there exists $K_\eta>0$ such that $\operatorname{sup}_{n\ge 1} \Pro^n ( \left| H^n\right|_T^*  \ge  K_\eta ) \le \eta$ \, and \, $\operatorname{sup}_{n\ge 1}\Pro^n ( N^T_\varepsilon(H^n) \ge K_\eta)\le \eta$. Let us denote $\nu^{n}_1:= \tau^{m,\varepsilon}_{1,1}$ and $\nu^{n}_{k}:= \operatorname{min}\{ \tau^{m,\varepsilon}_{i,j} \, : \, \tau^{m,\varepsilon}_{i,j} > \nu^{n}_{k-1}, \, i\ge 1, \, j\ge 1\}$ for all $k\ge 2$. Then, we obtain
	\begin{align}
		\Gamma_{m,n} \; \le \; 2\eta \; + \; \sum_{k=1}^{K_\eta} \; \Pro^n \biggl( \, \sup{0\, < \, z \, \le \, |\rho^m|} \; \left| X^n_{(\nu^{n}_k+z)\wedge T} -X^n_{\nu^{n}_k \wedge T}\right| \; \ge \; \frac{\gamma}{4\, K_\eta} \,  \biggr) \; . \label{eq:equation_in_Prop_3.22}
	\end{align}
	Hence, it suffices to show that each of the summands on the right side converges conveniently to zero. Since the $\nu^{n}_k$ are $H^n$-stopping times, using assumption (a), we obtain for each $k=1,...,K_\eta$,
	\begin{align*}
		&\Pro^n \Bigl( \, \sup{0\, < \, z \, \le \, |\rho^m|} \; \left| X^n_{(\nu^{n}_k+z)\wedge T} -X^n_{\nu^{n}_k \wedge T}\right| \; \ge \; \frac{\gamma}{4\, K_\eta} \,  \Bigr) \\
		& \; \le \;\sum_{i=1}^\infty \; \Pro^n \left( \, \nu_k^{n} \; \in \; [ \sigma^n_i \, , \, \sigma^n_{i+1} ) \, \right ) \; \Pro^n \biggl( \, \sup{\sigma^n_{i}\, \le \, z\, < \, \sigma^n_{i+1} \, \le \, \tilde{z} \, \le \, \sigma^n_{i+1} \,  + \, |\rho^m|}  \left| X^n_{\tilde{z} \wedge T} -X^n_{z }\right| \; \ge \; \frac{\gamma}{4\, K_\eta}\,  \biggr) \\
		& \; \le \;\sup{i\ge 1} \;\;  \Pro^n \biggl( \, \sup{\sigma^n_{i}\, \le \, z\, < \, \sigma^n_{i+1} \, \le \, \tilde{z} \, \le \, \sigma^n_{i+1} \,  + \, |\rho^m|} \left| X^n_{\tilde{z} \wedge T} -X^n_{z }\right| \; \ge \; \frac{\gamma}{4\, K_\eta}\,  \biggr).
	\end{align*}
	Thus, by assumption (b) this gives us $\lim_{m \to \infty} \,\limsup_{n\to \infty} \, \Gamma_{m,n} \le 2\eta$
	due to $|\rho^m|\to 0$ as $m\to \infty$. Since $\eta$ was chosen arbitrarily, this yields the claim. For a proof of the statement under (\ref{eq:oscillcond}) we refer to \cite{jakubowski}.
\end{proof}


\subsection{Weak continuity of stochastic integrals}\label{sect:stoch_int_conv_proof}

\subsubsection{Proof of Proposition \ref{prop:preserve_semimartin}}
\label{subsec:Proof_Semimart_Prop}

The proof follows similarly to Jakubowski, Mémin, \& Pagès \cite[Thm.~2.1, Lem.~1.3, \& Lem.~1.1]{jakubowskimeminpages}, only we rely on \eqref{eq:Mn_An_condition} in the end. For completeness we give the details.
\begin{proof}[Proof of Proposition \ref{prop:preserve_semimartin}]
	We only show the first part of the proposition. Let $(\mathcal{{F}}_{t})_{t\geq0}$ denote the natural filtration generated by $X$, and let $G$ be an arbitrary simple predictable process as in \cite[Ch.~2, Sec.~1]{protter}. Using a simple approximation argument for the stopping times used in this definition, it is not hard to see that we can assume them to be deterministic times. Fix $t>0$, $\theta,\delta>0$, and assume $|G|^*_t\le \delta$. Then, we can find continuous functions $f^{m,p}_{i}$ such that $\lim_{m \to \infty} \lim_{p\to \infty} G^{m,p}(X) = G$ almost surely with \\[-1ex]
	$$ G^{m,p}(X)\; = \; G_{0}\ind_{\{0\}}\, + \,\sum_{i=1}^{k}\, f_{i}^{m,p}\bigl(X_{r^m_{i,\, 1}}\, , \, \ldots \, , \, X_{r_{ i,\,\ell(i,m)}^m}\bigr)\, \ind_{(t_{i}\, ,\, t_{i+1}]} $$
	where $|f_i^{m,p}|^*_t \le \delta$. By right-continuity of  $X$ and the fact that $D\subseteq [0,\infty)$ is dense, without loss of generality we may assume that all of the $r^{m}_{i,u}$ and the $t_i$ belong to $D$. Consequently, the convergence of finite-dimensional distributions on $D$ and the continuous mapping theorem yield weak convergence of the simple integrals $\int_{0}^{t} G^{m,p}(X^n)\, \diff X^n \Rightarrow \int_{0}^{t} G^{m,p}(X)\, \diff X$ as $n\to \infty$ for fixed $m,p\ge 1$. Thus, by Portmanteau's theorem and the above approximation, we have\\[-1ex]
	\[
	\Pro \Bigl(\, \Bigl|\int_{0}^{t}\, G(X)_s \, \diff X_s \Bigr|\, >\, \theta\Bigr) \; \leq \; \liminf_{m\rightarrow\infty}\;  \liminf_{p \to \infty}\; \liminf_{n\rightarrow\infty}\; \Pro^n\Bigl(\, \Bigl|\int_{0}^{t}\, G^{m,p}(X^n)_s\, \diff X^{n}_s\Bigr|\, >\, \theta \Bigr)
	\]
	with $| G^{m,p}(X^n)|^*_t\leq \delta$ and the $G^{m,p}(X^n)$ being adapted to the filtration of the $X^n$.
	Now, using the good decompositions \eqref{eq:Mn_An_condition} of the $X^n$, we can write\\[-1ex]
	\[
	\mathbb{{P}}^n\Bigl( \, \Bigl|\int_{0}^{t}\, G^{m,p}(X^n)_s\diff X^{n}_s\Bigr| > \theta\Bigr) \leq  \mathbb{{P}}^n\Bigl(\, \Bigl|\int_{0}^{t}\, G^{m,p}(X^n)_s \diff M^{n}_s\Bigr| > \tfrac{\theta}{2} \Bigr) + \mathbb{{P}}^n\Bigl(\delta\, \text{TV}_{[0,t]}(A^{n})>  \tfrac{\theta}{2} \Bigr).
	\]
	Similarly to the considerations with respect to \eqref{eq:T4} in the proof of Theorem \ref{thm:3.19}, \eqref{eq:Mn_An_condition} ensures that the right-hand side tends to zero as $\delta \to 0$.
\end{proof}

\subsubsection{Proof of Theorem \ref{thm:3.19} and Theorem \ref{prop:4.36}}\label{sect:proof_main_weak_cont}
Before we proceed to prove the Theorems \ref{thm:3.19} and \ref{prop:4.36}, we briefly collect three useful observations for local martingales that follow more or less immediately from Lenglart's inequality.
\begin{lem}[Variations of Lenglart's inequality] \label{prop:Lenglart}
	There exist $c,C>0$ such that, for any local martingale $M$ with $M_0 =0$, the following holds for all $T>0$ and $\beta,\gamma,\delta>0$:
	\begin{enumerate}[(a)]
		\item if $\tilde \tau$ is any stopping time such that $\operatorname{inf}\{s: [M]_s\ge \gamma\} \wedge \operatorname{inf}\{s: |M|^*_s\ge \beta \} \, \le \, \tilde \tau $ and $\tilde \tau\, \le \, \operatorname{inf}\{s: [M]_s\ge \gamma\}$, then
		$$\Pro(\, |M|^*_T \, \ge \, \beta\, ) \; \le \; \frac{C(\sqrt{\gamma} \, +\,   \, \E[\, |\Delta M_{T\wedge \tilde \tau}|\, ])}{\beta} \; + \; \Pro(\, [M]_T \, \ge \, \gamma\, )\; ;$$
		\item if $\sigma$ is any stopping time such that $ \operatorname{inf}\{s: [M]^{1/2}_s\ge \beta \} \wedge \operatorname{inf}\{s: |M|^*_s\ge \gamma \}\, \le \, \sigma$ and $ \sigma \le \, \operatorname{inf}\{s: |M|^*_s\ge \gamma \}$, then
		$$\Pro(\, [M]^{1/2}_T \, \ge \, \beta\, ) \; \le \; \frac{c(\gamma \, +\,  \, \E[\, |\Delta M_{T\wedge \sigma}|\, ])}{\beta} \; + \; \Pro(\, |M|^*_T \, \ge \, \gamma\, )\; ;$$
		\item for any $\varepsilon>0$, if $H$ is such that $|H|_T^* \le \varepsilon$ and the Itô integral $Y:=\int_0^\bullet H_{s} \diff M_s$ exists and is a local martingale, then
		$$\Pro(\, |Y|^*_T \, \ge \, \beta\, ) \; \le \; \frac{C \varepsilon ( \sqrt{\gamma} \, +\, \E[\, |\Delta M_{T\wedge \tau}|\, ])}{\beta} \; + \; \Pro(\, [M]_T \, \ge \, \gamma \, ) \; + \; \Pro(\, |M|^*_T \, \ge \, \delta\, )$$
		where $\tau$ is any stopping time such that $\operatorname{inf}\{s: [M]_s\ge \gamma\} \wedge \operatorname{inf}\{s: |M|^*_s\ge \delta\} \, \le \, \tau$ and $\tau \, \le \, \operatorname{inf}\{s: [M]_s\ge \gamma\}$.
	\end{enumerate}
\end{lem}
\begin{proof}
	The proof proceeds as in \cite[Lem.~I.3.30]{shiryaev} with the $L$-domination property satisfied due to the classical Burkholder--Davis--Gundy inequality.
\end{proof}

We can now prove Theorems \ref{thm:3.19} and \ref{prop:4.36}. We only consider the case in which the pairs of integrands and integrators converge weakly on the product space $(\D_{\R^d}[0,\infty), \, \dM)\times (\D_{\R^d}[0,\infty), \, \dM)$, as in \eqref{eq:conv_together}. The statements on convergence together on the strong $M_1$ space, as in \eqref{eq:concerted_stoch_int_conv}, follow analogously (it will be clear from the proof where and how things should be changed). 

\begin{proof}[Proof of Theorem \ref{thm:3.19} and Theorem \ref{prop:4.36}] 
	According to Proposition \ref{prop:preserve_semimartin}, $X$ is a semimartingale with respect to the natural filtration of the limiting pair $(X,H)$. Denote by $\rho^m=\{0=t^m_0< t^m_1 < ...\}$ a sequence of deterministic partitions of $[0,\infty)$ such that $\rho^m \subseteq \rho^{m+1}$ and with $|\rho^m|:= \operatorname{sup}_{i\ge 1} \,  |t^m_{i+1}-t^m_i| \, \to \, 0$ as $m\to \infty$.  Without loss of generality, we may assume that $t^m_i  \notin  \operatorname{Disc}_\Pro(H,X) $
	for all $m,i\ge 1$, where $\operatorname{Disc}_\Pro(H,X)$ is defined just after \eqref{eq:maxnumosc}. Indeed, this is possible since $\operatorname{Disc}_\Pro(H,X)$ is well-known to be at most countable. Let the partition operator $\rho^{m,\varepsilon}$ and its stopping times $\tau^{m,\varepsilon}_{k,i}$ be defined as in Definition \ref{def:3.21}. Using the notation from Definition \ref{def:3.20}, set $H^{\, \vert \, m} := I_{\rho^m}(H)$ and denote $H^{n\, \vert \,  m} := I_{\rho^m}(H^n)$ as well as $H^{n \, \vert \,  m,\varepsilon} := I_{\rho^{m,\varepsilon}(H^n)}(H^n)$ for $n\ge 1$. Again, note that $|I_{\rho^{m,\varepsilon}(H^n)}(H^n) - H^n |^*_T  \le \varepsilon$ for all $\varepsilon>0$, $T>0$ and $m,n\ge 1$. Let $f: \D_{\R^{2d}}[0,\infty) \to \R$ be bounded and Lipschitz continuous with respect to the $M_1$ topology and therefore also with respect to the weighted uniform norm. Let us denote its Lipschitz constant by $\norm{f}_{\operatorname{Lip}}$. We are going to show that 
	\begin{align}
		\E\Big[ f \big(\, X, \, \int_0^\bullet \, H_{s-} \; \diff X_s \, \big)\Big] \; - \;\E^n\Big[ f \big(\, X^n, \, \int_0^\bullet \, H^n_{s-} \; \diff X^n_s \, \big)\Big] \; \; \conv{\phantom{n \to \infty}}{}\; \; 0 \label{eq:4.8}
	\end{align}
	as $n\to \infty$. To achieve the convergence in (\ref{eq:4.8}), we rewrite the left-hand side of (\ref{eq:4.8}) as 
	\begin{align*}
		&\E\Big[ f \big(\, X, \,\int_0^\bullet \, H_{s-} \; \diff X_s \, \big)\Big] \; - \; \E\Big[ f \big(\, X, \, \int_0^\bullet \, H^{\, \vert \, m}_{s-} \; \diff X_s \, \big)\Big] \label{eq:T1} \tag{T1}\\
		& \qquad + \; \E\Big[ f \big(\, X, \, \int_0^\bullet \, H^{\, \vert \, m}_{s-} \; \diff X_s \,  \big)\Big] \; - \; \E^n\Big[ f \big(\, X^n, \,\int_0^\bullet \, H^{n \, \vert \, m}_{s-} \; \diff X^n_s \, \big)\Big] \label{eq:T2} \tag{T2}\\
		& \qquad + \;\E^n\Big[ f \big(\, X^n, \,\int_0^\bullet \, H^{n\, \vert \, m}_{s-} \; \diff X^n_s \, \big)\Big] \; - \; \E^n\Big[ f \big(\, X^n, \, \int_0^\bullet \, H^{n \, \vert \, m,\varepsilon}_{s-} \; \diff X^n_s \, \big)\Big] \label{eq:T3}\tag{T3}\\
		& \qquad + \;\E^n\Big[ f \big(\, X^n, \, \int_0^\bullet \, H^{n \, \vert \, m,\varepsilon}_{s-} \; \diff X^n_s \, \big)\Big] \; - \; \E^n\Big[ f \big(\, X^n, \, \int_0^\bullet \, H^{n}_{s-} \; \diff X^n_s \, \big)\Big] \label{eq:T4} \tag{T4}
	\end{align*}
	and we show that each of the four differences tends to zero conveniently as $m,\varepsilon, n\to 0$. Let $\theta>0$, choose $T>0$ such that $\norm{f}_{\operatorname{Lip}} \int_T^\infty \e^{-t} \diff t \le \theta/6$ and denote $\xi_T:= 1- \e^{-T}$. Recall the standard definition of $\dM$ on $\D_{\R^d}[0,\infty)$, given in \eqref{eq:$M_1$_metric_real_line} in Appendix \ref{app:J1_M1_tops}.
    
    \eqref{eq:T1}: For the first term, we observe that due to the Lipschitz continuity of $f$, Tonelli's theorem and the fact that $\dM^{[0,T]}(x,y)\le |x-y|^*_T$ , it holds 
		\begin{align*}
			|\eqref{eq:T1}| \; &\le \; \frac \theta 6 \; \; + \; \; \norm{f}_{\operatorname{Lip}} \, \xi_T \; \;  \E\Big[ \; \big| \, \int_0^\bullet \, (H_{s-} - H^{\, \vert \, m}_{s-}) \; \diff X_s \, \big|^*_T \; \wedge \; 1 \; \Big]  .
		\end{align*}
		Since $X$ is a semimartingale and $|\rho^m|\to 0$, by the Riemann-Stieltjes approximation theorem for Itô integration with respect to càdlàg semimartingale integrators, we have that 
		$$ \int_0^\bullet \, H^{\,\vert \, m}_{s-} \; \diff X_s \; = \; \sum_{i=0}^{\infty} \; H_{t_{i}^m} \, \big( X_{t_{i+1}^m \, \wedge\,  \bullet}-X_{t_i^m \, \wedge \, \bullet}\big) \; \; \conv{\phantom{m\to \infty}}{} \; \; \int_0^\bullet \, H_{s-} \; \diff X_s $$ 
		uniformly on compacts in probability as $m\to \infty$. In particular, this means that we can find $m_1 \ge 1$ such that, for all $m\ge m_1$, it holds that $ |\eqref{eq:T1}| \; \le \; \theta/3$.

        \eqref{eq:T4}: We are going to make use of the good decompositions \eqref{eq:Mn_An_condition}, $X^n=M^n+A^n$. Let us first define $\tau^n_\delta:=\operatorname{inf}\{s: |M^n|^*_s\ge \delta \}$, $\sigma^n_\gamma:=\operatorname{inf}\{s: [M^n]_s\ge \gamma\}$ and $\Gamma(\rho):=\E^n[\, |\Delta M^n_{T\wedge \rho}|\, ]$. Recall that since the decompositions $X^n=M^n+A^n$ are good, the $M^{n}$ are local martingales satisfying that for all $\delta>0$ it holds $\limsup_{n\to \infty} \Gamma(\tau^n_{\delta})< \infty$ and the $A^{n}$ are processes of locally finite variation with tight total variation over compacts. Clearly, by convergence in the Skorokhod space it holds that $(|X^n|_T^*)_{n\ge 1}$ is tight and due to the tightness of the total variation of the $A^n$, also $(|A^n|_T^*)_{n\ge 1}$ is tight, which together---by the decomposition $X^n=M^n+A^n$---implies that $(|M^n|_T^*)_{n\ge 1}$ is tight. Denote $Y^{n\, \vert \, m, \varepsilon}:= \int_0^\bullet (H^n_{s-}-H^{n\, \vert \, m, \varepsilon}_{s-}) \diff M^n_s$ and recall $|H^n-H^{n\, \vert \, m, \varepsilon}|\le \varepsilon$. Applying first (c) of Lenglart's inequality  \ref{prop:Lenglart}, we obtain
		\begin{align*}
			&\E^n \Big[ \, \big| \, Y^{n\, \vert \, m, \varepsilon} \, \big|^*_T \; \wedge \; 1 \, \Big] \; \le \; \beta \; + \; \Pro^n \big( |Y^{n\, \vert \, m, \varepsilon}|^*_T \; > \; \beta \big) \\
			\;\;&\le \; \beta \; + \; \frac{C\varepsilon (\sqrt{\gamma} \, + \, \Gamma(\sigma^n_{\gamma}\wedge \tau^n_\delta \,))}{\beta} \; + \; \Pro^n(\, [M^n]_T \, \ge \, \gamma \, ) \;+\; \Pro^n(|M^n|^*_T \, \ge \, \delta)
		\end{align*}
		for all $\beta,\gamma,\delta >0$. Another application of Lenglart's inequality Lemma \eqref{prop:Lenglart}(b) gives
		\begin{align*}\E^n \Big[ \, \big| \, Y^{n\, \vert \, m, \varepsilon} \, \big|^*_T \; \wedge \; 1 \, \Big] \; \le \; \beta \; &+
			  \frac{C\varepsilon(\sqrt{ \gamma} \, + \, \Gamma(\sigma^n_{\gamma}\wedge \tau^n_\delta))}{\beta} \; + \; \frac{c(\delta \, + \, \Gamma(\tau^n_{\delta }))}{\sqrt{\gamma}} \; \\
			&\;\;\; + \; 2\Pro^n(|M^n|^*_T \, \ge \, \delta)
			\end{align*}
		for all $\beta, \gamma, \delta>0$. Now, note that $\Gamma(\sigma^n_{\gamma}\wedge \tau^n_\delta)\le \operatorname{max}\{2\delta,\Gamma(\tau^n_\delta)\}$ and hence
		$$ \limsup\limits_{n\to \infty} \, \sup{m \, \ge\, 1} \; \E^n \Big[ \, \big| \, \int_0^\bullet (H^n_{s-}-H^{n\, \vert \, m, \varepsilon}_{s-}) \diff M^n_s\, \big|^*_T \; \wedge \; 1 \, \Big] \; \longrightarrow \; 0$$
		as $\varepsilon \to 0$, due to the tightness of $(|M^n|^*_T)_{n\ge 1}$ and the fact that $\beta, \delta,\gamma$ were arbitrary (first choose $\beta$ small enough and $\delta$ large enough to make the first and last term small, then increase $\gamma$ to downsize the third term and finally let $\varepsilon \to 0$).
		Furthermore,  by basic properties of Lebesgue-Stieltjes integration, it is true that 
		$$ \Big| \, \int_0^\bullet \, (H_{s-}^n - H^{n \, \vert \, m,\varepsilon}_{s-}) \; \diff A^{n}_s \, \Big|^*_{T} \; \le \; \varepsilon \, \operatorname{TV}_{[0,T]}(A^{n})$$
		and therefore, by tightness of the total variation of the $A^n$ over compacts,
		$$ \E^n\Big[ \; \big| \, \int_0^\bullet \, (H_{s-}^n - H^{n \, \vert \, m,\varepsilon}_{s-}) \; \diff A^{n}_s \, \big|^*_{T}  \; \wedge \; 1 \, \Big]  \; \le \;  \beta  +  \Pro^n\big( \; \operatorname{TV}_{[0,T]}(A^{n}) \, > \, \beta / \varepsilon  \; \big) \; \longrightarrow  \;  \; \beta $$
		as $\varepsilon \to 0$ and hence, since $\beta$ was arbitrary,
		$$ \sup{n,m\, \ge \, 1}\; \E^n\Big[ \; \big| \, \int_0^\bullet \, (H_{s-}^n - H^{n \, \vert \, m,\varepsilon}_{s-}) \; \diff A^{n}_s \, \big|^*_{T}  \; \wedge \; 1 \, \Big]   \; \longrightarrow\;  \; 0 \; .$$
		Thus, there exists $\varepsilon_{4}>0$ such that for all $0 < \varepsilon \le \varepsilon_{4}$, $m\ge 1$ and $n\ge 1$ large enough it holds
		$$ \norm{f}_{\operatorname{Lip}} \, \xi_T \; \; \E^n\Big[ \; \big| \, \int_0^\bullet \, (H_{s-}^n - H^{n \, \vert \, m,\varepsilon}_{s-}) \; \diff X^n_s \, \big|^*_T \; \wedge \; 1 \;  \Big]  \; \le \; \frac \theta {6}. $$
		Hence, similarly as in the first step of the consideration of \eqref{eq:T1}, we obtain $|\eqref{eq:T4}| \le \theta/3$ for all $0<\varepsilon\le\varepsilon_4$, $m\ge 1$ and $n\ge 1$ large enough.
        
        \eqref{eq:T3}: Again, as before, we have 
		\begin{align*}
			|\eqref{eq:T3}| \; \le \; \frac \theta 6 \; \; + \; \; \norm{f}_{\operatorname{Lip}} \, \xi_T \; \;  \E^n\Big[ \; \big| \, \int_0^\bullet \, (H^{n \, \vert \, m}_{s-} - H^{n \, \vert \, m,\varepsilon}_{s-}) \; \diff X^n_s \, \big|^*_T \; \wedge \; 1 \; \Big].
		\end{align*}
		We recall that $H^{n\, \vert \, m}:= I_{\rho^m}(H^n)$ and $H^{n \, \vert \, m, \varepsilon}:= I_{\rho^{m,\varepsilon}(H^n)}(H^n)$ and due to the functional tightness of $H^n, X^n$ (which implies the assumptions of Proposition \ref{prop:3.22}, see Corollary \ref{cor:A9}), we can apply Proposition \ref{prop:3.22}. Hence, there exists $m_3\ge 1$ such that for all $m\ge m_3$ and all $\varepsilon>0$ it holds $\limsup_{n\to \infty}  |\eqref{eq:T3}|  \le \theta/3$.
		
	    \eqref{eq:T2}: In a last step, we seek to control \eqref{eq:T2}. First, fix $\varepsilon$ such that $0< \varepsilon < \varepsilon_4$ and $m\ge \operatorname{max}\{m_1,m_3\}$. We recall, that according to the previous two points, it then holds that $\limsup_{n\to \infty}  [ \, | \eqref{eq:T1}|  + |\eqref{eq:T3}|  +  |\eqref{eq:T4}| \, ]  \le \theta$. Hence, it only remains to show that for the fixed $m$, we have $|\eqref{eq:T2}| \to 0$ as $n\to \infty$. By Skorokhod's representation theorem, there exists a probability space $(\tilde\Omega, \tilde{\mathcal{F}}, \tilde\Pro)$ and random vectors $(\tilde{H}^n, \tilde{X}^n)$ and $(\tilde H,\tilde X)$ on this common probability space, mapping into $\D_{\R^d}[0,\infty)\times \D_{\R^d}[0,\infty)$, such that $(\tilde{H}^n, \tilde{X}^n) \; \to (\tilde{H}, \tilde{X})$ almost surely as $n\to \infty$ on $(\D_{\R^d}[0,\infty), \tilde \rho)\times (\D_{\R^d}[0,\infty), \rho)$ and such that the law of $(\tilde{H}^n, \tilde{X}^n)$ coincides with the law of $(H^n, X^n)$ for all $n\ge 1$ as well as the law of $(\tilde{H}, \tilde{X})$ coincides with the law of $(H,X)$. Since $\Pro (\Delta (H_{t^m_i},X_{t^m_i})  \neq (0,0) \, )  =  0$ 
		for all $i\ge 0$ by choice of the $t^m_i$, we also have
		$$ \tilde\Pro \, \Big( \, \union{i=0}{\infty} \{ \Delta (\tilde{H}_{t^m_i},\tilde{X}_{t^m_i}) \, \neq \, (0,0)\} \, \Big) \; =  \; 0 $$
		due to the identical distribution of the Skorokhod representations. Hence, there is almost sure convergence as above such that the limits do not have discontinuities at  $t^m_0,t^m_1,...$ .\\
		\phantom{mi} For simplicity of notation, we will denote $(\tilde\Omega, \tilde{\mathcal{F}}, \tilde\Pro)$ by $(\Omega, \mathcal{F}, \Pro)$, $(\tilde{H}^n, \tilde{X}^n)$ by $(H^n, X^n)$ and $(\tilde H, \tilde X)$ by $(H,X)$ as there is no risk of confusion at this point. Due to the fact that $t^m_0,t^m_1,...$ are almost surely continuity points of $H$, it is well known that the projection maps $\pi_{t^m_i}: \D_{\R^d}[0,\infty) \to \D_{\R^d}[0,\infty)$ with $\pi_{t^m_i}(z)=z(t^m_i)$ are continuous. Hence, $(H^n_{t^m_0},..., H^n_{t^m_{k}})  \to  (H_{t^m_0},..., H_{t^m_{k}})$
		almost surely in $\R^{k}$ as $n\to \infty$ for all $k\ge 1$, and thus for all $s>0$,
		$$ H^{n \, \vert \, m}_{ -}|_{[0,s]} \; = \; \sum_{i=0}^{\infty} \, H^n_{t^m_i} \; \ind_{(t^m_i \wedge s, \, t^m_{i+1}\wedge s]} \; \; \conv{n\to \infty}{} \; \; \sum_{i=0}^{\infty} \, H_{t^m_i} \; \ind_{(t^m_i\wedge s, \, t^m_{i+1}\wedge s]}\; = \;  H^{\, \vert \, m}_{ -}|_{[0,s]} $$
		almost surely with respect to the uniform topology (and therefore also in the $J_1$ and $M_1$ topology), since there are only finitely many $t^m_i \le s$. In particular, the sequence $H_{-}^{n\, \vert \, m}|_{[0,s]}$, $n\ge 1$, almost surely has a uniformly bounded number of discontinuities. In addition, we know that $\operatorname{Disc}_\Pro(X)=\{q>0\,: \, \Pro(\Delta X_q \neq 0)> 0\}$ is countable and for $s \notin \operatorname{Disc}_\Pro(X)$ it holds $X^n|_{[0,s]} \to X|_{[0,s]}$ almost surely. Further, $H_{\bullet\, -}^{\, \vert \, m}$ and $X$ almost surely have no common discontinuities (as $X$ almost surely is continuous at $t^m_0,t^m_{1},...$). Proposition \ref{thm:ContinuityOfSimpleIntegralsInM_1Topology} on simple integral convergence for $\rho=\dM$ (or \cite[Equations (1.12)-(1.13)]{kurtzprotter} for $\rho=\dJ$) therefore gives us that 
		$$ \Big( \, X^n|_{[0,s]}, \, \int_0^{\bullet} \, H^{n\, \vert \, m}_{ r- } \; \diff X^n_r\big|_{[0,s]}\, \Big) \; \; \conv{n\to \infty}{} \; \;  \Big( \, X|_{[0,s]}, \,\int_0^{\bullet} \, H^{\, \vert \, m}_{ r- } \; \diff X_r\big|_{[0,s]}\, \Big)$$
		for all $s>0$ in a co-countable subset of $[0,\infty)$ almost surely on $(\D_{\R^{2d}}[0,s], \, \rho)$. By construction of the $M_1$ metric for the time interval $[0,\infty)$, this directly implies that we also have $(  X^n ,\, \int_0^{\bullet} \, H^{n\, \vert \, m}_{ r- } \, \diff X^n_r )\to  (  X, \,\int_0^{\bullet} \, H^{\, \vert \, m}_{ r- } \, \diff X_r )$ almost surely on $(\D_{\R^{2d}}[0,\infty), \, \dM)$.
		Thus, by dominated convergence and the continuity of $f$ for the $M_1$ topology, $|\eqref{eq:T2}| \to 0$ as $n\to \infty$.
        
	Consequently, for all $0<\varepsilon<\varepsilon_4$ and $m\ge \operatorname{max}\{m_1,m_3\}$ it holds $\limsup_{n\to \infty} [|\eqref{eq:T1}|+ |\eqref{eq:T2}| + |\eqref{eq:T3}| +|\eqref{eq:T4}| ] \le \theta$. As $\theta>0$ was chosen arbitrarily, this gives us \eqref{eq:4.8}.
\end{proof}

\subsubsection{Proof of Proposition \ref{prop:3.3}}
\label{subsec:Suff_Cond_AVCO}

\begin{proof}[Proof of Proposition \ref{prop:3.3}] \ 
	The first part of the proposition follows directly from the standard $J_1$ tightness criteria using the $J_1$ modulus of continuity. Therefore it only remains to establish the second part: let $\gamma>0$. Due to Skorokhod's representation theorem, without loss of generality, we may assume that $(H^n, X^n)\to (H,X)$ almost surely
	in $(\D_{\R^d}[0,\infty), \, \dM)^2$ as $n\to \infty$, with all quantities defined on a common probability space $(\Omega, \mathcal{F}, \Pro)$. By continuous mapping, we also have $(H^n|_{[0,T]}, \, X^n|_{[0,T]}) \to (H|_{[0,T]}, \, X|_{[0,T]})$ almost surely in $(\D_{\R^d}[0,T], \dM)^2$ as $n\to \infty$ for all $T$ in some co-countable set $D\subseteq (0,\infty)$. If we show that (\ref{eq:oscillcond}) holds for all $T \in D$, then we are done by monotonicity of the consecutive increment function $\hat w$ with respect to the interval on which it is defined. Let $T \in D$. To lighten the notation, we will write $H^n, H, X^n, X$ instead of $H^n|_{[0,T]}, \, H|_{[0,T]}, \, X^n|_{[0,T]}, \, X|_{[0,T]}$. Clearly, $H$ and $X$ almost surely have no common discontinuities, i.e. there exists $A \in \mathcal{F}$ with $\Pro(A)=1$ such that $$\operatorname{Disc}(H(\omega)) \; \cap \; \operatorname{Disc}(X(\omega)) \; = \; \emptyset $$
	for all $\omega \in A$. This holds true since the Skorokhod representation of the limit and the original limit coincide in law and the set of all $(x,y)$ such that $\operatorname{Disc}(x) \cap \operatorname{Disc}(y)\neq \emptyset$ is a measurable set of $\mathscr{B}(\D_{\R^d}[0,T]) \otimes \mathscr{B}(\D_{\R^d}[0,T])$. Fix $\omega \in A$. Since $t\mapsto H_t(\omega)$, $t\mapsto X_t(\omega)$ are càdlàg, there exist $0= t_1^H<...<t^H_{k_H}= T$ and $0= t_1^X<...<t^X_{k_X}= T$ (which depend on $\omega$) such that $\{t_2^H,...,t^H_{k_H-1}\} \cap \{t_2^X,...,t^X_{k_X-1}\}= \emptyset$ and
	\begin{align}
		&\sup{\substack{s,t \, \in \, [t_i^H\, , \, t_{i+1}^H) \\ i=1,...,k_H-1}} \; |H_s(\omega) \, -\, H_t(\omega)| < \frac \gamma 2  \qquad \text{and} \quad \sup{\substack{s,t \, \in \, [t_i^X\, , \, t_{i+1}^X) \\ i=1,...,k_X-1}} \; |X_s(\omega) \, -\, X_t(\omega)|  < \frac \gamma 2 . \label{eq:3.12} 
	\end{align}
	Choose
	$$ \lambda \; := \; \lambda(\omega) \; < \; \frac14 \; \min{} \left\{ \, |t-s| \; : \; t,s \in \{t^H_1,...,t^H_{k_H}, t^X_1,...,t^X_{k_X}\},\; t\neq s \, \right\}$$
	the scaled shortest time difference between two distinct $t_i^H, t_j^X$ and let $N:=N(\omega) \ge 1$ such that for all $n\ge N$, we have
	\begin{align} \dM\left(H^n(\omega), H(\omega)\right) \; < \; \frac \gamma 4 \; \wedge \; \lambda \qquad \text{ and } \qquad \dM\left(X^n(\omega), X(\omega)\right) \; < \; \frac \gamma 4 \; \wedge \; \lambda  \,. \label{eq:3.11}\end{align}
	Assume for the sake of a contradiction that there exists an $n\ge N$ and $s,t,v$ with $0\le s <t < v\le (s+\lambda)\wedge T$ such that
	\begin{align}
		\hat{w}^T_{\lambda}(H^n(\omega), X^n(\omega)) \; \ge \; |H^n_s(\omega) \, - \, H^n_t(\omega)| \; \wedge \;  |X^n_t(\omega) \, - \, X^n_v(\omega)| \; > \; \gamma \, .\label{eq:3.10}
	\end{align}
	We will show that this leads to a contradiction. Once we have established this contradiction, it implies that the set 
	$$  \intersection{\ell = 1}{\infty} \; \intersection{k =1}{\infty} \; \union{n=k}{\infty} \; \{ \hat{w}^T_{\lambda_\ell}(H^n, X^n) \; > \; \gamma\} \; \subseteq \; \Omega \setminus A$$
	has zero probability for any sequence $\lambda_\ell \downarrow 0$ as $\ell \to \infty$. Then, using the continuity of the probability measure and the definition of $\hat{w}$, we are able to write
	\begin{align*}
		0 \; = \; \Pro \biggl( \; \intersection{\ell = 1}{\infty} \; \intersection{k =1}{\infty} \; \union{n=k}{\infty} \; \{ \hat{w}^T_{\lambda_\ell}(H^n, X^n) \; > \; \gamma\}  \, \biggr) \; \ge \; \lim\limits_{\ell \to \infty} \; \limsup\limits_{n\to \infty} \;\Pro \Bigl( \, \hat{w}^T_{\lambda_\ell}(H^n, X^n) \; > \; \gamma   \, \Bigr)
	\end{align*}
	and we deduce $\lim_{\lambda \downarrow 0} \,\limsup_{n\to \infty} \,\Pro ( \hat{w}^T_{\lambda}(H^n, X^n)  >  \gamma )  = 0 $. Since $\gamma>0$ was chosen arbitrarily, this then holds for all $\gamma>0$. Thus, it only remains to show that, for the fixed $\gamma>0$, $\omega \in A$, $\lambda>0$ and $n\ge N$, we have that (\ref{eq:3.10}) yields a contradiction: due to (\ref{eq:3.11}), we can find parametric representations $(u^{X^n}, r^{X^n}) \in \Pi(X^n(\omega))$ and $(u^X, r^X) \in \Pi(X(\omega))$ such that 
	$|u^{X^n} - u^X|^*_1  \vee   |r^{X^n} - r^X|^*_1  <  \gamma /4  \wedge \lambda$.
	Furthermore, there exist $z_t, z_v \in [0,1]$ such that
	$$ u^{X^n}(z_t)=X^n_t(\omega), \; r^{X^n}(z_t)=t \qquad \text{ and } \qquad u^{X^n}(z_v)=X^n_v(\omega), \; r^{X^n}(z_v)=v \,.$$
	Therefore, in particular, $r^X(z_t), r^X(z_v) \in (t-\lambda, v+ \lambda)\cap [0,T] \subseteq (s-\lambda, v+\lambda) \cap [0,T]$ and 
	$|u^X(z_t) - X^n_t(\omega) |  <   \gamma /4$ and $|u^X(z_v) - X^n_v(\omega) | < \gamma / 4$ implying, together with \eqref{eq:3.10}, that 
	\begin{align*}
		\gamma \; < \; |X^n_t(\omega) - X^n_v(\omega)| \; &\le \; |X^n_t(\omega) - u^X(z_t)|+ |X^n_v(\omega) - u^X(z_v)| + |u^X(z_v) - u^X(z_v)|\\
		& < \; \frac \gamma 4+ \frac \gamma 4 + |u^X(z_t) - u^X(z_v)|
	\end{align*}
	and thus $|u^X(z_t) - u^X(z_v)|  >  \frac \gamma 2$. Clearly, $(u^X(z_t),r^X(z_t))$ and $(u^X(z_v),r^X(z_v))$ just lie on the completed graph of $X(\omega)$. Still, since $X(\omega)$ is càdlàg, this means that we can find $p_1, p_2 \in (s-\lambda, v+ \lambda)\cap [0,T)$ with $p_1<p_2$ (and therefore $|p_1-p_2|< 3\lambda$) so that $|X_{p_1}(\omega) - X_{p_2}(\omega)|  >  \frac \gamma 2$.
	According to (\ref{eq:3.12}) there exists a $t^X_i$ such that
	$$ 0 \vee (s- \lambda) \; \le \; p_1 \; < \; t^X_i \; \le \; p_2 \; \le \; (v+ \lambda) \wedge T$$
	and there cannot be another $t_j^H$ or $t_i^X$ between $0\vee(s-\lambda)$ and $(v+\lambda) \wedge T$ due to the choice of $\lambda$ and the fact that $v-s\le \lambda$. By repeating this procedure for $H^n$, however, we obtain the existence of a $t_j^H$ between $0\vee(s-\lambda)$ and $(v+\lambda) \wedge T$. This means that there are $t_i^X$ and $t_j^H$ lying somewhere in the interval $[0\vee (s-\lambda), (v+\lambda)\wedge T]$ yielding a contradiction since the interval length is $(v+\lambda)-(s-\lambda)=(v-s)+2\lambda \le 3\lambda<4\lambda$ but by definition $|t_i^X-t_j^H|\ge 4\lambda$.
\end{proof}

\subsubsection{Proof of Corollary \ref{cor:2.14}}
\label{subsec:Proof_of_Cor_3.16}

\begin{proof}[Proof of Corollary \ref{cor:2.14}]
	Denote the integral processes by $Y^n:= \int_0^\bullet H^n_s \diff X^n_s$. Since $[Y^n] \Rightarrow [Y]$ on $(\D_{\R^d}[0,\infty), \dM)$ by \eqref{eq:interplay_cond} and $[Y]$ is almost surely continuous, it is well-known that in fact $[Y^n] \Rightarrow [Y]$ on $(\D_{\R^d}[0,\infty), |\cdot|^*_\infty)$ which in turn by Prokhorov's theorem and relative compactness in this space implies in particular that for all $T>0$, $\eta >0$,
	\begin{align} \lim\limits_{\theta \downarrow 0} \; \; \sup{n\ge 1}\; \Pro^n \Big( \, \sup{0\, \le \, s \, \le \, t \, \le \, (s+\theta) \wedge T} \; [Y^n]_t- [Y^n]_s \; > \; \eta \, \Big) \; = \; 0 . \label{eq:ModOfCont_QuadrVar}\end{align}
	Note that for any two $\mathbbm{F}^n$-stopping times $\rho^n, \sigma^n $ which are bounded by $T>0$ and for which holds $\rho^n\le \sigma^n \le \rho^n +\theta$, $\theta>0$, we have for any $\varepsilon, \eta>0$,
	\begin{align*}
		\Pro^n \big( |Y^n_{\rho^n} - Y^n_{\sigma^n}|\, \ge \, \eta \big) \; &\le \; \Pro^n \big( |Y^n_{\bullet \wedge \sigma^n} - Y^n_{\bullet \wedge \rho^n}|^*_T \, \ge \, \eta \big) \\ &\le \;  \frac{C(\sqrt{\varepsilon} + 2\tilde C_n)}{\eta}\; + \; \Pro^n \big( \, [Y^n]_{\sigma^n} - [Y^n]_{\rho^n} \, \ge \, \varepsilon \big)
	\end{align*}
	by Lemma \ref{prop:Lenglart}(a), where $\tilde{C}_n=\E[\, |\Delta (Y^n_{T\wedge \tau^n_\varepsilon})|\,]$ and $\limsup_{n\to \infty} \tilde{C}_n \to 0$ according to \eqref{eq:vanishing_jumps_condition}. Fix $\eta>0$,  making use of \eqref{eq:ModOfCont_QuadrVar}, we therefore obtain 
	$$\lim\limits_{\theta \downarrow 0} \; \limsup\limits_{n\to \infty} \; \sup{\rho^n, \sigma^n} \Pro^n \big( |Y^n_{\sigma^n} - Y^n_{\rho^n}| \, > \, \eta \big) \; = \; 0$$
	where the supremum runs over all $\mathbbm{F}^n$-stopping times $\rho^n\le \sigma^n \le \rho^n+\theta$ bounded by $T$. According to Aldous' criterion \cite[Thm.~VI.4.5]{shiryaev} this yields tightness of $(Y^n)_{n\ge 1}$ on the $J_1$ Skorokhod space (note that the stochastic boundedness on compacts follows similarly from Lemma \ref{prop:Lenglart} and the tightness of the $[Y^n]$). Now, consider any subsequence of $(Y^n)_{n\ge 1}$ and choose a further subsequence---for simplicity of notation also denoted by $(Y^n)_{n\ge 1}$---which converges to some process $Z$ on $(\D_{\R^d}[0,\infty), \dJ)$. The weak limit $X$ of the $X^n$ being continuously supported, we deduce that $(X^n, Y^n) \Rightarrow (X, Z)$ on $(\D_{\R^{2d}}[0,\infty), \dJ)$. Since \eqref{eq:vanishing_jumps_condition} implies \eqref{eq:Mn_An_condition} of the $Y^n$, according to the (ii) of Proposition \ref{prop:3.3} and Theorem \ref{thm:3.19} it holds that $(X^n,Y^n,\int Y^n_{s-} \diff Y^n_{s}) \Rightarrow  (X,Z,\int Z_{s-} \diff Z_{s})$. By the continuous mapping theorem and the fact that the map $(\alpha,\beta,\gamma)\mapsto (\alpha, \beta, \beta^2-\beta^2(0)-2\gamma)$ is continuous on the strong $J_1$ space, we derive $(X^n,Y^n, [Y^n]) \Rightarrow (X,Z, [Z])$ on $(\D_{\R^{2d}}[0,\infty), \dJ)$ from the definition $[Y^n]=(Y^n)^2-(Y^n_0)^2 - 2\int Y^n_{s-} \diff Y^n_s$. Then, uniqueness of weak limits implies $[Z]=[Y]$. We assume $Z$ and $X$ to be defined on the same filtered probability space (otherwise move to the product space). By bilinearity and the above equality,\\[-2ex]
	$$ [Z-Y] =  2[Z] - 2[Z,Y] \; = \; 2[Y] \; - \; 2 \int_0^\bullet H_s \, \diff [Z, X]_s$$
	and if $[Z, X]=[Y, X]$, then this implies $[Z-Y]=0$. The $(X^n, Y^n)$ converging together on the strong space, a similar argument gives $[Y^n, X^n]  \Rightarrow  [Z,X]$,
	while at the same time we have\\[-1ex]
	$$ [Y^n, X^n] \; = \; \int_0^\bullet  H^n_s \diff [X^n]_s \; \Rightarrow \; \int_0^\bullet  H_s \diff [X]_s \; = \; [Y, X]$$
	by \eqref{eq:interplay_cond}. Again from uniqueness of weak limits, this yields $[Z,X]=[Y,X]$ and therefore, as described above, $[Z-Y]=0$. Finally, we observe that the process $Z$ is a local martingale. Indeed, first choose a sequence of $\mathbbm{F}^n$-stopping times $\sigma^n$ such that the processes $Y^n_{\bullet \, \wedge\,  \sigma^n}$ are martingales and $Y^n_{\bullet \, \wedge\,  \sigma^n } \Rightarrow Z$. Defining the stopping times $\tau_a(X):= \operatorname{inf} \{ t \ge 0 : |X_t| \ge a \text{ or } |X_{t-}|\ge a\}$, the processes $Y^n_{\bullet \, \wedge \, \sigma^n \, \wedge \,\tau_a(Y^n_{\bullet \, \wedge \, \sigma^n})}$ are still martingales and it holds $Y^n_{\bullet \, \wedge\, \sigma^n\, \wedge \, \tau_a(Y^n_{\bullet \wedge \sigma^n})} \Rightarrow Z_{\bullet \, \wedge \, \tau_a(Z)}$ for all $a$ in some well-chosen co-countable dense subset $D \subseteq [0,\infty)$ (due to continuity properties of this very stopping action, see e.g.~\cite[Prop.~VI.2.12]{shiryaev}). Proceeding similarly to the proof of \cite[Prop.~IX.1.12]{shiryaev}---using \eqref{eq:vanishing_jumps_condition} to be able to approximate the expectation of the processes through the expectation of their truncated alterations---guarantees that $Z_{\bullet \, \wedge \, \tau_a(Z)}$ is a martingale, and hence $Z$ is a local martingale. Thus, by \cite[Prop.~I.4.50d]{shiryaev}, $Z=Y$ follows.
\end{proof}

\subsubsection{Proof of Proposition \ref{thm:ContinuityOfSimpleIntegralsInM_1Topology} }\label{sec:6.5}

To prove Proposition \ref{thm:ContinuityOfSimpleIntegralsInM_1Topology}, we first need an auxiliary result. It will help us to construct convenient parametric representations of the simple integrals. For the definition of parametric representations, as they are employed in the definition of the $M_1$ metric, see Appendix \ref{app:J1_M1_tops}.

\begin{prop}[Parametric representation of simple integral] \label{prop:ParameterisationOfSimpleIntegrals} 
	Let $x \in \D_\R[0,T]$ be such that $x(\bullet \,-)$ is a simple left-continuous path as in \eqref{eq:ClassicalLeftContinuousPaths}. Let $y \in \D_\R[0,T]$. Given any parametric representation $(u,r) \in \Pi(x,y)$ with $u=(u^{(1)}, u^{(2)})$, we define \\[-4ex]
	\begin{align}
		\tilde{u}(z) &\vcentcolon= \sum_{i=1}^k u^{(1)}(\ubar{z}_{i+1})  \Big[u^{(2)}(z \wedge \bar{z}_{i+1}) - u^{(2)}(z \wedge \bar{z}_{i}))\Big] \label{eq:specific_form_param_representation_simple_deterministic}
	\end{align}
	where $\bar{z}_i \vcentcolon= \sup{} \{\bar z \in [0,\,1] : r(\bar z) = t_i\}$ and $\ubar{z}_i \vcentcolon= \inf{} \{\bar z \in [0,\,1] : r(\bar z) = t_i\}$. Then, it holds that $\big((u^{(1)},u^{(2)}, \tilde{u}),\, r\big) \, \in \, \Pi\big(x, \,y, \, \int_0^\bullet \, x(s \, -) \, \diff y(s)\big)$.
\end{prop}
\begin{proof}
	Let $f$ denote the simple integral of $x(\bullet \,-)$ with respect to $y$. Note that $\operatorname{Disc}(f) \subseteq \operatorname{Disc}(y)$. We begin the proof by showing that $((u^{(1)}, u^{(2)}, \tilde{u}),r)$ maps onto the completed graph $\Gamma_{(x,\, y,\,f)}$ (see Appendix \ref{app:J1_M1_tops} for a definition of the completed graph).
	Let first $s \in \operatorname{Disc}(x)^\complement \cap \operatorname{Disc}(y)^\complement$, then there exists $z \in [0,\,1]$ such that $((u^{(1)},u^{(2)}),r) (z) = ((x(s),y(s)),s)$. In particular, there exists $j \in \{1,...,k\}$ such that $s \in (t_j, t_{j+1})$ and by definition $z \in (\Bar{z}_j,\, \ubar{z}_{j+1})$. Then,\\[-2ex]
	\[
		 \tilde{u}(z)=\sum_{i=1}^j u^{(1)}(\ubar{z}_{i+1})  \left[u^{(2)}(z \wedge \bar{z}_{i+1}) - u^{(2)}(z \wedge \bar{z}_{i})\right] = \sum_{i=1}^j x(t_i) \left[y(s \wedge t_{i+1}) - y(s \wedge t_{i})\right] = f(s).
	\]
	Let now $s \in \operatorname{Disc}(x) \cap \operatorname{Disc}(y)$, implying $s = t_{j+1}$ for some $j \in \{1,\ldots,k\}$. In addition, let $\alpha \in [0,1]$ and choose $z \in [0,1]$ such that $r(z) = t_{j+1}$ and $(u^{(1)},u^{(2)})(z) = \alpha (x,y)(t_{j+1}-) + (1 - \alpha) (x,y)(t_{j+1})$. By definition, we have that $z \in [\ubar{z}_{j+1},\, \bar{z}_{j+1}]$. Therefore, we deduce
	\begin{align*}
		\tilde{u}(z) \; &= \; \sum_{i=1}^{j-1} u^{(1)}(\ubar{z}_{i+1}) \left[u^{(2)}(\bar{z}_{i+1}) - u^{(2)}( \bar{z}_{i})\right] \; + \; u^{(1)}(\ubar{z}_{j+1}) \left[u^{(2)}(z) - u^{(2)}( \bar{z}_{j})\right] \\
		&= \; \sum_{i=1}^{j-1} x(t_i)  \left[y(t_{i+1}) - y(t_{i})\right] \; + \;  x(t_j)  \left[\alpha y(t_{j+1}-) + (1 - \alpha) y(t_{j+1}) - y(t_{j})\right] \\
		&= \alpha  \Big(\sum_{i=1}^{j-1} x(t_i)  \left[y(t_{i+1}) - y(t_{i})\right] + x(t_j) \left[y(t_{j+1}-) - y(t_{j})\right]\Big) \\
		& \quad \, + (1-\alpha) \Big(\sum_{i=1}^{j-1} x(t_i)  \left[y(t_{i+1}) - y(t_{i})\right] + x(t_j)  \left[y(t_{j+1}) - y(t_{j})\right]\Big) \\
		&= \alpha f(t_{j+1}-) + (1-\alpha)f(t_{j+1})\, .
	\end{align*}
	The cases $s \in \operatorname{Disc}(x)^\complement \cap \operatorname{Disc}(y)$ and $s \in \operatorname{Disc}(x) \cap \operatorname{Disc}(y)^\complement$ follow analogously. Hence, $((u^{(1)}, u^{(2)},\tilde{u}),r): [0,1] \to \Gamma_{(x,y,f)}$ is surjective. Clearly, the continuity of $((u^{(1)}, u^{(2)},\tilde{u}),r)$ can easily be derived from the continuity of $((u^{(1)}, u^{(2)}), r)$. 
    Finally, $((u^{(1)}, u^{(2)},\tilde{u}),r)$ being non-decreasing with respect to the ordering on $\Gamma_{(x,y,f)}$ follows directly from the fact that $((u^{(1)},u^{(2)}),r) \in \Pi(x,y)$ is non-decreasing, and that $((u^{(1)},u^{(2)},\tilde{u}),r)$ is onto.
\end{proof}

We can now  prove the pathwise convergence result for simple It{\^o} integrals.
\begin{proof}[Proof of Proposition \ref{thm:ContinuityOfSimpleIntegralsInM_1Topology}]
	We only consider $d=1$. The general case follows analogously by first extending Proposition \ref{prop:ParameterisationOfSimpleIntegrals} in a straightforward way to a multidimensional case and then proceeding similarly to the one-dimensional proof. Let $\varepsilon > 0$. Since $x$ and $y$ have no common discontinuities, it is well-known that $(x_n,y_n) \xrightarrow[]{} (x,y)$ in $(\D_{\R^2}[0,T], \, \dM)$ (see e.g.~\cite[Thm.~12.6.1]{whitt}). Hence, there is a sequence of parametric representations $(u_n,r_n)=((u^{(1)}_n, u^{(2)}_n), r_n) \in \Pi(x_n,y_n)$ and $(u,r)=((u^{(1)}, u^{(2)}), r) \in \Pi(x,y)$ such that $|u_n-u|^*_1\vee |r_n-r|^*_1 \to 0$. Thus, there is an $N \in \N$ so that $|u_n-u|_1^* \; \leq \; \varepsilon [4\tilde{K}(|u^{(1)}|^*_1 + |u^{(2)}|^*_1 + 1)]^{-1}$ and $|u_n^{(1)}|^*_1 \leq |u^{(1)}|^*_1 + 1$ for every $n \geq N$. Employing Proposition \ref{prop:ParameterisationOfSimpleIntegrals}, we know that $((u^{(1)}, u^{(2)},\tilde{u}),r)$ and $((u^{(1)}_n, u^{(2)}_n,\tilde{u}_n),r_n)$ are parametrisations of $(x,y,\int_0^\bullet x(s-)\diff y(s))$ and $(x_n,y_n,\int_0^\bullet x_n(s-)\diff y_n(s))$ respectively, taking the specific form \eqref{eq:specific_form_param_representation_simple_deterministic}. We denote $t^n_1,...,t^n_{k_n}$ the discontinuity times of $x_n$, and $\ubar{z}^n_{i}, \bar{z}^n_{i}, \tilde{u}_n, u_n^{(1)},u_n^{(2)}$ the specific quantities from \eqref{eq:specific_form_param_representation_simple_deterministic} associated to $(x_n,y_n)$. Analogously, we omit the index $n$ to denote the respective quantities for $(x,y)$.

We only need to show $ |\tilde{u}_n-\tilde{u}|^*_1 \to 0$ as $n\to \infty$ in order to deduce $(x_n, y_n, \int_0^\bullet x_n(s-)\diff y_n(s))$ $\to (x, y, \int_0^\bullet x(s-)\diff y(s))$ in $(\D_{\R^3}[0,T], \dM)$. We will do this by induction over intervals to each of the jumps of $x$. Let $z \in [0,\,\bar{z}_2]$. We recall that $u_1 \equiv u^{(1)}(\ubar{z}_2)$ on $[0,\, \ubar{z}_2] $ and $r(\ubar{z}_2) = t_2$. Furthermore, by definition $t_1 = 0 = \bar{z}_1 = \bar{z}^n_1 $ and $u^{(2)}(\ubar{z}_1) = u^{(2)}(\bar{z}_1)$ since $x$ and $y$ do not jump at the same time by assumption. Now, we can write $|\tilde{u}(z) - \tilde{u}_{n}(z)|$ as
	$$ \Big|u^{(1)}(\ubar{z}_2)\left[u^{(2)}(z)-u^{(2)}(\bar{z}_1)\right] - \sum_{i=1}^{k(2,n)} u_n^{(1)}(\ubar{z}^n_{i+1})  \Big[u_n^{(2)}(z \wedge \bar{z}^n_{i+1}) - u_n^{(2)}(z \wedge \bar{z}^n_{i})\Big]\Big|$$
	where $t_1^n, \ldots, t_{k(2,n)}^n$ are jump times of $x_n$ with $k(2,n)$ being minimal such that $\Bar{z}_2  \leq  \bar{z}_{k(2,n)}^n$. We rewrite the first summand $\tilde{u}(z)$ using a telescoping sum to obtain that $\tilde{u}(z) - \tilde{u}_n(z)$ equals 
	\begin{align*}
		\sum_{i=1}^{k(2,n)-1} u^{(1)}(\ubar{z}_{2}) \Big[u^{(2)}(z \wedge \bar{z}^n_{i+1}) - u^{(2)}(z \wedge \bar{z}^n_{i})\Big] &- \sum_{i=1}^{k(2,n)-1}  u_n^{(1)}(\ubar{z}^n_{i+1})\Big[u_n^{(2)}(z \wedge \bar{z}^n_{i+1}) - u_n^{(2)}(z \wedge \bar{z}^n_{i})\Big] \\ &+  u^{(1)}(\ubar{z}_{2}) \Big[u^{(2)}(z \wedge \bar{z}^n_{1}) - u^{(2)}(z \wedge \bar{z}_{1})\Big] .
	\end{align*}
	The last term of this sum equals zero as $\bar{z}_1 = \bar{z}^n_1=0$. Hence, using the fundamental inequality $|ab-a_n b_n| \leq |a_n||b_n - b| + |b| |a_n - a|$ we obtain
	\begin{align*}
		|\tilde{u}(z) - \tilde{u}_n(z)| \; &\leq \; \sum_{i=1}^{k(2,n)-1} |u_n^{(1)}(\ubar{z}^n_{i+1})| \cdot \Bigl|u_n^{(2)}(z \wedge \bar{z}^n_{i+1}) - u_n^{(2)}(z \wedge \bar{z}^n_{i}) - \\[-10pt]
		&\hspace{200pt} u^{(2)}(z \wedge \bar{z}^n_{i+1}) + u^{(2)}(z \wedge \bar{z}^n_{i})\Bigr| \\[-6pt]
		&\qquad + \sum_{i=1}^{k(2,n)-1} \Big|u^{(2)}(z \wedge \bar{z}^n_{i+1}) - u^{(2)}(z \wedge \bar{z}^n_{i})\Big| \cdot |u_n^{(1)}(\ubar{z}^n_{i+1}) - u^{(1)}(\ubar{z}_{2})| \\
		&\leq \; 2  (k(2,n) - 1) \; (\,|u^{(1)}|^*_1 + |u^{(2)}|^*_1 + 1) \; |u_n - u|^*_1 \;\leq \; \varepsilon \;\; \text{for all } n\geq N.
	\end{align*}
    Towards the inductive step, let $\ell \in \{1, \ldots, k\}$ and assume that, for any $z \in [0,\, \bar{z}_\ell]$ and $n \geq N$,\\[-3ex]
	$$ |\tilde{u}(z) - \tilde{u}_n(z)| \; \leq \; 2 C (k(\ell,n)+\ell) \, |u - u_n|^*_1 \; \leq \; 4 \tilde{K} C |u - u_n|^*_1\; < \; \varepsilon,
	$$
	where $C := |u^{(1)}|^*_1 + |u^{(2)}|_1^* + 1$ and $k(\ell,n)$ is minimal such that $\bar{z}_\ell \leq \bar{z}^n_{k(\ell,n)}$. We will show that this bound also holds for every $z \in [0,\bar{z}_{\ell+1}]$. Let $z \in (\bar{z}_{\ell},\,\bar{z}_{\ell+1}]$ and note $\tilde{u}(z) - \tilde{u}_n(z)$ equals
	\begin{align*}
		&\sum_{i=1}^{k} u^{(1)}(\ubar{z}_{i+i}) \Big[u^{(2)}(z \wedge \bar{z}_{i+1}) - u^{(2)}(z \wedge \bar{z}_{i})\Big]  - \sum_{i=1}^{k_n} u_n^{(1)}(\ubar{z}^n_{i+1}) \left[u_n^{(2)}(z \wedge \bar{z}^n_{i+1}) - u_n^{(2)}(z \wedge \bar{z}^n_{i})\right] \\
		&= \,  \sum_{i=1}^{\ell} u^{(1)}(\ubar{z}_{i+i})  \Big[u^{(2)}(z \wedge \bar{z}_{i+1}) - u^{(2)}(z \wedge \bar{z}_{i})\Big]  - \sum_{i=1}^{k(\ell+1,n) -1} \hspace{-2ex}u_n^{(1)}(\ubar{z}^n_{i+1})  \Big[u_n^{(2)}(z \wedge \bar{z}^n_{i+1}) - u_n^{(2)}(z \wedge \bar{z}^n_{i})\Big]
	\end{align*}
	where we recall that $k_n$ is the total number of jumps of $x_n$ on $[0,\,T]$. Dividing these sums into parts gives us that $\tilde{u}(z)  - \tilde{u}_n(z)$ equals
	\begin{align*}
		&\sum_{i=1}^{\ell-1} u^{(1)}(\ubar{z}_{i+1})  \Big[u^{(2)}(z \wedge \bar{z}_{i+1}) - u^{(2)}(z \wedge \bar{z}_{i})\Big] - \sum_{i=1}^{k(\ell,n) -2} u_n^{(1)}(\ubar{z}^n_{i+1}) \Big[u_n^{(2)}(z \wedge \bar{z}^n_{i+1}) - u_n^{(2)}(z \wedge \bar{z}^n_{i})\Big]\\    
		& - u_n^{(1)}(\ubar{z}^n_{k(\ell,n)})  \Big[u_n^{(2)}(z \wedge \bar{z}_{\ell}) - u_n^{(2)}(z \wedge \bar{z}^n_{k(\ell,n)-1})\Big] +  u^{(1)}(\ubar{z}_{\ell+1})  \Big[u^{(2)}(z) - u^{(2)}(z \wedge \bar{z}_{\ell})\Big]  \\
		&- \sum_{i=k(\ell,n)}^{k(\ell+1,n) -1} \hspace{-2ex} u_n^{(1)}(\ubar{z}^n_{i+1})  \Big[u_n^{(2)}(z \wedge \bar{z}^n_{i+1}) - u_n^{(2)}(z \wedge \bar{z}^n_{i})\Big] -   u_n^{(1)}(\ubar{z}^n_{k(\ell,n)})  \Big[u_n^{(2)}(z \wedge \bar{z}^n_{k(\ell,n)}) - u_n^{(2)}(z \wedge \bar{z}_{\ell})\Big].\notag
	\end{align*}
	Since $z > \bar{z}_\ell$ and $\ubar{z}^n_{k(\ell,n)}\ge \bar{z}_\ell$ by definition, we may rewrite the first sum, and combine the second sum and the third term in order to obtain that $\tilde{u}(z)  - \tilde{u}_n(z)$ is equal to
	\begin{align}
		&\sum_{i=1}^{\ell-1} u^{(1)}(\ubar{z}_{i+1}) \Big[u^{(2)}(\bar{z}_\ell \wedge \bar{z}_{i+1}) - u^{(2)}(\bar{z}_\ell \wedge \bar{z}_{i})\Big]- \sum_{i=1}^{k(\ell,n) -1} \hspace{-2ex} u_n^{(1)}(\ubar{z}^n_{i+1})  \Big[u_n^{(2)}(\bar{z}_\ell \wedge \bar{z}^n_{i+1}) - u_n^{(2)}(\bar{z}_\ell \wedge \bar{z}^n_{i})\Big] \label{eq:A.1} \\
		&   +  u^{(1)}(\ubar{z}_{\ell+1})  \Big[u^{(2)}(z) - u^{(2)}(z \wedge \bar{z}_{\ell})\Big] - \sum_{i=k(\ell,n)}^{k(\ell+1,n) -1} \hspace{-2ex} u_n^{(1)}(\ubar{z}^n_{i+1}) \Big[u_n^{(2)}(z \wedge \bar{z}^n_{i+1}) - u_n^{(2)}(z \wedge \bar{z}^n_{i})\Big] \notag  \\
		&-  u_n^{(1)}(\ubar{z}^n_{k(\ell,n)}) \Big[u_n^{(2)}(z \wedge \bar{z}^n_{k(\ell,n)}) - u_n^{(2)}(z \wedge \bar{z}_{\ell})\Big] . \notag
	\end{align}
	\ \\[-1ex]
	By the induction hypothesis, each of the two sums in \eqref{eq:A.1} are bounded in absolute value by $2 C(k(\ell,n)+\ell)\, |u - u_n|_1^*$. Using a telescoping sum argument, we find that the absolute value of the remaining part is bounded by $2 C (k(\ell+1,n) - k(\ell,n) + 1) \, |u - u_n|_1^*$. Thus,
	\begin{align*}
		|\tilde{u}(z) - \tilde{u}_n(z)| \; &\leq \; 2 C(k(\ell,n)+\ell) \, |u - u_n|_1^* + 2 C(k(\ell+1,n) - k(\ell,n) +1 ) \, |u - u_n|^*_1\\
		&= \; 2 C(k(\ell+1,n)+ \ell+1) \,  |u - u_n|_1^* \; < \; \varepsilon,
	\end{align*}
for all $n\geq N$.\end{proof}

\begin{remark}\ 
	If we have another sequence of càdlàg functions $z_n$ such that $(y_n,z_n) \to (y,z)$ in $(\D_{\R^{2d}}[0,T], \dM)$, then it is immediate from the proof above that one obtains 
	$$\Bigl(x_n,y_n,z_n, \int_0^\bullet \! x_n(s-)\diff y_n(s)\Bigr) \longrightarrow  \Bigl(x,y,z, \int_0^\bullet \! x(s-) \diff y(s)\Bigr)$$
	on $(\D_{\R^{4d}}[0,\,T], \dM)$ as $n\to \infty$.
\end{remark}


\section{Proofs of the results from Section \ref{sect:martingale_integrators}} \label{sect:proofs_sect4}
Here we give the proofs of the results in Section \ref{sect:martingale_integrators}, in the same order as they were presented.

\subsection{Relatively compact local martingales in $M_1$}

We begin by proving an alternative tightness criterion for $J_1$ akin to Aldous' criterion \cite[Thm.~16.10]{billingsley}. Denote by $D\subseteq [0,\infty)$, with $D \neq \emptyset$, any set such that $\operatorname{sup}(D)=+\infty$.

\begin{prop}[An Aldous-type tightness criterion for $J_1$] \label{prop:alternative_$J_1$_tightness_criterion}
	A sequence of càdlàg adapted processes $(X^n)_{n\ge 1}$ on $(\Omega^n, \mathcal{F}^n, \mathbbm{F}^n, \Pro^n)$ is tight in $(\D_{\R^d}[0,\infty), \dJ)$ if
	\begin{enumerate}[(i)]
		\item for every $\varepsilon>0$ and $t>0$ there exists $K=K(\varepsilon,T)>0$ such that 
		$$ \sup{n\ge 1} \, \Pro^n(|X^n|^*_T > K) \; \le \; \varepsilon \, ;$$
		\item for all $\eta>0$ and $T\in D$ it holds
		$$ \lim\limits_{\theta \downarrow 0} \; \limsup\limits_{n\to \infty} \; \sup{\substack{\tau,\,\sigma \in \, \mathcal{T}^n \\ 0\, \le \, \tau\,  \le\,  \sigma\, \le \, (\tau+\theta)\wedge T}} \; \; \Pro^n\biggl( \, \sup{0\le s \le \theta} \, |X^n_{(\tau-s)\vee 0}-X^n_\tau|\wedge |X^n_{\tau}-X^n_{\sigma}| \, > \, \eta \biggr) \; = \; 0 $$
		where $\mathcal{T}^n$ denotes the set of all $\mathbbm{F}^n$-stopping times; and 
		\item for all $\eta>0$ and $T\in D$, 
		$$ \lim\limits_{\theta \downarrow 0} \; \limsup\limits_{n\to \infty} \; \Pro^n \biggl( \sup{0\le s \le \theta}\, |X^n_s-X^n_0| \vee |X^n_{T-s}-X^n_T|\, > \, \eta \biggr) \; = \; 0 \,.$$
	\end{enumerate}
\end{prop}
\begin{proof} Assume the conditions of Proposition \ref{prop:alternative_$J_1$_tightness_criterion} and, towards a contradiction, suppose the sequence $(X^n)_{n \ge 1}$ is not tight in $(\D_{\R^d}[0,\infty), \dJ)$. Consider a lack of the $J_1$ tightness condition based on $\hat w$ as specified in \cite[Thm.~12.4]{billingsley}. Clearly, this means that there exist $\eta,\varepsilon >0$, $T>0$ and a subsequence, which we also denote by $(X^{n})_{n\ge 1}$, such that 
	\begin{align}
		\Pro^n \left( \hat w_{n^{-1}}(X^n,X^n) \, > \, \eta \right) \; \ge 2\varepsilon \label{eq:altern_crit_$J_1$_tight_1}
	\end{align}
	for all $n\ge 1$, where we recall that $\hat w^T_\delta(X,X)= \operatorname{sup}\{ |X_t-X_s| \wedge |X_u-X_t| : 0\le s < t<u\le (s+\delta)\wedge T\}$. By monotonicity of $\hat w$ in $T$, without loss of generality we can assume $T \in D$. Now, according to condition (ii), there 
	is $\theta_1:=\theta(\eta/2,\varepsilon/2)$ such that, for all $n\ge 1$,
	\begin{align} 
		\Pro^n\Big( \, \sup{0\le s \le \theta_1} \, |X^n_{(\tau-s)\vee 0}-X^n_\tau|\wedge |X^n_{\tau}-X^n_{\sigma}| \, > \, \eta/2 \Big) \; \le \; \frac \varepsilon 2\; \label{eq:altern_crit_$J_1$_tight_2}
	\end{align}
	whenever $\tau, \sigma \in \mathcal{T}^n$ with $0\le \tau \le \sigma \le (\tau+ \theta_1)\wedge T$. Fix $q \in 2\N$ such that $q \theta_1> 4T$. Define $\mathbbm{F}^n$--stopping times $\tau^n_0 \equiv 0$ and $\tau^n_{k+1}:= \operatorname{inf}\{t>\tau^n_{k} : |X^n_t-X^n_{\tau^n_k}|> \eta/2\}$ and deduce that 
	\begin{align} 
		\Pro^n ( \tau^n_{k+2} \le T ,  \tau^n_{k+2}-\tau^n_{k} \le \theta_1) \le  \Pro^n\Big(  \sup{0\le s \le \theta_1}  |X^n_{(\tau -s)\vee 0}-X^n_{\tau}|\wedge |X^n_{\tau}-X^n_{\sigma}|  > \eta/2 \Big) \label{eq:altern_crit_$J_1$_tight_3}
	\end{align}
	for all $k\ge 1$ by applying the above to $\tau=\tau^n_{k+1}\wedge (\tau^n_k +\theta_1)\wedge T$ and $\sigma=\tau^n_{k+2}\wedge (\tau^n_k +\theta_1)\wedge T$. Making use of \eqref{eq:altern_crit_$J_1$_tight_2}-\eqref{eq:altern_crit_$J_1$_tight_3}, we find that 
	\begin{align*}
		T &\Pro^n( \tau^n_q \le T) \; \ge \; \E^n[ \tau^n_q \ind_{\{\tau^n_q \le T\}}] \; = \; \E^n\Big[ \sum_{k=1}^{q/2} (\tau^n_{2k}-\tau^n_{2k-2}) \ind_{\{\tau^n_q \le T\}} \Big] \\[-4pt]
		\;\;&\ge \; \E^n\Big[ \sum_{k=1}^{q/2} (\tau^n_{2k}-\tau^n_{2k-2}) \ind_{\{\tau^n_q \le T\, , \, \tau^n_{2k}-\tau^n_{2k-2} > \theta_1 \}} \Big] \\
		\;\;&\ge  \; \frac{q\theta_1}{2} \; \operatorname{sup}_{k=1,\, ...,\, q/2} \, \Pro^n(\tau^n_q \le T\, , \, \tau^n_{2k}-\tau^n_{2k-2} > \theta_1 ) \; \ge \; \frac{q\theta_1}{2} \; \Big( \Pro^n(\tau^n_q \le T) - \frac{\varepsilon}{2} \Big) 
	\end{align*}
	Thus, $\Pro^n(\tau^n_q \le T)\le \varepsilon$ by our choice of $q$ with $q \theta_1> 4T$. Combined with \eqref{eq:altern_crit_$J_1$_tight_1}, it follows that 
	\begin{align}
		\Pro^n(\tau^n_q> T \, , \, \hat w^T_{n^{-1}}(X^n,X^n)>\eta)  \;\ge  \; \varepsilon \label{eq:altern_crit_$J_1$_tight_6}
	\end{align} 
	for all $n\ge 1$. Now, as in \eqref{eq:altern_crit_$J_1$_tight_2} there exists $\theta_2:=\theta(\eta/2, \varepsilon/(2q))$ such that we have 
	\begin{align} 
		\Pro^n\Big( \, \sup{0\le s \le \theta_2} \, |X^n_{(\tau-s)\vee 0}-X^n_\tau|\wedge |X^n_{\tau}-X^n_{\sigma}| \, > \, \eta/2 \Big) \; \le \; \frac \varepsilon {2q} 
	\end{align}
	for all $n\ge 1$ and $\tau, \sigma \in \mathcal{T}^n$ with $0\le \tau \le \sigma \le (\tau+ \theta_2)\wedge T$, in particular 
	\begin{align} 
		\Pro^n\Big(  \sup{0\le s \le \theta_2} |X^n_{((\tau^n_{k} \wedge T) -s)\vee 0}-X^n_{\tau^n_k\wedge T}|\wedge |X^n_{\tau^n_k\wedge T}-X^n_{\tau^n_{k+1}\wedge (\tau^n_k+\theta_2)\wedge T}|  >  \eta/2 \Big)  \le  \frac \varepsilon {2q}\; \label{eq:altern_crit_$J_1$_tight_4}
	\end{align}
	for all $n\ge 1$ and $k\ge 1$. Denote $\Lambda^n_k$ the event inside the probability on the left-hand side of \eqref{eq:altern_crit_$J_1$_tight_4} and choose $n^{-1}\le \theta_2$. Then, we will show that 
	\begin{align} 
		\{\tau^n_q \, > \, T\} \; \cap \;\{ \hat w_{n^{-1}}(X^n,X^n)\, > \, \eta \} \; \subseteq \; \{\tau^n_q \, > \, T\} \; \cap \; \bigcup_{k=1}^{q} \Lambda^n_k  \; . \label{eq:altern_crit_$J_1$_tight_5}
	\end{align}
	Once this has been done, using \eqref{eq:altern_crit_$J_1$_tight_6} and then combining \eqref{eq:altern_crit_$J_1$_tight_5} and \eqref{eq:altern_crit_$J_1$_tight_4}, we immediately obtain the desired contradiction as follows:
	$$ \varepsilon \; \le \; \Pro^n ( \tau^n_q  >  T \, ,\, \hat w_{n^{-1}}(X^n,X^n) >  \eta) \; \le \; \sum_{k=1}^q \Pro^n (\Lambda^n_k) \; \le \; \frac {\varepsilon}{2} \; .$$
	Therefore, it only remains to show the inclusion \eqref{eq:altern_crit_$J_1$_tight_5}: let $\omega \in \{\tau^n_q  >  T\}  \cap \{ \hat w^T_{n^{-1}}(X^n,X^n) > \eta \}$ and consider $0\le s(\omega) \le t(\omega) \le u(\omega) \le (s(\omega)+n^{-1})\wedge T \le (s(\omega)+\theta_2)\wedge T$ such that 
	\begin{align}\label{eq:altern_crit_$J_1$_tight_7}
		|X^n_{s(\omega)}(\omega)-X^n_{t(\omega)}(\omega)| \, \wedge \, |X^n_{t(\omega)}(\omega)-X^n_{u(\omega)}(\omega)| \; > \;  \eta\; .
	\end{align}
	In view of $s(\omega)\le T < \tau^n_q(\omega)$, define $\ell:= \operatorname{max}\{0 \le k : \tau^n_k(\omega) < s(\omega)\} \le q-1$. From here on, to lighten the notation, we omit the dependence on $\omega$. We proceed by case distinction.
	
\emph{Case 1:} If $|X^n_{\tau^n_\ell}-X^n_s|>\eta/2$, then, by definition of $\ell$, it must hold that $\tau^n_{\ell+1}=s$ and $s<\tau^n_{\ell+2} \le t$. And clearly, either $|X^n_{\tau^n_{\ell+2}}-X^n_t|> \eta/2$ or $|X^n_{\tau^n_{\ell+2}}-X^n_u|> \eta/2$, since otherwise $|X^n_t-X^n_u|< \eta$ by the triangle inequality, which would violate \eqref{eq:altern_crit_$J_1$_tight_7}. Thus, we obtain $s<\tau^n_{\ell+3}\le u$, which yields $\ell\le q-4$ with $|\tau_{\ell+1}^n -\tau_{\ell+2}^n| \leq \theta_2 $ and $|\tau_{\ell+3}^n -\tau_{\ell+2}^n| \leq \theta_2 $, so one readily deduces from the definition of $\Lambda^n_{\ell+2}$ in \eqref{eq:altern_crit_$J_1$_tight_4} that $\omega \in \Lambda^n_{\ell+2}$.

\emph{Case 2:} If $|X^n_{\tau^n_\ell}-X^n_s|\le \eta/2$, then $|X^n_{\tau^n_\ell}- X^n_t| > \eta/2$ since otherwise, again by the triangle inequality, $|X^n_s- X^n_t|\le \eta$. Hence, we conclude $s<\tau^n_{\ell+1}\le t$ and either $|X^n_s- X^n_{\tau^n_{\ell+1}}|> \eta/2$ or $|X^n_t- X^n_{\tau^n_{\ell+1}}|> \eta/2$. In the first case, we proceed as above to obtain $s<\tau^n_{\ell+2}\le u$ and therefore $\omega \in \Lambda^n_{\ell+1}$. For the second case, it follows immediately that $s<\tau^n_{\ell+2}\le t$ and, by same argument as before, $s<\tau^n_{\ell+3}\le u$ with $\ell\le q-4$, which yields $\omega \in \Lambda^n_{\ell+2}$.
\end{proof}

Equipped with Proposition \ref{prop:alternative_$J_1$_tightness_criterion}, we can now proceed to the proof of Theorem \ref{cor:3.11}.
\begin{proof}[Proof of Theorem \ref{cor:3.11}] \ 
	It suffices to prove the assertion for $d=1$. Indeed, a sequence of $\R^d$-valued càdlàg functions $x_n$ converges to some càdlàg limit $x$ on $(\D_{\R^d}[0,\infty), \rho)$ where $\rho \in \{\dJ, \dM\}$, if and only if for every $\zeta \in \R^d$ it holds that $\zeta^T x_n$ converges to $\zeta^T x$ in $(\D_{\R}[0,\infty), \rho)$ (see e.g. \cite[Thm.~12.7.2]{whitt} for the $M_1$ topology).
	
	We argue by contraposition. To this end, suppose $\mathcal{A}$ is not relatively compact in $J_1$. In particular, we can then find a sequence $(X^n)_{n\ge 1} \subseteq \mathcal{A}$ which has no weakly convergent subsequence in $J_1$. Now, assume, for a contradiction, that $\mathcal{A}$ is relatively compact in $M_1$, implying the existence of a subsequence, also denoted by $(X^{n})_{n\ge 1}$, converging weakly in $M_1$ to some càdlàg process $X$. Let $D= [0,\infty) \setminus \operatorname{Disc}_\Pro(X)$ be the co-countable set of almost sure continuity points of $X$ and note that, since our subsequence $(X^{n})_{n\ge 1}$ is tight in $M_1$ due to Prokhorov's theorem, conditions (i) and (iii) of Proposition \ref{prop:alternative_$J_1$_tightness_criterion} are satisfied at all $T\in D$. Again according to Prokhorov's theorem, the latter sequence is, however, not tight in $J_1$, which then entails that condition (ii) of Proposition \ref{prop:alternative_$J_1$_tightness_criterion} must be violated. Thus, there exists a further subsequence---which we will still denote by $(X^n)_{n\ge 1}$---as well as $\eta_1>0$, $0<\varepsilon_1 <1$, $T \in D$ and $\mathbbm{F}^n$--stopping times $\tau^n,\sigma^n$ with $0\le \tau^n\le \sigma^n\le (\tau^n+n^{-1})\wedge T$ such that 
	$$ \Pro^n\left( \, \sup{0\le s \le n^{-1}} \, |X^n_{(\tau^n-s)\vee 0}-X^n_{\tau^n}|\wedge |X^n_{\tau^n}-X^n_{\sigma^n}| \, > \, \eta_1 \right) \; \; \ge \; \; 2\varepsilon_1$$
	for all $n\ge 1$. Then, without loss of generality we can assume that
	\begin{align} \Pro^n\left( \, \sup{0\le s \le n^{-1}} \, |X^n_{(\tau^n-s)\vee 0}-X^n_{\tau^n}|\wedge (X^n_{\sigma^n}-X^n_{\tau^n}) \, > \, \eta_1 \right) \; \; \ge \; \; \varepsilon_1 \label{eeq:4.17} \end{align}
	for all $n\ge 1$ (otherwise, consider $-X^{n}$). Denote the event inside the probability on the left-hand side of \eqref{eeq:4.17} by $\{\bar w^{\uparrow}_{n^{-1}}(X^n) > \eta_1\}$. From the steps above (towards a contradiction), the $X^n$ are $M_1$ tight, so the classical $M_1$ tightness criteria (see e.g.~Theorem \ref{thm:A.8} or \cite[Thm.~12.12.3]{whitt}) give that, for every $\varepsilon_2, \eta_2>0$, there exists $\theta:=\theta(\varepsilon_2, \eta_2) >0$ such that
	\begin{align} \Pro^{n} \left( w^{\prime \prime}_{\theta}(X^{n}) \, \le \, \eta_2 \right) \; \ge \; 1-\varepsilon_2, \label{eeq:4.18}\end{align}
	for all $n\ge 1$, where we refer to Definition \ref{def:A7} for the definition of $w''$. We now make use of the localised uniform integrability. Without loss of generality, for every $n\ge 1$, we can choose $\ell:=\ell(n)\ge 1$ large enough so that, for the stopping times $\rho^n:=\rho^n_\ell$, the family $(X^n_{\bullet \wedge \rho^n})_{n\ge 1}$ is uniformly integrable and \eqref{eeq:4.17} and \eqref{eeq:4.18} hold true for $X^n_{\bullet \wedge \rho^n}$ instead of $X^n$. Otherwise, repeat the above with $4\varepsilon_1$ and $\varepsilon_2/2$ instead of $2\varepsilon_1$ and $\varepsilon_2$, and choose to every $n\ge 1$ an $\ell \ge 1$ with $\Pro^n(\rho^n_\ell\le T) \le \varepsilon_1 \wedge (\varepsilon_2/2)$. By uniform integrability of the sequence $(X^n_{\bullet \wedge \rho^n})_{n\ge 1}$, we know in particular that we have uniform integrability of the family 
	$$ \left\{ \; X^{n}_{\mu \wedge \rho^n} \,= \, \E^{n}\Bigl[X^{n}_{T\wedge \rho^n} \,| \, \mathcal{F}^{n}_\mu \Bigr]  \; : \; k \ge 1 \text{ and } \mu \text{ is an $\mathbbm{F}^{n}$-stopping time bounded by $T$} \; \right\} .$$
	Thus, there exists $\delta>0$ such that, for all $n \ge 1$ and all $\mathbbm{F}^{n}$-stopping times $\rho^n$ bounded by $T$, 
	\begin{align} 
		\E^{n}  \Bigl[ \, |X^{n}_{\mu\wedge \rho^n}| \, \ind_{\Lambda_n} \, \Bigr] \; < \; \frac{\varepsilon_1\eta_1}{8} \label{eeq:4.19unifint}
	\end{align}
	whenever $\Lambda_n \in \mathcal{F}^{n}$ with $\Pro^{n}(\Lambda_n) < \delta$. Choose $\eta_2:=\eta_1 \varepsilon_1/4$ and $\varepsilon_2<(\varepsilon_1/4) \wedge \delta$. Fix $n\ge 1$ such that $2n^{-1}< \theta$ and, to simplify notation, set $Z:=X^n_{\bullet \wedge \rho^n}$ and $(\Omega,\mathcal{F},\mathbbm{F},\Pro):=(\Omega^n,\mathcal{F}^n,\mathbbm{F}^n,\Pro^n)$ as well as $\tau:=\tau^n$ and $\sigma:=\sigma^n$. Consider the sets 
	\begin{align*} 
		A:= \left\{ \bar w^{\uparrow}_{n^{-1}}(Z) > \eta_1 \right\} \cap \left\{w''_{\theta}(Z) \le  \eta_2 \right\} \;\, \text{ and } \;\,
		B :=  \left\{ \bar w^{\uparrow}_{n^{-1}}(Z) > \eta_1\right\}^\complement \cap \; \left\{w''_{ \theta}(Z)  \le \eta_2 \right\}
	\end{align*}
	along with $C:= \{w''_{\theta}(Z) >  \eta_2 \}$, noting that $A\cup B\cup C=\Omega$. Define $\Gamma :=\{\operatorname{sup}_{\, 0\le s \le n^{-1}} \, (Z_{\tau}-Z_{(\tau-s)\vee 0}) >  \eta_1\}$ and $\mathbbm{F}$--stopping times  
	$\tilde{\tau} := \tau \ind_{\Gamma}  +  T  \ind_{\Gamma^\complement}$, \, $\tilde{\sigma} := \sigma \ind_{\Gamma}  +  T \ind_{\Gamma^\complement}$. These are indeed stopping times since $\Gamma \in \mathcal{F}_\tau \subseteq \mathcal{F}_\sigma$, due to $\tau\le \sigma$. As the above stopping times are bounded by $T$, and $Z$ is a càdlàg martingale, by optional stopping we obtain 
	\begin{align} \E [ \, Z_{\tilde\tau} \, ] \; = \; \E[Z_0] \; = \;  \E \left[ \, Z_{ \tilde \sigma} \, \right] . \label{eeq:4.19} \end{align}
	We will show that this yields a contradiction. To see this, we proceed in three steps:
	
\emph{Step 1:} We begin by considering the set $A$, for which we claim that $A\cap \Gamma^\complement=\emptyset$. Indeed, on $\{ \bar w^{\uparrow}_{n^{-1}}(Z) > \eta_1 \}$, we have $\operatorname{sup}_{\, 0\le s \le n^{-1}} \, |Z_{(\tau-s)\vee 0}-Z_{\tau}| >  \eta_1$ and therefore, on $\{ \bar w^{\uparrow}_{n^{-1}}(Z) > \eta_1 \}\cap \Gamma^\complement$, it must hold $\operatorname{inf}_{\, 0\le s \le n^{-1}} \, (Z_{\tau}-Z_{(\tau-s)\vee 0}) <  -\eta_1$. Yet, this implies that there exists a random time $0\le s \le n^{-1}$ such that $Z_\tau < (Z_{(\tau-s)\vee 0}\wedge Z_{\sigma})-\eta_1$ and $\sigma-(\tau-s)\le 2n^{-1}<\theta$, since on $Z_\sigma>Z_\tau$ on $\{ \bar w^{\uparrow}_{n^{-1}}(Z) > \eta_1 \}$ by definition of the event after \eqref{eeq:4.17}. Consequently, we obtain $\lVert Z_{\tau} - [ Z_{(\tau-s)\vee 0}, Z_{\sigma}] \rVert > \eta_1 > \eta_2$, where $\norm{\cdot}$ is defined as in Definition \ref{def:A7}. This however yields that $\{w''_\theta(Z)\le \eta_2\} \cap \{ \bar w^{\uparrow}_{n^{-1}}(Z) > \eta_1 \}\cap \Gamma^\complement=\emptyset$ and so we deduce $A\cap \Gamma^\complement=\emptyset$. Now, with $A_1:=\{ \bar{w}^{\uparrow}_{n^{-1}}(Z)\, > \, \eta_1 \}$ and $A_2:=\{w^{''}_{\theta}(Z) \, \le \, \eta_2 \}$, we have
		\begin{align*} \Pro(A) \;= \; \Pro(A_1 \cap A_2) \; = \; \Pro(A_1) \, + \, \Pro(A_2) \, - \, \Pro(A_1\cup A_2) \; &\ge \; \varepsilon_1 \, + \, (1-\varepsilon_2) \, - \, 1 \; \ge \; \frac{3\varepsilon_1}{4},
		\end{align*} where the first inequality comes from \eqref{eeq:4.17}-\eqref{eeq:4.18} and the second is due to the fact that we chose $\varepsilon_2< \varepsilon_1/4$. Next, it follows from the definition of $\{ \bar w^{\uparrow}_{n^{-1}}(Z) > \eta_1 \}$ that we have $(Z_{\tilde\sigma}-Z_{\tilde\tau})=(Z_{\sigma}-Z_{\tau})>\eta_1$ on the set $A=A\cap \Gamma$, which combined with the above gives
		\begin{align} \E[ Z_{\tilde \sigma} \ind_A] \; \, > \, \; \E[ Z_{\tilde\tau} \ind_A] \; + \; \eta_1 \Pro(A) \; \ge \; \E[ Z_{\tilde\tau} \ind_A] \; + \; \frac{3\varepsilon_1\eta_1}{4}. \label{eeq:4.20} \end{align}
		
		\emph{Step 2:} On $B\cap \Gamma^\complement$, the definition of $\tilde\tau, \tilde \sigma$ yields $Z_{\tilde \sigma}-Z_{\tilde \tau}=Z_T-Z_T=0$. Consider now the set $B\cap \Gamma$, on which $\tilde\tau=\tau, \tilde\sigma=\sigma$ and there exists a random time $0 \le s \le n^{-1}$ such that $Z_{\tau}-Z_{(\tau-s)\vee 0} >\eta_1>\eta_2$. Since $B\subseteq\{w''_\theta(Z)\le \eta_2\}$, we deduce  $Z_\sigma\ge Z_\tau- \eta_2 = Z_\tau - \eta_1\varepsilon_1/4$. Otherwise there would exist $\omega \in B\subseteq\{w''_\theta(Z)\le \eta_2\}$, evaluated at which, it would hold $Z_{\tau} > (Z_{\tau-s} \vee Z_{\sigma})+\eta_2$ and this yields $\lVert Z_\tau - [Z_{\tau-s}, Z_\sigma] \rVert> \eta_2$. Due to $\sigma-(\tau-s)\le 2n^{-1}< \theta$, the latter would imply $w''_\theta(Z)>\eta_2$, which is a contradiction to $\omega \in B \subseteq\{w''_\theta(Z)\le \eta_2\}$. Hence,
		\begin{align} \E \left[ \, Z_{\tilde \sigma} \, \ind_B \, \right] \; \ge \; \E \left[ \, Z_{\tilde \tau} \, \ind_B\, \right] \; - \; \frac{\varepsilon_1\eta_1}{4}.\label{eeq:4.21}\end{align}
		
	\emph{Step 3:} Lastly, recalling $C=\{w''_{\theta}(Z) >  \eta_2 \}$, since $\Pro(C)< \varepsilon_2 \le \delta$ due to \eqref{eeq:4.18}, we obtain, on behalf of the uniform integrability bound (\ref{eeq:4.19unifint}), the inequality
		\begin{align*}
			\E \left[  (Z_{\tilde \tau}-Z_{\tilde\sigma} )  \ind_C  \right] \, \le \, \E \left[  |Z_{\tilde \sigma}| \ind_C  \right] \, + \, \E \left[ |Z_{\tilde \tau}|  \ind_C  \right] 
			\, \le \, \frac{\varepsilon_1\eta_1}{8} \, + \, \frac{\varepsilon_1\eta_1}{8} \, = \, \frac{\varepsilon_1 \eta_1}{4}
		\end{align*}
		and, in turn, we get
		\begin{align} \E \left[ \, Z_{\tilde \sigma} \, \ind_C \, \right] \; \ge \; \E \left[ \, Z_{\tilde \tau} \, \ind_C \, \right] \; - \; \frac{\varepsilon_1\eta_1}{4}.\label{eeq:4.22}\end{align}
		
	Finally, bringing together (\ref{eeq:4.20}), (\ref{eeq:4.21}) and (\ref{eeq:4.22}), we conclude that $\E[ \, Z_{\tilde{\sigma}} ] \ge  \E [ \, Z_{\tilde\tau} ]  + (\varepsilon_1\eta_1)/4$, which yields the desired contradiction to (\ref{eeq:4.19}). Hence, the family $\{X^i : i \in I\}$ cannot be relatively compact for the $M_1$ topology.
\end{proof}

 \subsection{Proof of Proposition \ref{prop:propagation_of_local_martingale_property}}
 
	\begin{proof}[Proof of Proposition \ref{prop:propagation_of_local_martingale_property}] 
		Denote $\mathbbm{F}^n:=(\mathcal{F}^n_t)_{t\ge 0}$ the natural filtrations generated by $(H^n,M^n)$. Without loss of generality, we may assume the $M^n$ to be martingales. Indeed, by the local martingale property, there exist $\mathbbm{F}^n$-stopping times $\sigma^n$ such that $M^n_{\bullet \wedge \sigma^n}$ are martingales and $\Pro^n(\sigma^n\le T) \to 0$ as $n\to \infty$ for all $T>0$. Hence, $M^n_{\bullet \wedge \sigma^n}\Rightarrow X$ on $(\D_{\R^d}[0,\infty),\dM)$ since convergence in the $M_1$ topology on $\D_{\R^d}[0,\infty)$ is uniquely characterised by $M_1$ convergence on almost all compacts. Next, condition \eqref{eq:localised_uniform_integrability_of_martingale_jumps} ensures the localised uniform integrability of the local martingales $M^n$ on compact time intervals (as in Definition \ref{defn:local_UI}), and therefore, by Corollary \ref{cor:$M_1$_conv_impl_$J_1$_conv_martingales}, there is convergence $(H^n,M^n)\Rightarrow (H,X)$ on $(\D_{\R^d}[0,\infty),\dM)\times (\D_{\R^d}[0,\infty),\dJ)$. For any $\alpha \in \D_{\R^d}[0,\infty)$ and $c>0$ define $\rho_c(\alpha):= \operatorname{inf}\{t>0: |\alpha(t)|\ge c \text{ or } |\alpha(t-)|\ge c\}$. As described at the beginning of the proof of \cite[Prop.~IX.1.17]{shiryaev}, there exists a co-countable subset $C\subseteq [0,\infty)$ such that for all $c\in C$ it holds that $(H^n,M^n_{\bullet \wedge \rho_c(M^n)})\Rightarrow (H,X_{\bullet \wedge \rho_c(X)})$ on $(\D_{\R^d}[0,\infty),\dM)\times (\D_{\R^d}[0,\infty),\dJ)$. Finally, let $c\in C$. Following the argument in the proof of \cite[Prop.~IX.1.12]{shiryaev} (with the definition of quantities as in the proof of \cite[Prop.~IX.1.10]{shiryaev}) we first deduce the uniform integrability on compacts of $X_{\bullet \wedge \rho_c(X)}$ from \eqref{eq:localised_uniform_integrability_of_martingale_jumps}, and then its martingale property in terms of the filtration generated by $(H,X)$. 
\end{proof}

\section{Proofs of the results from Section \ref{sect:CTRW_applications}}

\subsection{Proof of Proposition \ref{prop:CTRW_are_localised_UI}} 

\begin{proof}[Proof of Proposition \ref{prop:CTRW_are_localised_UI}]
    Let $\tau^n_c$ be stopping times defined as in the definition of \eqref{eq:Mn_An_condition}, and let $K>0$. Furthermore, let $\rho^n_c:= \operatorname{inf}\{t>0: n^{-\beta}N_{nt}>c\}$. The martingality of the processes $X^n$ is shown in \cite[Prop.~3.3]{andreasfabrice}. Clearly, the stopping times $\sigma^n_c:= \tau^n_c\wedge \rho^n_c$ are a localising sequence as in Definition \ref{defn:local_UI}. For $0\le t \le T$ and $K>2c$ it holds that 
   \begin{align*}
       \E \Big[|X^n_{t\wedge \sigma^n_c}| \ind_{\{|X^n_{t\wedge \sigma^n_c}|>K\}}\Big] \; \le \; c \, \Pro \big(|X^n|^*_{T}>K\big) \,+ \, \E \Big[|\Delta X^n_{t\wedge \sigma^n_c}| \ind_{\{|\Delta X^n_{t\wedge \sigma^n_c}|>K/2\}}\Big].
   \end{align*}
   The first summand on the right-hand side converges to zero, uniformly in $n$, as $K\to \infty$ due to the tightness of the $X^n$ on the Skorokhod space. For the second summand, we note that 
     $$\E \Big[|\Delta X^n_{t\wedge \sigma^n_c}| \ind_{\{|\Delta X^n_{t\wedge \sigma^n_c}|>K/2\}}\Big] \, \le \, \frac 1 {n^{-\frac \beta \alpha}} \sum_{j=1}^{cn^\beta} \E\Big[ |\theta_j| \ind_{\{ |\theta_j|> n^{\frac \beta \alpha} K/2\}}\Big] \, = \, \frac {cn^\beta} {n^{-\frac \beta \alpha}} \E\Big[ |\theta_0|^{\alpha-\gamma} \ind_{\{ |\theta_0|> n^{\frac \beta \alpha} K/2\}}\Big] $$
   with $\gamma:=\alpha-1>0$. Then, proceed as in the proof of \cite[Prop.~5.3]{andreasfabrice} to obtain
   $$ \operatorname{sup}_{0\le t\le T,\;  n\ge 1} \, \E \Big[|\Delta X^n_{t\wedge \sigma^n_c}| \ind_{\{|\Delta X^n_{t\wedge \sigma^n_c}|>K/2\}}\Big]  \; \le \; \frac{32^\gamma c\alpha}{\gamma K^{\gamma}} \; \to \;  0 $$
   as $K\to \infty$, which yields the desired localised uniform integrability.
\end{proof}

\subsection{Construction of the counterexample in Proposition \ref{counterexample:3.4}} \label{sec:4.1}

\begin{proof}[Proof of Proposition \ref{counterexample:3.4}] We split the construction into four steps.\\
	\phantom{ee} \emph{Step 1:} We can rewrite the $X^n$ as\\[-4ex]
	\begin{align*} X^n_t \; = \; \frac{1}{n^{\frac1\alpha}} \; \frac{1}{c_0+c_1} \; \sum_{k=1}^{\lfloor nt \rfloor} \; (c_0\theta_k + c_1 \theta_{k-1}) \; = \; \frac{1}{n^{\frac1\alpha}} \;   \sum_{k=1}^{\lfloor nt \rfloor} \; \theta_k \; - \; \frac{1}{n^{\frac1\alpha}}\, \frac{c_1}{c_0+c_1} \, ( \theta_{\lfloor n t\rfloor} - \theta_0) 
	\end{align*}
	and we define $A^n_t:=n^{-\frac1\alpha}  \sum_{k=1}^{\lfloor nt \rfloor} \theta_k$ and $B^n_t:=c_1 n^{-\frac1\alpha} (c_0+c_1)^{-1}  ( \theta_{\lfloor n t\rfloor} - \theta_0)$. It is well-known that the $A^n$ converge weakly in the $J_1$ topology to an $\alpha$-stable Lévy process $A$ if $\E[\theta_0]=0$ (see \eqref{eq:mov_av_scaling_limit}). As uncorrelated moving averages, the $A^n$ satisfy \eqref{eq:Mn_An_condition} according to Theorem \ref{thm:CTRW_has_GD}. By Proposition \ref{prop:3.3} and Theorem \ref{thm:3.19}, for every sequence of càdlàg processes $H^n$ adapted to the natural filtration of the $X^n$ and converging weakly to the zero process, this implies $\int_0^\bullet H^n_{s-} \diff A^n_s \Rightarrow 0$. In particular, its f.d.d.~converge to zero. Towards a contradiction, suppose the f.d.d.~of $(\int_0^\bullet H^n_{s-} \diff X^n_s)_{n\ge 1}$ are tight, and let $D\subseteq [0,\infty)$ be countable. Then, by Prokhorov's theorem and a diagonal sequence argument, we can extract a subsequence for which the f.d.d. converge on $D$. Without loss of generality, assume that this holds for the original sequence $(\int_0^\bullet H^n_{s-} \diff X^n_s)_{n\ge 1}$. The continuous mapping theorem now implies that the f.d.d. of $\int_0^\bullet  H^n_{s-} \diff B^n_s = \int_0^\bullet H^n_{s-} \; \diff A^n_s  - \int_0^\bullet  H^n_{s-} \diff X^n_s $ converge on $D$. Fix $s \in D$. The $X^n$ being càdlàg pure jump processes, we can easily express the integrals in terms of \\[-4ex]
	\begin{align}
		\int_0^s \, H^n_{r-} \; \diff B^n_r \; = \; \frac{c}{n^{\frac 1 \alpha}}\;  \sum_{k=1}^{\lfloor ns \rfloor} \, H^n_{\frac{k}{n}}\; ( \theta_{k} - \theta_{k-1})\; = \; \bar{B}_n \; + \; \hat{B}_n \label{eq:5.10}
	\end{align}
	for $n$ large enough, where $c:=c_1(c_0+c_1)^{-1}$ and \\[-2ex]
    $$\bar{B}_n \, := \, \frac{c}{n^{\frac 1 \alpha}}\, \big( H^n_{\frac{\lfloor n s\rfloor}{n}} \theta_{n} \; - \; H^n_{\frac{1}{n}} \theta_0\big)  \qquad \text{and} \qquad \hat{B}_n \, := \, \frac{c}{n^{\frac 1 \alpha}}\;  \sum_{k=1}^{\lfloor ns \rfloor -1} \, \theta_k \; \big( H^n_{\frac k n} - H^n_{\frac{k+1}{n}}\big).$$
    Clearly, $\bar{B}_n$ converges to zero in probability, which yields that $\hat{B}_n$ converges weakly. Hence, it is enough to find an (adapted) sequence of $H^n$ converging weakly to zero so that the $\hat{B}_n$ are not tight in $\R$. It will turn out that we can find convenient processes $H^n$ preserving enough of the variation of the $\theta_k$ through integration and therefore turning $\hat{B}_n$ into a weighted sum of absolute values of a `sufficient' number of the $\theta_k$.
	
	\phantom{ee} \emph{Step 2: the integrators}.
	We define i.i.d.~random variables $W_k$, $k\ge 1$, which are distributed according to a Pareto distribution with parameters $\alpha>0$ and $x_{\text{min}}>0$. We recall that these random variables have cumulative distribution function $F_{\alpha,x_{\text{min}}}(x) = 1 - ( x_{\text{min}}/x)^{\alpha}$ if $x \ge x_{\text{min}}$, and $F_{\alpha,x_{\text{min}}}(x)=0$ otherwise.
	In order to center the random variables $W_k$, we symmetrise them by multiplying with i.i.d. Rademacher random variables, i.e. we set 
	$\theta_k :=  \xi_k  W_k$, $k \ge 0$, where $\Pro(\xi_1=1)=\Pro(\xi_1=-1)=1/2$ and the $\xi_k$ are i.i.d. and independent of the $W_k$. The $\theta_k$ are symmetric around zero and in the normal domain of attraction of an $\alpha$-stable law $\textbf{S}_\alpha(1, 0, 0)$. To see this, one may check the assumptions of \cite[Thm.~4.5.2]{whitt}.

	\phantom{ee} \emph{Step 3: the integrands}.
	Let $\varepsilon>0$ such that $\alpha^{-1}+ \varepsilon < 1$. For $n \ge 1$, define a stochastic càglàd pure jump process $H^n$ inductively by $H_t^n\equiv 0$ for $t \in [0, 2/n]$
	and 
	\begin{align*} H^n_t \; := \; \begin{cases}
			0, \qquad &\text{if } \operatorname{sgn}(\theta_k)\neq \operatorname{sgn}(\theta_{k-1}) \text{ and } H^n_{k/n}=- \frac{1}{n^\varepsilon} \, \operatorname{sgn}(\theta_{k-1}) \\
			- \frac{1}{n^\varepsilon} \, \operatorname{sgn}(\theta_{k}), \qquad &\text{if } \operatorname{sgn}(\theta_k)\neq \operatorname{sgn}(\theta_{k-1}) \text{ and } H^n_{k/n}=0 \\
			H^n_{k/n}, \qquad &\text{if } \operatorname{sgn}(\theta_k)= \operatorname{sgn}(\theta_{k-1})
	\end{cases} \end{align*}
	for $t \in (k/n,(k+1)/n]$, $k \ge 2$, where we note that $H^n_{k/n} \in \{0, -n^{-\varepsilon} \operatorname{sgn}(\theta_{k-1})\}$ and they are adapted to the natural filtration generated by the $\theta_k$. Furthermore, since $H^n_t \in \{0, \pm n^{-\varepsilon}\}$ for all $n\ge 1$ and $t\ge 0$, the sequence $H^n$ tends to zero almost surely uniformly on $[0,\infty)$ as $n\to \infty$. The choice of the $H^n$ is precisely such that 
	\begin{align} \theta_k \, (H^n_{k/n} - H^n_{(k+1)/n})\; = \; \begin{cases}
			\frac{1}{n^\varepsilon} \, |\theta_k|, \qquad &\text{if } \operatorname{sgn}(\theta_k)\neq \operatorname{sgn}(\theta_{k-1}) \\
			0, \qquad &\text{if } \operatorname{sgn}(\theta_k) = \operatorname{sgn}(\theta_{k-1}) 
		\end{cases} \label{eq:countex_summands}
	\end{align}
	as shows a simple calculation and it will turn out that this is sufficient for $\hat{B}_n$ in (\ref{eq:5.10}) to explode as $n\to \infty$. In other words, the $H^n$ are constructed in such a way that the stochastic integrals $\int_0^\bullet H^n_{r-} \diff B^n_{r}$ `pick up' enough of the total variation of the process $t\mapsto \theta_{\lfloor n t\rfloor}$ inducted by jumps of opposite direction (relative to the respective previous jump).
	
	\phantom{ee} \emph{Step 4: deriving a contradiction}.
	From the tightness of $(\hat{B}_n)_{n\ge 1}$, to every $\delta< \frac{1}{2}$, we obtain the existence of a $K_\delta>0$ such that \\[-3ex]
	\begin{align} \Pro \Big( \, \frac{c}{n^{\frac 1 \alpha+\varepsilon}}\;  \sum_{k=2}^{\lfloor n s\rfloor  -1} \, |\theta_k| \, \ind_{\{\operatorname{sgn}(\theta_k)\, \neq \, \operatorname{sgn}(\theta_{k-1})\}} \; > \; K_\delta \Big) \; \le \; \delta \label{eq:5.12}
	\end{align}
	for all $n\ge 1$, due to \eqref{eq:countex_summands}. Further, for $n\ge 1$, let 
	$$ \Lambda_n \; \; \; := \;\; \; \union{\substack{k_1,...,k_{m} \, \in \, \{2,...\lfloor n s\rfloor -1\}\\ m \, \ge \, \lfloor \frac{\lfloor n s\rfloor-1}{2}\rfloor }}{} \; \; \intersection{j=1}{m} \; \big\{ \operatorname{sgn}(\theta_{k_j}) \; \neq \; \operatorname{sgn}(\theta_{k_j-1}) \big\}$$
	and it is not hard to see that $\Pro(\Lambda_n)\ge \frac12$ for $n\ge 3/s$. Next, define random times $\tau_0 \equiv 0$ and $\tau_j  := \operatorname{inf}\, \{ \, k>\tau_{j-1} \,  :\,   \operatorname{sgn}(\theta_{k}) \neq \operatorname{sgn}(\theta_{k-1})\}$, $j\ge 1$. On $\Lambda_n$, we have $\tau_j \in \{ 2,...,\lfloor n s\rfloor-1\}$ for all $j=1,..., \lfloor \frac{\lfloor n s\rfloor-1}{2} \rfloor$. Consequently, the left-hand side of (\ref{eq:5.12}) is greater than or equal to
	\begin{align}
	\Pro \biggl(  \Bigl\{ \sum_{k=1}^{\lfloor \frac{\lfloor n s\rfloor-1}{2} \rfloor } \, |\theta_{\tau_k}|  \; > \; \frac{n^{\frac 1 \alpha+\varepsilon}\, K_\delta}{c}   \Bigr\}  \, \cap \,  \Lambda_n \biggr) \; &\ge \; \Pro  \biggl(   \intersection{k=1}{\lfloor \frac{\lfloor n s\rfloor-1}{2} \rfloor} \, \Bigl\{ \, |\theta_{\tau_k}|  \; > \; \frac{n^{\frac 1 \alpha+\varepsilon}\, K_\delta}{\lfloor \frac{\lfloor n s\rfloor-1}{2} \rfloor \, c}  \, \Bigr\} \; \cap \; \Lambda_n  \biggr) \notag \\
		&\ge \; \Pro \biggl(  \intersection{k=2}{\lfloor n s\rfloor-1} \, \Bigl\{ \, |\theta_{k}|  \; > \; \frac{n^{\frac 1 \alpha+\varepsilon}\, K_\delta}{\frac{n s-3}{2} \, c}  \, \Bigr\} \; \cap \; \Lambda_n  \biggr)  . \label{eq:5.13}
	\end{align}
	Since $\varepsilon$ is such that $\alpha^{-1}+\varepsilon<1$, it follows $n^{\frac1\alpha +\varepsilon}(ns-3)^{-1} \to 0$ as $n\to \infty$. Choose $N\ge 1$ large enough that $C_n := \; 2n^{\frac 1 \alpha+\varepsilon} K_\delta (ns-3)^{-1}c^{-1} < 1$ for all $n\ge N$. Then, for every $k \ge 1$, we have $\{|\theta_k| > C_n\} \supseteq \{ W_k > 1\}$ and the latter set is an almost sure set due to the distributional assumption on the $W_k$. Thus, the set $\bigcap_{k=2}^{n-1} \{ |\theta_{k}| > C_n\}$ has full probability for all $n\ge N$. Therefore, (\ref{eq:5.13}) yields \\[-2ex]
	$$ \frac{1}{2} \; > \; \Pro \Biggl( \, \intersection{k=2}{n-1} \, \biggl\{ \, |\theta_{k}|  \; > \; \frac{n^{\frac 1 \alpha+\varepsilon}\, K_\delta}{\frac{ns-3}{2} \, c}  \, \biggr\} \; \cap \; \Lambda_n \, \Biggr) \; = \; \Pro \left( \, \Lambda_n \, \right) \; \ge \; \frac{1}{2}$$
	for all $n\ge N$, yielding the desired contradiction.
\end{proof}


\appendix

\section{Appendix: Skorokhod's $J_1$ and $M_1$ topologies}\label{app:J1_M1_tops}

This brief appendix covers the fundamentals of the $J_1$ and $M_1$ topologies, introduced in \cite{skorokhod}. 


\begin{defi}[$J_1$ metric]\label{def:$J_1$_metric} For any $x,y \in \D_{\R^k}[0,T]$ we define
	$$\dJ(x,y) \; := \; \underset{\lambda \in \Lambda}{\operatorname{inf}} \; \left\{\,  |\lambda - \operatorname{Id} |_T^* \; \vee \; \; |x \circ \lambda - y|^*_T \,  \right\}$$
	on $\D_{\mathbbm{R}^k}[0,T]$, where $|z|_T^*:= \operatorname{sup}_{0\, \le \, t \, \le \, T} \, |z(t)|$ denotes the respective supremum norm and $\Lambda$ is the set of increasing homeomorphisms $\lambda: [0,T] \to [0,T]$ (i.e., `time-changes').
\end{defi}


For a càdlàg function $x \in \D_{\mathbbm{R}^k}[0,T]$ we define the completed graph as the set
$$ \Gamma_x \; := \; \left\{ (z,t) \in \R^k \times [0,T] \; : \; z = \alpha x(t-) + (1-\alpha) x(t), \; \alpha \in [0,1]\right\}$$
where $x(t-):= \lim_{s \uparrow t} \, x(s)$ and $x(0-):=x(0)$.
The completed graph is a connected subset of the plane $\R^{k+1}$ containing the line segments joining $(x(t),t)$ and $(x(t-),t)$ for all discontinuity points $t$ of $x$.  We can introduce a total order $\preceq$ given by
\begin{align*}\label{def:$M_1$_metric}
	(z_1, t_1)  \preceq (z_2, t_2) \;\, :\Leftrightarrow \;\, t_1 < t_2  \;\; \text{or} \; \,  \left[  t_1=t_2 \;\, \text{and} \; |x^{(i)}(t_1-)-z_1^{(i)}|  \le  |x^{(i)}(t_1-)-z_2^{(i)}|\; \forall i\right]
\end{align*}
for every $(z_j,t_j) \in \Gamma_x$, $j=1,2$, and where $x=(x^{(1)},...,x^{(k)})$. Based on this ordering, we define a \textit{parametric representation} as a continuous, non-decreasing and surjective map $(u,r):[0,1] \to \Gamma_x$, and we denote by $\Pi(x)$ the set of all parametric representations of $\Gamma_x$. 

\begin{defi}[$M_1$ Metric]\label{def:$M_1$_metric} For any $x_1, x_2 \in \D_{\R^k}[0,T]$ we set
	$$\dM(x_1,x_2) \; := \; \underset{(u_j,r_j) \; \in \; \Pi(x_j),\, j=1,2}{\operatorname{inf}} \; \left( \;  |u_1-u_2|_1^* \, \vee \, |r_1-r_2|^*_1 \; \right). $$
\end{defi}
The extension of both metrics to a metric on $\D_{\R^k}[0,\infty)$ is achieved as follows: convergence on $\D_{\R^k}[0,\infty)$ holds if and only if there is convergence in the metric on $\D_{\R^k}[0,T]$ for almost every $T>0$. We thus set
\begin{equation} \rho(x_1,x_2) \; := \; \int_0^\infty \; \e^{-T}\, \left[\rho^{[0,T]}\left(x_1|_{[0,T]}\, ,\, x_2|_{[0,T]}\right) \wedge 1\right] \; \diff T \label{eq:$M_1$_metric_real_line}
\end{equation}
with $\rho \in \{\dJ, \dM\}$, where $ \rho^{[0,T]}$ denotes the respective metric on $[0,T]$, and $x_1|_{[0,T]}, \,x_2|_{[0,T]}$ denote the restrictions of $x_1,x_2$ on $[0,T]$.  Both metrics induce a \emph{separable topology} strictly weaker than the uniform topology and it holds
$$ \dM(x_1, x_2) \; \le \; \dJ(x_1, x_2) \; \le \; |x_1 - x_2|^*_T$$
for all $x_1, x_2 \in \D_{\R^k}[0,T]$. We shall note that if $x_n \to x$ in $(\D_{\R^k}[0,\infty), \rho)$, $\rho \in \{\dJ, \dM\}$, and the limit $x$ is continuous, then the convergence holds in the uniform topology.
Although the Skorokhod space equipped with the above $J_1$ or $M_1$ metrics is not complete, there exist complete metrics generating these very topologies. Unfortunately, the $J_1$ and $M_1$ spaces are \emph{not} topological vector spaces since in general they fail to make addition a continuous operation. However, if the convergence is in a stricter sense, addition preserves convergence.
\begin{prop} \label{prop:A3}
	If $x_n \to x$ in $(\D_{\R^{k \times \ell }}[0,T], \rho)$, $\rho \in \{ \dJ, \dM\}$, linear combinations of the components preserve convergence, i.e., $x_n \eta \to x \eta$ in $(\D_{\R^k}[0,T], \rho)$ for all $\eta \in \R^\ell $.
\end{prop}
Given two convergent sequences of c\`adl\`ag paths $x_n$ and $y_n$ in dimensions $k$ and $\ell$, respectively, there is generally no guarantee that these sequences will converge together in $(\D_{\R^{k+\ell}}[0,T], \rho)$. A sufficient condition for them to converge together is the following.
\begin{prop} 
	If $x_n \to x$ in $(\D_{\R^{k}}[0,T], \rho)$ and $y_n \to y$ in $(\D_{\R^{\ell}}[0,T], \rho)$, where $\rho \in \{ \dJ, \dM\}$, and $\operatorname{Disc}(x) \cap \operatorname{Disc}(y) = \emptyset$, then $(x_n, y_n) \to (x,y)$ in $(\D_{\R^{k+\ell}}[0,T], \rho)$
\end{prop}
Within the framework of the $M_1$ topology, preservation of convergence through addition can even be attained on the basis of the following weaker assumption given in \cite{whitt}.
\begin{prop}
	If $x_n \to x$ in $(\D_{\R^{k}}[0,T], \dM)$ and $y_n \to y$ in $(\D_{\R^{k}}[0,T], \dM)$ such that that $(x^{(i)}(t) - x^{(i)}(t-)) (y^{(i)}(t) - y^{(i)}(t-)) \ge 0$ for all $0< t \le T$ and $i=1,...,k$, then $x_n+y_n \to x+y$ in $(\D_{\R^{k}}[0,T], \dM)$.
\end{prop}
Marginal projections and running suprema preserve convergence, as stated next.
\begin{prop}
	Let $x_n \to x$ in $(\D_{\R^{k}}[0,T], \rho)$, where $\rho \in \{ \dJ, \dM\}$, and $t \notin \operatorname{Disc}(x)$, then $x_n(t) \to x(t)$ and $|x_n|_t^* \to |x|_t^*$.
\end{prop}
Regarding tightness of stochastic processes on Skorokhod space, via Prokhorov’s theorem, the central ingredient is a characterisation of relatively compact subsets. Akin to the Arzelà–Ascoli theorem for continuous paths, we now recall the key characterisations from \cite{skorokhod}.

\begin{defi}[Moduli of continuity] \label{def:A7}
	For a $x: [0,T] \to \R^k$ and $\theta>0$ we define 
	\begin{align*}
		w'(\alpha, \theta) \; &:= \; \inf{} \biggl\{ \sup{i\le r} \sup{s,t \, \in \, [t_{i-1},t_{i})} |\alpha(s)-\alpha(t)|\; : \; t_0<...<t_r \text{ and } \inf{1\le i \le r}(t_i-t_{i-1})\ge \theta \biggr\}\end{align*}
	with $t_0=0$ and $t_r=T$, where the infimum in $w'$ runs over all partitions of $[0,T]$ with mesh size no smaller than $\theta$, and where we have defined
	\begin{align*}
		&w''(\alpha, \theta) :=  \tilde{w}(\alpha, \theta) \vee  \bar{v}(\alpha,0,\theta) \vee  \bar{v}(\alpha,T,\theta), \\
		&\tilde{w}(\alpha, \theta)  := \sup{0\le t \le T} \sup{} \Bigl\{  \norm{\alpha(t_2) - [ \alpha(t_1), \alpha(t_3)]} \, : \, 0\vee (t-\theta) \le t_1 < t_2 < t_3 \le (t+\theta) \wedge T \Bigr\},\\
		&\bar{v}(\alpha,r,\theta)  :=  \sup{(r-\theta) \vee 0 \le s,t \le (r+\theta) \wedge T} \, |\alpha(t)- \alpha(s)|,
	\end{align*}
	with $[ \alpha(t_1), \alpha(t_3)] := \{ \lambda \alpha(t_1)+(1-\lambda) \alpha(t_3) \, : \, \lambda \in [0,1]\}$ denoting the line segment between $\alpha(t_1)$ and $\alpha(t_3)$ and $\norm{x - A}:= \operatorname{inf}_{y \in A} |x-y|$.
\end{defi}

\begin{theorem} \label{thm:A.8}
	A subset $A \subseteq \D_{\R^k} [0,T]$ is relatively compact for the $J_1$ (respectively $M_1$) topology if and only if 
	\begin{enumerate}[(i)]
		\item $A$ is uniformly bounded, i.e., $ \sup{}_{\!\alpha \in A} \, |\alpha|^*_T \; < \; \infty$, and\vspace{-8pt}
  \item $A$ satisfies  
		$\lim_{\theta \downarrow 0} \; \sup{}_{\!\alpha \in A} \; w (\alpha, \theta) \; = \; 0$, where $w=w'$ (respectively $w=w''$ for $M_1$).
	\end{enumerate}
\end{theorem}
Recall $N_\delta^T$ from \eqref{eq:maxnumosc} which counts the number of increments above size $\delta>0$.
\begin{cor} \label{cor:A9}
	Let $A \subseteq \D_{\R^k} [0,T]$ be relatively compact for the $J_1$ or $M_1$ topology. Then, for any $\delta>0$, it holds
	$ \sup{}_{\!\alpha \in A}\; N^T_\delta(\alpha) \; < \; \infty.$
\end{cor}
\begin{proof}
    The $J_1$ statement follows easily from Theorem \ref{thm:A.8}. For $M_1$, we sketch a contradiction proof in the case $k=1$. If the number were unbounded, there exist $\alpha_n\in A$ with increasing numbers of $\delta$-increments. Uniform boundedness on $A$ then implies that the increments cannot all occur along a single monotone excursion, so we obtain an unbounded number of consecutive increments that \emph{backtrack} by at least $\delta/2$. On $[0,T]$ these must concentrate on increasingly small intervals, forcing $\operatorname{sup}_{\alpha \in A} \tilde w(\alpha, \theta) \ge \delta/2$, in contradiction with (ii) of Theorem~\ref{thm:A.8}.
\end{proof}

\section{Appendix: The P-UT and UCV conditions}\label{app:UCV_and_PUT}

The P-UT condition was introduced in \cite{stricker}, and is the starting point of \cite{jakubowskimeminpages} as well as \cite{jakubowski}.

\begin{defi}[Predictable uniform tightness]\label{def:PUT}
	Let $(X^n)_{n\ge 1}$ be a sequence of $\R^d$-valued semimartingales on filtered probability spaces $(\Omega^n, \mathcal{F}^n, \mathbbm{F}^n, \Pro^n)$. The $X^n$ are said to be \emph{predictably uniformly tight} (P-UT) on $[0,T]$ if the family of random variables 
	$$ \mathcal{A}_t \; := \; \union{n\in\N}{} \left\{ \int_0^t  H_s \, \diff X^n_s \; : \; H \in \textbf{S}(X^n), \, |H|^*_t \le 1 \right\}$$
	is tight in $\R$ for all $t\in [0,T]$, where $\textbf{S}(X^n)$ denotes the space of all real-valued simple predictable processes on the filtered probability space $(\Omega^n, \mathcal{F}^n, \mathbbm{F}^n, \Pro^n)$ where $X^n$ is defined.
\end{defi}

The authors of \cite{jakubowskimeminpages} themselves pointed out that the abstract P-UT condition \emph{`does not seem to be readily verifiable'} (translated from French). Yet, we stress that there are several insightful equivalent characterisations: see \cite[Thms.~IV.6.15, IV.6.16, \& IV.6.21]{shiryaev}, \cite[Lem.~3.1]{jakubowskimeminpages}, and \cite[Thm.~1.4]{meminslominski}. Unlike \eqref{eq:Mn_An_condition}, these are all stated in terms of the semimartingale characteristics or the canonical decompositions obtained after suitably truncating the large jumps.

In a later work, \cite{kurtzprotter} introduced the slightly weaker notion of \emph{uniformly controlled variations} (UCV). First, one must suitably truncate the large jumps through the mapping
$$ (J_\delta(x))(t) \; := \; \sum_{0 < s \le t} \; \Bigl( \Bigl(1- \frac{\delta}{|\Delta x(s)|}\Bigr) \, \vee \, 0 \Bigr) \, \Delta x(s), $$
for $x \in \D_{\R^d}[0,\infty)$. Then, for a given sequence $(X^n)_{n\geq 1}$, one defines $X^{n,\delta}:= X^n - J_\delta(X^n)$. Note that when $X^n$ is a semimartingale, so is $X^{n,\delta}$, since $J_\delta(X^n)$ is of finite variation.
\begin{defi}[Uniformly controlled variations]\label{defn:UCV} 
	Let $(X^n)_{n\ge 1}$ be a sequence of $\R^d$-valued semimartingales on filtered probability spaces $(\Omega^n, \mathcal{F}^n, \mathbbm{F}^n, \Pro^n)$. Fixing $\delta \in (0,\infty]$, the $X^n$ are said to have \emph{uniformly controlled variations} (UCV), for this $\delta$, if there are decompositions $X^{n, \delta}= M^{n, \delta} + A^{n,\delta}$, where the $M^{n, \delta}$ are $\mathbbm{F}^n$-local martingales and the $A^{n,\delta}$ are $\mathbbm{F}^n$-adapted finite variation processes so that: (i) for each $\alpha>0$ there are $\mathbbm{F}^n$-stopping times $\tau_\alpha^{n}$ with $\Pro^n \left( \tau_\alpha^{n} \le \alpha\right) \le \alpha^{-1}$ for all $n\ge 1$, and (ii) for all $t>0$, we have
	\begin{align}
		\sup{n\ge 1} \; \E^n \left[ \, [ M^{n,\delta} ]_{t\wedge \tau_\alpha^{n}} \; + \; \operatorname{TV}_{[0\, , \, t\wedge \tau_\alpha^{n}]}(A^{n,\delta}) \, \right] \; < \; \infty.  \label{eq:UCV} 
	\end{align}
\end{defi}

It was shown in \cite{kurtzprotter} that the two definitions are equivalent under $J_1$ tightness. The P-UT condition implies the UCV condition. Conversely, if $(X^n)_{n\ge 1}$ satisfies the UCV condition and $\operatorname{TV}_{[0,t]}(J_\delta^n(X^n))$ is tight, for all $t> 0$, then also the P-UT condition is satisfied.

\section{Appendix: Good decompositions}

We give a characterisation of \eqref{eq:Mn_An_condition} that makes a precise link to the P-UT and UCV conditions in \cite{jakubowskimeminpages, kurtzprotter}. The necessary part (or more precisely Corollary \ref{rem:GD_necessity}) resembles \cite[Thm.~4]{KP_Durham} on the necessity of P-UT and UCV for the notion of `good integrators' in \cite{KP_Durham}.

\begin{prop}[`Continuity' characterisation of good decompositions] \label{prop:necessity_GD}
	Let $(X^n)_{n\ge 1}$ be a semimartingale sequence and consider all càdlàg adapted integrands $H^n$ on the same filtered probability spaces as the $X^{n}$. Then, the sequence $(X^n)_{n\ge 1}$ enjoys the property that
	\begin{equation}\label{eq:GD_indispensable}\int_0^\bullet \, H^{n}_{s-} \, \diff X^{n}_s \; \Rightarrow \; 0 \quad \text{ on } \quad (\D_{\R^{d}}[0,\infty), \, \dM),
	\end{equation}
whenever $H^{n} \to 0$ ucp, if and only if it has good decompositions \eqref{eq:Mn_An_condition} and $|X^n|_t^*$ is tight, for all $t> 0$. Furthermore, when this is the case, one can take $M^n$ and $A^n$ in \eqref{eq:Mn_An_condition} to be such that $|\Delta M^n|\le 2$, $A^n = J^n + \tilde{A}^n$ with $\tilde{A}^n$ predictable and $J^n = \sum_{s\leq \bullet} \Delta X^n_s \, \ind_{\{|\Delta X^n_s|\geq 1\}}$.
\end{prop}
\begin{proof} It suffices to consider $d=1$. We begin with the `if' part and let $X^n=M^n+A^n$ satisfy \eqref{eq:Mn_An_condition}. Since $A^n$ is of tight total variation, we have \eqref{eq:GD_indispensable} with $A^n$ as the integrators. Moreover, as we assume $|X^n|_t^*$ is tight, for all $t>0$, the tight total variation of $A^n$ gives that $|M^n|_t^*$ is tight for all $t>0$. Consequently, using the control on the jumps of $M^n$ in \eqref{eq:Mn_An_condition}, it follows from Lemma \ref{prop:Lenglart}(b) that also $[M^n]_t$ is tight, for all $t>0$, and Lemma \ref{prop:Lenglart}(c) then gives that \eqref{eq:GD_indispensable} holds with $M^n$ as the integrators. For the `only if' part, notice first that, for any $t>0$, we can immediately deduce the tightness of $|X^n|^*_t$ from \eqref{eq:GD_indispensable}. Similarly, we can then see that, for any subsequence $(X^{n_k})_{k\geq1}$ and $t>0$, $R_k \rightarrow \infty$ implies
	\[
	\lim_{k\rightarrow \infty} \; \mathbb{P}^{n_k}\Bigl( \Bigl| \int_0^\bullet \, X^{n_k}_{s-} \,\diff X^{n_k}_s \Bigr|_t^* > R_k  \Bigr)=0,
	\]
	since $|X^{n}|_t^*$ is tight and hence $X^{n_k}/R_k \rightarrow 0$ ucp, so \eqref{eq:GD_indispensable} applies. Thus, $\int_0^t X^n_{s-} \diff X^n_s$ is tight, for all $t\geq0$. Writing $[X^n]_t=(X^n_t)^2-(X_0^n)^2 - 2\int_0^t X^n_{s-} \diff X^n_s$, we conclude that $[X^n]_t$ is tight. Now let $X^n-J^n=M^n+\tilde{A}^n$ be the canonical decomposition of the special semimartingale $X^n-J^n$. Then $X^n=M^n +A^n$ with $A^n=J^n + \tilde{A}^n$, and recall that it satisfies $|\Delta M^n_t|,|\Delta \tilde{A}^n_t| \leq 2$, for $t\geq 0$, with $\tilde{A}^n$ predictable. Observe that $J^n$ is of tight total variation, since
	$$ \operatorname{TV}_{[0,\bullet]}(J^n) \; =\;  \sum_{s \leq \bullet} |\Delta X^n_s| \, \ind_{\{|\Delta X^n_s|\geq 1\}} \; \leq \; [X^n],$$
	where $[X^n]_t$ is tight for all $t>0$. Therefore it remains to confirm that $\tilde{A}^n$ must be of tight total variation. To this end, set $\sigma^n_c := \operatorname{inf} \{ t>0 : |M^n|^*_t \geq c \,\; \text{or}\, \operatorname{TV}_{[0,t]}(\tilde{A}^{n}) \geq c \}$ and let $\tau^n$ be an arbitrary bounded $\mathbb{F}^n$-stopping time. By boundedness of the jumps and the predictability of $\tilde{A}^n$, we have $\mathbb{E}^n[\, [M^n,\tilde{A}^n]_{\sigma^n_c \land \tau^n}]=0$, so Fatou's lemma gives
	\[
	\mathbb{E}^n[\, [M^n]_{\tau^n}] \leq \liminf_{c \rightarrow \infty} \, \mathbb{E}^n[\, [M^n]_{\sigma^n_c \land \tau^n}] \leq \liminf_{c \rightarrow \infty} \, \mathbb{E}^n[\, [M^n+\tilde{A}^n]_{\sigma^n_c \land \tau^n}] \leq \mathbb{E}^n[\, [X^n]_{\tau^n}].
	\]
	Thus, the uniform bound on the jumps of $M^n$ and Lenglart's inequality \cite[Lem.~I.3.30]{shiryaev} gives tightness of $[M^n]_t$, for all $t\geq 0$. Since $\tilde{A}^n$ is predictable, for any $R_n\rightarrow \infty $ and $t>0$,  we can find c\`adl\`ag processes $H^{n}$ with $|H^{n}|_t^*\le 1$ such that
	\[\Pro^{n}\Bigl( \Bigl| \operatorname{TV}_{[0,t]}(\tilde{A}^{n}) - \int_0^t H^{n}_{s-} \diff \tilde{A}^{n}_s\Bigr| \le 1\Bigr)> 1-1/R_n.
	\]
	As in the `if' part of the proof, we deduce from tightness of $\operatorname{TV}_{[0,t]}(J^n)$ and $[M^n]_t$, for all $t>0$, that $R_n^{-1}\int_0^\bullet H^n_{s-} \diff (M^n +J^{n})_s \rightarrow 0$. Now suppose, for a contradiction, that $\tilde{A}^n$ fails to be of tight total variation on some compact interval $[0,t]$. Then there is a $\delta >0$ so that, along a suitable subsequence, we have $\Pro^{n_k}(\operatorname{TV}_{[0,t]}(A^{n_k})> \tilde{R}_{k} +1 ) \geq \delta$ and hence
	\[
	\Pro^{n_k}\Bigl(\int_0^t H^{n_k}_{s-} \diff \tilde{A}^{n_k}_s > \tilde{R}_{k} \Bigr) \geq  \delta - 1/\tilde{R}_k,
	\]
	with $\tilde{R}_k \rightarrow \infty$. But this would contradict \eqref{eq:GD_indispensable}, in view of what we just deduced, so $\tilde{A}^n$ must be of tight total variation, and hence we have good decompositions of the desired form.
\end{proof}

A straightforward argument shows that \eqref{eq:GD_indispensable} implies the P-UT condition. Conversely, P-UT implies \eqref{eq:GD_indispensable} via the characterisation in \cite[Thm.~VI.6.15]{shiryaev}. Moreover, \cite[p.1068-69]{kurtzprotter} showed that P-UT is equivalent to UCV when there is tightness of both $\{|X^n|_t^*\}_{n\ge 1}$ and $\{N^t_\delta(X^n)\}_{n\ge 1}$ for some $\delta>0$ and all $t> 0$. $M_1$ tightness implies the latter by Corollary \ref{cor:A9}. Hence, \eqref{eq:Mn_An_condition} is equivalent to the P-UT and  UCV conditions when the sequence $(X^n)_{n\ge 1}$ is tight in $M_1$.

\begin{cor}[`Necessity' of \eqref{eq:Mn_An_condition}]\label{rem:GD_necessity} Let $(X^n)_{n\geq 1}$ be semimartingales such that $X^n \Rightarrow X$ on $(\D_{\R^d}[0,\infty), \dM)$. Suppose $X$ is also semimartingale and \eqref{eq:concerted_stoch_int_conv} holds whenever \eqref{eq:joint_conv} and \eqref{eq:oscillcond} are satisfied for the given $(X^n)_{n\geq 1}$. Then $(X^n)_{n\geq 1}$ has good decompositions \eqref{eq:Mn_An_condition}.   
\end{cor}
\begin{proof} With $H^{n} \to 0$, $(H^n,X^n)$ satisfies \eqref{eq:oscillcond} by (ii) of Proposition \ref{prop:3.3}, and we also get \eqref{eq:joint_conv}, i.e., weak convergence of $(H^n,X^n)$ to $(0,X)$ on the $M_1$ product space. Thus, \eqref{eq:concerted_stoch_int_conv} applies and gives that $(X^n)_{n\geq 1}$ satisfies \eqref{eq:GD_indispensable}, so the claim follows from Proposition \ref{prop:necessity_GD}.
\end{proof}

Proposition \eqref{prop:necessity_GD} gives a simple, elegant way of seeing the following.

\begin{lem}[Stochastic calculus with \eqref{eq:Mn_An_condition}]\label{eq:integrals_Ito_GD} 
	Let $(X^n)_{n\ge 1}$ have \eqref{eq:Mn_An_condition} and let $|X^n|^*_t$ be tight, for all $t>0$. Consider
	\[
	\diff Z^n_s=G^n_{s-} \,\diff X^n_s\quad \text{or} \quad \diff Z^n_s=\diff f_n(s,X^n_s),
	\]
	for $f_n\in C^{1,2}([0,\infty)\times\mathbb{R}^d;\mathbb{R})$, where also the $|G^n|_t^*$ are tight, for all $t>0$, and $f_n$ and its derivatives are uniformly bounded in $n$ on compacts. Then $(Z^n)_{n\geq 1}$ also has \eqref{eq:Mn_An_condition}.
    \end{lem}
    \begin{proof} It suffices to confirm \eqref{eq:GD_indispensable} for the $Z^n$. The first case is immediate from \eqref{eq:GD_indispensable} for the $X^n$. In the second case, we apply Itô's formula and verify \eqref{eq:GD_indispensable} individually for each term, again utilising \eqref{eq:GD_indispensable} for the $X^n$. In this respect, note that, if $X^n=M^n+A^n$ are good decompositions as given in Proposition \ref{prop:necessity_GD}, then $[M^n]_t$ is tight, for every $t>0$ (this follows from Lemma \ref{prop:Lenglart}(b), since $|\Delta M^n|\le 1$ uniformly and $|M^n|^*_t$ is tight, due to the tight total variation of $A^n$ and tightness of $|X^n|_t^*$). Hence, one readily deduces the tightness of $[X^n]_t$, for every $t>0$. Furthermore, on behalf of Taylor's theorem,
	$$ \sum_{0\le s \le t} \, \Big| \Delta f_n(s,X^n_s) \, - \, \sum_{i=1}^d \frac{\partial}{\partial x_i} f_n(s- ,X^{n,(i)}_{s-}) \, \Delta X^{n,(i)}_{s}\Big| \; \lesssim \; \sum_{0\le s \le t} \, \sum_{i=1}^d |\Delta X^{n,(i)}_{s}|^2 $$
	on $\{|X^n|^*_t \le K\}$, where `$\lesssim$' denotes inequality up to a multiplicative constant depending only on $t$ and $K$. Thus, the corresponding terms appearing from Itô's formula can be controlled by means of the tightness of the quadratic variations $[X^n]_t$.
    \end{proof}

	\vspace{2pt}

\noindent\textbf{Acknowledgments.} We thank Alda\"{i}r Petronilia for valuable contributions to Section \ref{sec:6.5}, an early version of which was part of an unpublished note on the $M_1$ topology. We also thank the associate editor and an anonymous referee for their helpful comments which have significantly sharpened the presentation. Finally, we are grateful to Ben Hambly for useful comments on the manuscript, and we would like to thank Philip Protter and Richard Davis for fruitful discussions and suggestions during a visit of AS to the Department of Statistics at Columbia University. The research of FW was funded by the EPSRC grant EP/S023925/1.

	\vspace{6pt}

\noindent\textbf{Data availability statement.}
No data were created or analysed in this work. Data sharing is not applicable to this article.




 \printbibliography

\end{document}